\documentclass[a4paper]{scrartcl}
\usepackage{fullpage,authblk,amsmath,amssymb,amsthm,amsfonts}
\usepackage{bm, bbm, dsfont}
\usepackage{enumerate,enumitem}
\usepackage{url,mathtools}
\usepackage{mathrsfs}
\usepackage{xcolor}
\usepackage{thm-restate}

\usepackage[pdffitwindow=true,
		   pdfstartview={FitH},
            plainpages=false,
            pdfpagelabels=true,
            pdfpagemode=UseOutlines,
            pdfpagelayout=SinglePage,
            bookmarks=false,
            colorlinks=true,
            hyperfootnotes=false,
            linkcolor=blue,
            urlcolor=blue!30!black,
            citecolor=green!50!black]{hyperref}

\usepackage{cleveref}

\usepackage{comment} 

\setenumerate[1]{label=(\roman*)}

\newtheorem{theorem}{Theorem}[section]
\newtheorem{corollary}[theorem]{Corollary}
\newtheorem{lemma}[theorem]{Lemma}
\newtheorem{proposition}[theorem]{Proposition}

\newtheorem{definition}[theorem]{Definition}

\theoremstyle{definition}

\newtheorem{remark}[theorem]{Remark}
\newtheorem{example}[theorem]{Example}

\numberwithin{equation}{section}

\newcommand{\R}{\mathbb{R}}
\newcommand{\N}{\mathbb{N}}

\newcommand{\rd}{\mathrm{d}}

\newcommand{\diam}[1]{\textup{diam}({#1})}

\newcommand{\card}[1]{\textup{card}({#1})}
\DeclareMathOperator{\supp}{supp}

\DeclareMathOperator*{\essinf}{ess\,inf}
\DeclareMathOperator*{\esssup}{ess\,sup}

\newcommand{\abs}[1]{\vert #1\vert}
\newcommand{\Abs}[1]{\left\vert#1\right\vert}
\newcommand{\norm}[1]{\Vert #1 \Vert}
\newcommand{\Norm}[1]{\left\Vert #1 \right\Vert}

\newcommand{\Hel}[0]{\textup{H}}
\newcommand{\Lip}[0]{\textup{Lip}}
\newcommand{\KL}[0]{\textup{KL}}
\newcommand{\TV}[0]{\textup{TV}}
\newcommand{\Was}[0]{\textup{W}}
\newcommand{\opt}[0]{\textup{opt}}

\newcommand*{\todo}[1]{\bgroup\color{red}TODO: #1\egroup}

\begin{document}
\title{Upper and lower bounds for local Lipschitz stability of Bayesian posteriors}

\author{Nada Cvetkovi\'{c}, Han Cheng Lie\thanks{~han.lie@uni-potsdam.de}}
\affil{~Institut f\"ur Mathematik, Universit\"at Potsdam, Campus Golm, Potsdam OT Golm 14476, Germany}
\renewcommand\Affilfont{\small}

\date{}

\maketitle

\begin{abstract}
 The work of Sprungk (Inverse Problems, 2020) established the local Lipschitz continuity of the misfit-to-posterior and prior-to-posterior maps with respect to the Kullback--Leibler divergence and the total variation, Hellinger, and 1-Wasserstein metrics, by proving certain upper bounds.
 The upper bounds were also used to show that if a posterior measure is more concentrated, then it can be more sensitive to perturbations in the misfit or prior.
 We prove upper bounds and lower bounds that emphasise the importance of the evidence.
 The lower bounds show that the sensitivity of posteriors to perturbations in the misfit or the prior not only can increase, but in general will increase as the posterior measure becomes more concentrated, i.e. as the evidence decreases to zero.
 Using the explicit dependence of our bounds on the evidence, we identify sufficient conditions for the misfit-to-posterior and prior-to-posterior maps to be locally bi-Lipschitz continuous.
\end{abstract}


%

\section{Introduction}
\label{sec_introduction}

In the theory of inverse problems, one is given an observed data vector $y\in\mathcal{Y}$, an unknown parameter $\theta\in\Theta$, and a function $G:\Theta\to\mathcal{Y}$ such that $y=G(\theta)$, and the task is to infer $\theta$ given $y$. 
In the frequentist Bayesian approach to inverse problems, one assumes that the observation is a realisation of a $\mathcal{Y}$-valued random variable $Y=G(\theta^\dagger)+\varepsilon$, where $\theta^\dagger$ is fixed and $\varepsilon$ is observation noise, and one models the unknown $\theta^\dagger$ using a $\Theta$-valued random variable with prior law $\mu$. The solution to the Bayesian inverse problem is given by the posterior law $\mu_\Phi$ of $\theta$ given the observed quantity $y$, where $\mu_{\Phi}$ is absolutely continuous with respect to the prior $\mu$ and uniquely defined by the normalised likelihood $\ell_{\Phi,\mu}$:
\begin{equation}
\label{eq_generic_posterior}
 \ell_{\Phi,\mu}(\theta;y)\coloneqq \frac{\rd\mu_\Phi}{\rd\mu}(\theta)=\frac{\exp(-\Phi(\theta;y))}{Z_{\Phi,\mu}(y)} ,\quad Z_{\Phi,\mu}(y)\coloneqq \int \exp(-\Phi(\theta';y))\rd \mu(\theta').
\end{equation}
We shall refer to the negative logarithm $\Phi$ of the unnormalised likelihood and the normalisation constant $Z_{\Phi,\mu}$ as the `misfit' and the `evidence' respectively.
The misfit is determined by the `forward model' $G$ that maps unknowns to observed quantities and the distribution of the observation noise $\varepsilon$.
For example, if $\mathcal{Y}=\R^d$ for some $d\in\N$ and $\varepsilon$ is Gaussian with mean 0 and positive-definite covariance $C$, then $\Phi(\theta;y)=\tfrac{1}{2}\Norm{C^{-1/2}(G(\theta)-y)}^2$.

An important theme in the theory of Bayesian inverse problems is the stability of the posterior with respect to perturbations in one of the objects that define it.
Let $\mathfrak{d}_{\mathcal{P}(\Theta)}$ be some chosen metric or divergence on the space of probability measures on a fixed measurable space $(\Theta,\Sigma)$.
For a fixed prior $\mu$, the stability of the posterior $\mu_\Phi$ with respect to the misfit $\Phi$ can be described by the local Lipschitz continuity of the misfit-to-posterior mapping $L^p_\mu\ni \Phi\mapsto \mu_\Phi$, i.e. by an upper bound of the form
\begin{equation}
\label{eq_example_upper_bound}
 \mathfrak{d}_{\mathcal{P}(\Theta)}(\mu_{\Phi_1},\mu_{\Phi_2})\leq C(\mu,\Phi_1)\Norm{\Phi_1-\Phi_2}_{L^p_\mu}.
\end{equation}
Among the best-known instances of stability results in Bayesian inverse problems of the form \eqref{eq_example_upper_bound} are those from \cite{Stuart2010}.
For example, the Hellinger metric was used to measure the distance between a reference posterior measure and an approximate posterior measure associated to a perturbed misfit in \cite[Theorem 4.6]{Stuart2010}. 
In the case of Gaussian observation noise, this stability result can then be used to show stability of the posterior with respect to the data.
An important consequence of the stability results in \cite{Stuart2010} is that many Bayesian inverse problems of interest are well-posed in the sense of Hadamard, i.e. the posterior depends continuously on the data $y$.

An important reason for the importance of stability with respect to the misfit is that it implies stability with respect to the forward model, under certain assumptions on the observational noise.
Stability with respect to the forward model is important in the context of numerical methods for solving Bayesian inverse problems, and more generally when accounting for the consequences of using an incorrect forward model in inference tasks.
This is because such stability results ensure that more accurate approximations of the true forward model lead to more accurate approximations of the true posterior, when all other factors are held fixed. 
Stability with respect to the forward model establishes a link between numerical analysis and approximation theory on one hand, and robust statistical inference on the other hand, in the sense that one can use error bounds for a given numerical method or approximation in order to bound the error in the output of a statistical procedure based on that method or approximation.
Although the constants in these bounds are often difficult or costly to evaluate in practice, the bounds nevertheless provide a more precise characterisation of the continuous dependence of the inference outcome on numerical or modelling approximations.

An important milestone in the stability analysis of Bayesian inverse problems is given in the work \cite{Sprungk2020}, which presents stability results of the posterior with respect to misfit perturbations and prior perturbations. Perturbations in the posterior are measured in the Kullback--Leibler divergence and the Hellinger, total variation, and 1-Wasserstein metrics.
Stability with respect to the prior is important because the choice of prior determines the maximal support of the resulting posterior, and thus the set of statistically admissible parameter values.
Furthermore, it is well-known from the mathematical-statistical theory of nonparametric Bayesian inference that the prior plays a key role in posterior consistency and contraction; see e.g. \cite[Chapter 6]{GhosalvanderVaart2017}.

\subsection{Contents of this work}
\label{section_contents}

The goal of this work is to further investigate the bounds and ideas relating to local Lipschitz stability results for the misfit-to-posterior and prior-to-posterior maps that were presented in \cite{Sprungk2020}.
To this end, we prove upper and lower bounds on Kullback--Leibler divergences and Hellinger, total variation, and 1-Wasserstein metrics between posterior measures associated to different misfits and priors.

Recalling that $\mathfrak{d}_{\mathcal{P}(\Theta)}$ denotes a chosen metric or divergence on the space of probability measures on $\Theta$, the upper bounds on posteriors that we state below are of the form
\begin{subequations}
\label{eq_prototypical_upper_bounds}
\begin{align}
  \mathfrak{d}_{\mathcal{P}(\Theta)}(\mu_{\Phi_1},\mu_{\Phi_2})\leq &C_1 \mathcal{U}_M(\Phi_1-\Phi_2,Z_{\Phi_1,\mu},Z_{\Phi_2,\mu})
 \label{eq_prototypical_upper_bound_misfit_perturbation}
 \\
  \mathfrak{d}_{\mathcal{P}(\Theta)}((\mu_1)_{\Phi},(\mu_2)_{\Phi})\leq & C_2 \mathcal{U}_P(\mathfrak{d}_{\mathcal{P}(\Theta)}(\mu_1,\mu_2),Z_{\Phi,\mu_1},Z_{\Phi,\mu_2}),
  \label{eq_prototypical_upper_bound_prior_perturbation}
\end{align}
\end{subequations}
where $C_1=C_1(\mu,\Phi_1)$ and $C_2=C_2(\mu_1,\Phi)$ are positive scalars.
In \eqref{eq_prototypical_upper_bound_misfit_perturbation}, we consider posteriors resulting from perturbed misfits and the same prior $\mu$, while in \eqref{eq_prototypical_upper_bound_prior_perturbation}, we consider posteriors resulting from different priors $\mu_1$ and $\mu_2$ and the same misfit $\Phi$.
The nonnegative functions $\mathcal{U}_M$ and $\mathcal{U}_P$ depend on the choice of $\mathfrak{d}_{\mathcal{P}(\Theta)}$ and converge to zero as $\Phi_1-\Phi_2$ converges to zero in an appropriate norm and as $\mathfrak{d}_{\mathcal{P}(\Theta)}(\mu_1,\mu_2)$ converges to zero respectively.

If $\mathfrak{d}_{\mathcal{P}(\Theta)}$ satisfies the triangle inequality, then the upper bounds above can be combined to obtain an upper bound on $\mathfrak{d}_{\mathcal{P}(\Theta)}((\mu_1)_{\Phi_1},(\mu_2)_{\Phi_2})$, i.e. on the measure of discrepancy when both the prior and misfit are perturbed.

While some of our upper bounds appear implicitly in the proofs of results in \cite{Sprungk2020}, others do not.
We seek upper bounds with explicit dependence on the evidence, because the evidence plays a crucial role in determining the behaviour of the posterior, by Bayes' formula \eqref{eq_generic_posterior}.

The second main class of results that we state below are lower bounds on the above-mentioned metrics or divergences between posteriors, where perturbed misfits and perturbed priors are considered separately.
That is, we investigate inequalities of the form
\begin{subequations}
\label{eq_prototypical_lower_bounds}
\begin{align}
 \mathfrak{d}_{\mathcal{P}(\Theta)}(\mu_{\Phi_1},\mu_{\Phi_2})\geq & C_3\mathcal{L}_M(\Phi_1-\Phi_2,Z_{\Phi_1,\mu},Z_{\Phi_2,\mu}),
 \label{eq_prototypical_lower_bound_misfit_perturbation}
 \\
  \mathfrak{d}_{\mathcal{P}(\Theta)}((\mu_1)_{\Phi},(\mu_2)_{\Phi})\geq & C_4 \mathcal{L}_P(\mathfrak{d}_{\mathcal{P}(\Theta)}(\mu_1,\mu_2),Z_{\Phi,\mu_1},Z_{\Phi,\mu_2}),
  \label{eq_prototypical_lower_bound_prior_perturbation}
\end{align}
\end{subequations}
where $C_3=C_3(\mu,\Phi_1)$ and $C_4=C_4(\mu_1,\Phi)$ are positive.
The functions $\mathcal{L}_M$ and $\mathcal{L}_P$ are nonnegative functions whose values converge to zero as $\Phi_1-\Phi_2$ converges to zero in an appropriate norm and as $\mathfrak{d}_{\mathcal{P}(\Theta)}(\mu_1,\mu_2)$ converges to zero respectively.
Lower bounds of the type shown in \eqref{eq_prototypical_lower_bounds} do not appear in \cite{Sprungk2020}. 

The first motivation for considering lower bounds is to shed light on whether the upper bounds shown in \eqref{eq_prototypical_upper_bounds} are sharp up to local Lipschitz continuity prefactors, i.e. whether $\mathcal{L}_M=\mathcal{U}_M$ for misfit perturbations and $\mathcal{L}_P=\mathcal{U}_P$ for prior perturbations. 
For example, when $\mathfrak{d}_{\mathcal{P}(\Theta)}$ is the total variation metric and one considers misfit perturbations, then \Cref{proposition_TV_bounds_misfit_perturbations_via_Lipschitz_continuity} below implies that this is indeed the case, with
\begin{equation*}
\mathcal{U}_M(\Phi_1-\Phi_2,Z_{\Phi_1,\mu},Z_{\Phi_2,\mu})= \Norm{\Phi_1-\Phi_2+\log \tfrac{Z_{\Phi_1,\mu}}{Z_{\Phi_2,\mu}} }_{L^1_\mu}=\mathcal{L}_M(\Phi_1-\Phi_2,Z_{\Phi_1,\mu},Z_{\Phi_2,\mu}).
\end{equation*}
On the other hand, when $\mathfrak{d}_{\mathcal{P}(\Theta)}$ is the total variation metric and one considers prior perturbations, then \Cref{proposition_TV_bounds_prior_perturbation_via_triangle_inequality} below shows that 
\begin{align*}
\mathcal{U}_P(d_{\TV}(\mu_1,\mu_2),Z_{\Phi,\mu_1},Z_{\Phi,\mu_2})=& 
d_{\TV}(\mu_1,\mu_2)+\frac{\abs{Z_{\Phi,\mu_2}-Z_{\Phi,\mu_1}}}{2},
\\
\mathcal{L}_P(d_{\TV}(\mu_1,\mu_2),Z_{\Phi,\mu_1},Z_{\Phi,\mu_2})=& 
\Abs{d_{\TV}(\mu_1,\mu_2)- \Abs{ \frac{Z_{\Phi,\mu_2}-Z_{\Phi,\mu_1}}{2Z_{\Phi,\mu_2}} }},
\end{align*}
so that $\mathcal{U}_P\neq \mathcal{L}_P$ in this case.

The second motivation for considering lower bounds is to provide additional evidence for the increasing sensitivity of posteriors to perturbations in misfits or priors as the posteriors become more concentrated. 
A posterior $(\mu_1)_{\Phi_1}$ is said to be `more concentrated' than another posterior $(\mu_2)_{\Phi_2}$ if the corresponding evidence terms $Z_{\Phi_1,\mu_1}$ and $Z_{\Phi_2,\mu_2}$ satisfy $0<Z_{\Phi_1,\mu_1}<Z_{\Phi_2,\mu_2}$. 
A posterior $\mu_{\Phi_1}$ is said to be `more sensitive' with respect to a perturbation $\Delta \Phi$ in the misfit than another posterior $\mu_{\Phi_2}$, if $\mathfrak{d}_{\mathcal{P}(\Theta)}(\mu_{\Phi_1},\mu_{\Phi_1+\Delta \Phi})$ is larger than $\mathfrak{d}_{\mathcal{P}(\Theta)}(\mu_{\Phi_2},\mu_{\Phi_2+\Delta\Phi})$. Increased sensitivity with respect to perturbations of the prior is defined analogously.
The upper bounds in \cite{Sprungk2020} on the Kullback--Leibler divergence and the Hellinger, total variation, and 1-Wasserstein metrics show that increasing sensitivity \emph{can} occur. This is because the upper bounds involve division by the evidence, and thus increase as the values of the evidence become smaller; see \cite[Remark 9, Remark 18]{Sprungk2020}.
In contrast, lower bounds that involve division by the evidence will show that increasing sensitivity \emph{must} occur.

Below, we discuss conditions under which equality holds in the upper and lower bounds.
However, we do not expect our bounds to be tight in general, due to the inequalities that we apply in our proofs.
Instead, we aim for bounds that highlight the contributions from two sources: differences in evidences or functions thereof, and differences due to $\Phi_1-\Phi_2$ and $\mathfrak{d}_{\mathcal{P}(\Theta)}(\mu_1,\mu_2)$ for misfit and prior perturbations respectively.
In doing so, we obtain sufficient conditions for which the misfit-to-posterior map and the prior-to-posterior map are locally bi-Lipschitz continuous.

We believe it is natural to investigate sufficient conditions for local bi-Lipschitz continuity of the misfit-to-posterior and prior-to-posterior maps, given that sufficient conditions for the local Lipschitz continuity of these maps were presented in \cite{Sprungk2020}.
The lower bounds that we investigate below relate to the local Lipschitz continuity of the inverses of these maps, when the inverses exist.
Our results suggest that one in general cannot expect these maps to be locally bi-Lipschitz on their natural domains of definition.
This is because the evidence term can lead to noninjectivity.
On the other hand, our results show that one can partition the natural domains of each map according to level sets with respect to the evidence, such that the restriction of each map to each level set is bi-Lipschitz continuous.

Inverting the misfit-to-posterior map (resp., the prior-to-posterior map) may be relevant when one is interested in determining what misfit (resp. prior) leads to a posterior distribution for a given fixed prior (resp. misfit).
This idea is similar to the idea of finding a reverse or backward diffusion process to pair with a forward diffusion process, in the context of diffusion probabilistic models or score-based generative models. We defer the investigation of this idea to future work. 

\paragraph{Outline} After reviewing notation and some preliminary material in \Cref{sec_notation}, we state some lemmas concerning the noninjectivity of maps defined on misfits and the noninjectivity of maps defined on priors in \Cref{sec_noninjectivity}.
We then present bounds for the total variation metric between posteriors in \Cref{sectionTVbounds}.
In this section, we consider two approaches for obtaining bounds: one based on local Lipschitz continuity of the exponential function, and another based on the triangle inequality. We assess the relative merits of each approach depending on whether one considers misfit perturbations or prior perturbations.
In \Cref{sectionHellingerbounds}, \Cref{sectionKLbounds}, and \Cref{sectionW1bounds}, we present bounds for the Hellinger metric, Kullback--Leibler divergence, and 1-Wasserstein metric respectively.

\subsection{Related literature}

This work focuses on local Lipschitz stability of the posterior with respect to perturbations in the misfits and perturbations in the priors. 
As was shown in \cite{Sprungk2020}, well-posedness of Bayesian inverse problems with respect to the data is a consequence of the local Lipschitz stability of the posterior with respect to perturbations in the misfits.
While well-posedness with respect to the data is important in Bayesian inverse problems, it is not the main subject of investigation in this work, and we do not consider it further. Readers interested in well-posedness can see \cite[Section 4]{Stuart2010} for prototypical well-posedness results, as well as \cite{Latz2023} and the references therein for more recent results in this context.

The local Lipschitz stability bounds in \cite{Sprungk2020} of posteriors with respect to misfit perturbations have been applied in different contexts. 
The result \cite[Corollary 19]{Sprungk2020}, which gives conditions for well-posedness in 1-Wasserstein metric of the posterior with respect to perturbations in the data, was used to study the robustness of conditional generative models with respect to perturbations in the data; see e.g. \cite[Lemma 3]{Altekrueger2023}. 
The result \cite[Theorem 15]{Sprungk2020}, which concerns local Lipschitz stability in the 1-Wasserstein metric of posteriors with respect to prior perturbations, was used to prove error bounds for approximate posteriors that were obtained using a generative adversarial neural network \cite[Theorem 2]{Mucke2023}.
The results \cite[Theorem 5, Theorem 8]{Sprungk2020} give sufficient conditions for local Lipschitz stability in the total variation and Hellinger metrics with respect to perturbations in the misfit.
These results were used to carry out a quantitative error analysis of an exploratory policy improvement algorithm in the context of $q$-learning; see the proof of \cite[Lemma 3.5]{Tang2024}.
The result \cite[Theorem 11]{Sprungk2020}, which gives conditions for local Lipschitz stability in the Kullback--Leibler divergence with respect to misfit perturbations, was applied to study the experimental design problem of choosing observation operators to mitigate errors in the forward model \cite[Theorem 2.1]{Cvetkovic2024}.

The literature on Bayesian inverse problems contains local Lipschitz stability results that are distinct from those in \cite{Sprungk2020}.
For example, in \cite[Section 4]{Teckentrup2018}, stability results are given for the approximate posterior measures in the Hellinger metric, where the approximations result from Gaussian process approximations of the forward model or the misfit.
The expected mean squared Hellinger metric of random approximate posterior measures resulting from random approximations of the misfit and forward model was investigated in \cite{LieSullivanTeckentrup2018}, without the assumption that the approximations are given by Gaussian processes. 
In \cite{Habeck2020}, local Lipschitz stability in the total variation metric and 1-Wasserstein metric of a deterministic, doubly intractable distribution with respect to a normalising constant was investigated; random distributions were also considered.
In the context of Bayesian optimal experimental design, the stability of the design with respect to the likelihood was shown in \cite[Proposition 3.3]{Duong2023}, and convergence analysis was performed for the design under convergence hypotheses on the misfits and forward models in Theorem 3.5 and Theorem 4.4 of \cite{Duong2023} respectively.
Local Lipschitz stability of posteriors with respect to perturbations in misfits and priors was considered for a class of integral probability metrics associated to distance-like costs in \cite[Section 3]{Garbunoinigo2023}.

Lower bounds of the type shown in \eqref{eq_prototypical_lower_bounds} do not appear in any of the above-cited work.

\subsection{Notation and preliminaries}
\label{sec_notation}
Let $(\Theta,\Sigma)$ denote the measurable space of admissible parameter values. Denote by $\mathcal{M}(\Theta)$ and $\mathcal{P}(\Theta)$ the set of measures and probability measures on the measurable space $(\Theta,\Sigma)$ respectively. 
For $\mu\in\mathcal{M}(\Theta)$, $\supp{\mu}$ denotes the support of $\mu$.
Given $\mu,\nu\in\mathcal{M}(\Theta)$, we denote the absolute continuity of $\mu$ with respect to $\nu$ by $\mu\ll\nu$, the mutual absolute continuity of $\mu$ and $\nu$ by $\mu\sim\nu$, and the singularity of $\mu$ and $\nu$ by $\mu\bot \nu$.
For integrals, we will omit the domain of integration and the variable of integration when it is clear from the context, e.g. $\int f\rd \mu\equiv \int_{\Theta} f(\theta)\mu(\rd\theta)$.
We denote the cardinality of a set $S$ by $\card{S}$ when $S$ is countable.

We write $a\vee b\coloneqq \max\{a,b\}$, $a\wedge b\coloneqq \min\{a,b\}$, $a_+\coloneqq a \vee 0$ and $a_-\coloneqq (-a)\vee 0$ for $a,b\in\R$. In particular, for $\R$-valued functions, we have $f=f_+-f_-$ and $\abs{f}=f_++f_-$.

Let $0< t \leq s$. Since $\tfrac{x}{1+x}\leq \log{(1+x)} \leq x$ for all $x\geq 0$, it follows that
\begin{align*}
   \abs{\log t -\log s} =& \log \frac{s}{t} = \log \left(1+\left(\frac{s}{t}-1\right) \right) \leq \frac{s}{t} - 1 = \frac{1}{t}(s-t) = \frac{1}{t} \abs{ s-t}
   \\
      \abs{\log t -\log s} =& \log \frac{s}{t} = \log \left(1+\left(\frac{s}{t}-1\right) \right) \geq 1-\frac{t}{s} = \frac{1}{s}(s-t) = \frac{1}{s} \abs{ s-t}.
\end{align*}
Thus, for $s,t>0$ and $x,y\in\R$,
\begin{subequations}
\begin{align}
\label{eq_Lipschitz_continuity_log_function}
    \frac{1}{s\vee t} \abs{s- t}\leq& \abs {\log t -\log s}  \leq \frac{1}{s\wedge t} \abs{s- t}
    \\
    \label{eq_Lipschitz_continuity_exp_function}
    \Longleftrightarrow  [\exp(x)\wedge \exp(y)]\abs{x- y}\leq& \abs{\exp(x) -\exp(y)}  \leq [\exp(x)\vee \exp(y)]\abs{x- y},
\end{align}
\end{subequations}
where equality holds in \eqref{eq_Lipschitz_continuity_log_function} if and only if $s=t$ and equality holds in \eqref{eq_Lipschitz_continuity_exp_function} if and only if $x=y$.

For $\mu\in\mathcal{P}(\Theta)$, $0<p<\infty$ and a nonempty subset $I\subset\R$, let $\norm{f}_{L^p_\mu}\coloneqq (\int \abs{f}^p\rd \mu)^p$, and let $L^p_\mu(\Theta,I)\coloneqq \{f:\Theta\to I\ :\ \norm{f}_{L^p_\mu}<\infty\}$. Let $C(\Theta,I)\coloneqq \{f:\Theta\to I\ :\ \text{$f$ is continuous}\}$ and $C_b(\Theta,I)\coloneqq \{f\in C(\Theta,I)\ :\ \text{$f$ is bounded}\}$. We denote the support of $f$ by $\supp{f}$.
Let $L^0(\Theta,I)\coloneqq \{ f: \Theta\to I\ :\ \text{$f$ is $\Sigma$-measurable}\}$ and $L^\infty_\mu(\Theta,I)$ denote the vector space of $\mu$-essentially bounded functions from $\Theta$ to $I$. 
Consider the following maps on $L^0(\Theta,\R)\times \mathcal{P}(\Theta)$:
\begin{subequations}
\label{eq_maps}
 \begin{align}
 (\Phi,\mu)\mapsto Z_{\Phi,\mu}\coloneqq & \int \exp(-\Phi)\rd \mu
 \label{eq_normalisationConstant_function}
 \\
  (\Phi,\mu)\mapsto \ell_{\Phi,\mu}\coloneqq & \exp\left( -\Phi-\log Z_{\Phi,\mu}\right) 
  \label{eq_likelihood_function}
  \\
  (\Phi,\mu)\mapsto \mu_{\Phi}(\rd \theta) \coloneqq & \ell_{\Phi,\mu}(\theta)\mu(\rd \theta).
  \label{eq_posterior_function}
\end{align}
\end{subequations}
Note that for every $\epsilon>0$, $Z_{\Phi,\mu}\geq \int_{\{\Phi<\epsilon\}}\exp(-\Phi)\rd\mu >\exp(-\epsilon)\mu(\Phi<\epsilon)$.
Thus, if there exists some $\epsilon>0$ such that $\mu(\Phi<\epsilon)>0$, then $Z_{\Phi,\mu}>0$.
A sufficient condition for $Z_{\Phi,\mu}>0$ to also be finite is that $\int \Phi~\rd \mu$ is finite, since by Jensen's inequality applied to the convex function $x\mapsto -\log x$, we have
\begin{equation*}
-\log Z_{\Phi,\mu}=-\log\int \exp(-\Phi)\rd \mu \leq \int \Phi~\rd \mu.
\end{equation*}
Finally, the notation $a\leftarrow b$ means that we set $a$ to have the value $b$.

\section{Noninjectivity}
\label{sec_noninjectivity}

In this section, we show that the misfit-to-evidence map $\Phi\mapsto Z_{\Phi,\mu}$ and the prior-to-evidence map $\mu\mapsto Z_{\Phi,\mu}$ are in general not injective.
The proofs of the results in this section are given in \Cref{section_noninjectivity_proofs}.

We first consider the noninjectivity of functions of misfits.
\begin{restatable}{lemma}{LemmaNoninjectivityMisfitMaps}
 \label{lemma_noninjectivity_of_dataMisfit_maps}
Fix an arbitrary $\mu\in\mathcal{P}(\Theta)$.  
 \begin{enumerate}
  \item If $\supp{\mu}$ is uncountable or if $\card{\supp{\mu}}\geq 2$, then the map $\Phi\mapsto Z_{\Phi,\mu}$ is not injective on $L^p_\mu(\Theta,\R)$ for any $p\in [1,\infty]$.
  \label{item_lemma_noninjectivity_of_dataMisfit_maps_noninjective}
  \item If $\Phi$ is such that $Z_{\Phi,\mu} \in \R_{>0}$, then for every $c\in\R\setminus\{0\}$, $Z_{\Phi+c,\mu}=\exp(-c) Z_{\Phi,\mu}$.
   \label{item_lemma_noninjectivity_of_dataMisfit_maps_noninvariance_under_translations}
   \item If $\Phi_i\in L^0(\Theta,\R)$ is such that $Z_{\Phi_i,\mu}\in \R_{>0}$, $i=1,2$, then $\log \ell_{\Phi_1,\mu}=\log \ell_{\Phi_2,\mu}$ $\mu$-a.s. if and only if $\Phi_1-\Phi_2$ is $\mu$-a.s. constant. In turn, $\Phi_1-\Phi_2$ is $\mu$-a.s. constant if and only if it is equal to $\log Z_{\Phi_2,\mu}-\log Z_{\Phi_1,\mu}$.
 \label{item_lemma_noninjectivity_of_dataMisfit_maps_equivalent_condition_for_agreement}
 \end{enumerate}
\end{restatable}
The significance of \Cref{lemma_noninjectivity_of_dataMisfit_maps}\ref{item_lemma_noninjectivity_of_dataMisfit_maps_equivalent_condition_for_agreement} is that the map $\Phi\mapsto \log \ell_{\Phi} $ is not invertible on $L^p_\mu(\Theta,\R)$ for any $p\in [1,\infty]$.
In particular, it is not in general true that for some $M>0$, there exists $C(M)>0$ such that for every $\Phi_i$ with $\norm{\Phi_i}_{L^\infty_\mu}\leq M<\infty$, 
\begin{equation}
\label{eq_lower_bound_desideratum}
 \norm{\Phi_1-\Phi_2}_{L^\infty_\mu}\leq C(M) \norm{\log \ell_{\Phi_1,\mu}-\log\ell_{\Phi_2,\mu} }_{L^\infty_\mu}.
\end{equation}
For example, fix $\Phi_1$ and let $\Phi_2=\Phi_1+c$ for some constant $c\in\R$. Then the left-hand side in \eqref{eq_lower_bound_desideratum} will be strictly positive, while the right-hand side will be equal to zero. 
\Cref{example_misfit_to_log_likelihood_map_not_locally_invertible} below shows that \eqref{eq_lower_bound_desideratum} can fail even if $\Phi_1-\Phi_2$ is not $\mu$-a.s. constant.
\begin{example}
\label{example_misfit_to_log_likelihood_map_not_locally_invertible}
Let $\mu\in\mathcal{P}(\Theta)$, $\Phi_1\in L^0(\Theta,\R_{\geq 0})$ be such that $Z_{\Phi_1,\mu}\in \R_{>0}$, and $M>0$ be arbitrary. 
Let $(S_n)_{n\in\N}\subset\Sigma$ be a collection of measurable subsets of positive $\mu$-measure that increase to $\Theta$, i.e. $\mu(S_n)>0$ and $S_n\subset S_{n+1}$ for every $n\in\N$, and $\lim_{n\to\infty}S_n=\Theta$.
For every $n\in\N$, define $\Phi_{n}\in L^0(\Theta,\R_{\geq 0})$ according to
 \begin{equation}
 \label{eq_Phi_n}
  \Phi_{n}\coloneqq \Phi_1+M\mathbb{I}_{S_n}=\begin{cases} 
              \Phi_1+M, & \text{on $S_n$}
              \\
              \Phi_1, & \text{on $S_n^\complement$.}
             \end{cases}
 \end{equation}
Then for every $n\in\N$, $\norm{\Phi_n-\Phi_1}_{L^\infty_\mu}=M$, and for every $p\in [1,\infty)$, $\norm{\Phi_n-\Phi_1}_{L^p_\mu}=M(\mu(S_n))^{1/p}$, so the left-hand side of \eqref{eq_lower_bound_desideratum} is bounded away from zero by $M$ if $p=\infty$ and by $M\min_n \mu(S_n)^{1/p}=M\mu(S_1)^{1/p}$ for $1\leq p<\infty$.
On the other hand, since $(S_n)_{n\in\N}$ increases to $\Theta$, it follows that
\begin{align*}
 Z_{\Phi_{n},\mu}=\exp(-M)\int_{S_n} \exp(-\Phi_1)\rd \mu+ \int_{S_n^\complement} \exp(-\Phi_1)\rd \mu\xrightarrow[n\to\infty]{} \exp(-M)Z_{\Phi_1,\mu},
\end{align*}
which is equivalent to $\lim_{n\to\infty}(-\log Z_{\Phi_1,\mu}+\log Z_{\Phi_{n},\mu})=-M$.
Thus, for any $p\in [1,\infty]$,
\begin{align*}
 \norm{\log \ell_{\Phi_1,\mu}-\log\ell_{\Phi_{n},\mu} }_{L^p_\mu}=&\Norm{-\Phi_1-\log Z_{\Phi_1,\mu}+\Phi_{n}+\log Z_{\Phi_{n},\mu}}_{L^p_\mu}
 \\
 =& \Norm{ M\mathbb{I}_{S_n} -\log Z_{\Phi_1,\mu}+\log Z_{\Phi_{n},\mu}}_{L^p_\mu}\xrightarrow[n\to\infty]{}  0,
\end{align*}
where the first equation follows from the definition \eqref{eq_likelihood_function} of $\ell_{\Phi_i,\mu}$, the second equation follows from \eqref{eq_Phi_n}, and the convergence statement follows from the hypothesis that $(S_n)_{n\in\N}$ increases to $\Theta$.
Thus, the right-hand side of \eqref{eq_lower_bound_desideratum} converges to zero for any $1\leq p\leq \infty$, and there does not exist some $0<C(M)<\infty$ such that \eqref{eq_lower_bound_desideratum} holds.
\end{example}

For an arbitrary $\nu\in\mathcal{M}(\Theta)$, define
\begin{equation}
 \label{eq_measures_dominated_by_nu}
 D(\nu)\coloneqq \{ \mu'\in\mathcal{P}(\Theta)\ :\ \mu'\ll\nu \}.
\end{equation}
We now describe the noninjectivity of the prior-to-evidence map on $D(\nu)$ and a sufficient condition on the misfit $\Phi$ for the prior-to-posterior map on $D(\nu)$ to be injective.
\begin{restatable}{lemma}{LemmaNoninjectivityPriorMaps}
\label{lemma_noninjectivity_of_prior_maps}
 Let $\nu\in\mathcal{M}(\Theta)$, $D(\nu)$ be as in \eqref{eq_measures_dominated_by_nu}, and $\Phi\in L^0(\Theta,\R)$ be such that $Z_{\Phi,\mu}\in \R_{>0}$ for every $\mu\in D(\nu)$.
 \begin{enumerate}
  \item \label{item_lemma_noninjectivity_of_prior_maps_noninjective}
  If $\supp{\nu}$ is uncountable or $\card{\supp{\nu}}\geq 3$, then the map $D(\nu)\ni \mu\mapsto Z_{\Phi,\mu}\in\R_{>0}$ is not injective.
  \item \label{item_lemma_noninjectivity_of_prior_maps_equivalent_condition_for_agreement}
  Suppose $\mu_i\in D(\nu)$, $i=1,2$ and $\exp(-\Phi)\in\R_{>0}$ $\nu$-a.s. 
 Then $(\mu_1)_\Phi=(\mu_2)_\Phi$ if and only if $\mu_1=\mu_2$.
 \end{enumerate}
\end{restatable}
The restriction to the set of priors that admit a common dominating measure is not new, and has an important advantage of ensuring that misspecification of the prior does not occur \cite[pp.6-7]{GhosalvanderVaart2017}.
By \Cref{lemma_esssup_essinf_different_measures_different_functions}\ref{item_esssup_essinf_different_measures}, if $\Phi\in L^\infty_\nu(\Theta,\R)$, then for every $\mu\in D(\nu)$, $\Phi \in L^\infty_\mu(\Theta,\R)$, which in turn implies that $Z_{\Phi,\mu}\in\R_{>0}$. Thus, the set of misfits $\Phi$ such that $Z_{\Phi,\mu}\in\R_{>0}$ for every $\mu\in D(\nu)$ contains $L^\infty_\nu(\Theta,\R)$, and the hypothesis on $\Phi$ of \Cref{lemma_noninjectivity_of_prior_maps} is not vacuous.
We comment on the assumption of bounded misfits in \Cref{remark_bounded_misfits} below.

We provide an example where the prior-to-evidence map $D(\nu)\ni\mu\mapsto Z_{\Phi,\mu}\in\R_{>0}$ is noninjective. 
\begin{example}
\label{example_prior_to_evidence}
 Let $\Theta=\R$, $\nu$ be Lebesgue measure, $\mu_1$ be the uniform measure on $[-2,1]$, $\mu_2$ be the uniform measure on $[-1,2]$, and 
$\Phi\coloneqq \mathbb{I}_{[-2,-1]\cup[1,2]}$. Then $\tfrac{\rd\mu_1}{\rd\nu}=\frac{1}{3}\mathbb{I}_{[-2,1]}$, $\tfrac{\rd\mu_2}{\rd\nu}=\frac{1}{3}\mathbb{I}_{[-1,2]}$, $\Phi\in L^\infty_{\mu_i}(\Theta,\R_{\geq 0})$ for $i=1,2$, and
\begin{equation*}
 \exp(-\Phi)(x)=\begin{cases}
              \exp(-1), & x\in[-2,-1]\cup [1,2]
              \\
              1, & x\notin [-2,-1]\cup [1,2],
             \end{cases}
\end{equation*}
 which imply
 \begin{align*}
  Z_{\Phi,\mu_1}=&\int \exp(-\Phi)\rd\mu_1=\frac{1}{3} \int_{-2}^{-1}\exp(-1) \rd x+\frac{1}{3}\int_{-1}^{1} 1 \rd x=\frac{\exp(-1)+2}{3}
  \\
  Z_{\Phi,\mu_2}=&\int \exp(-\Phi)\rd\mu_2= \frac{1}{3}\int_{-1}^{-1}1\rd x+\frac{1}{3}\int_{-1}^{2} \exp(-1) \rd x=\frac{2+\exp(-1)}{3}.
 \end{align*}
\end{example}
The following modification of \Cref{example_prior_to_evidence} shows that without the hypothesis that $\exp(-\Phi)\in\R_{>0}$ $\nu$-a.s., the conclusion of \Cref{lemma_noninjectivity_of_prior_maps}\ref{item_lemma_noninjectivity_of_prior_maps_equivalent_condition_for_agreement} need not hold, i.e. the prior-to-posterior map for a fixed misfit $\Phi$ need not be injective.
\begin{example}
 \label{example_prior_to_posterior_noninjective}
  Let $\Theta=\R$, $\nu$ be Lebesgue measure, $\mu_1$ be the uniform measure on $[-2,1]$, $\mu_2$ be the uniform measure on $[-1,2]$, and 
  \begin{equation*}
   \exp(-\Phi)(x)=\mathbb{I}_{[-1,1]}(x).
  \end{equation*}
Then
\begin{equation*}
 Z_{\Phi,\mu_1}=\frac{1}{3}\int_{-2}^{1}\exp(-\Phi)(x)\rd x=\frac{1}{3}\int_{-1}^{1}\rd x=\frac{1}{3}\int_{-1}^{2}\exp(-\Phi)(x)\rd x=Z_{\Phi,\mu_2}.
\end{equation*}
The above and \eqref{eq_likelihood_function} together imply that 
\begin{equation}
\label{eq_intermediate1}
 \ell_{\Phi,\mu_1}=\frac{\exp(-\Phi)}{Z_{\Phi,\mu_1}}=\frac{\exp(-\Phi)}{Z_{\Phi,\mu_2}}=\ell_{\Phi,\mu_2},\quad \text{$\nu$-a.s.}
\end{equation}
By the choices of $\mu_1$ and $\mu_2$,
\begin{equation}
\label{eq_intermediate2}
\frac{\rd\mu_1}{\rd\nu}\mathbb{I}_{[-1,1]}=\frac{1}{3}\mathbb{I}_{[-2,1]}\mathbb{I}_{[-1,1]} =\frac{1}{3}\mathbb{I}_{[-1,2]}\mathbb{I}_{[-1,1]} =\frac{\rd\mu_2}{\rd\nu}\mathbb{I}_{[-1,1]}.
\end{equation}
Now
\begin{align*}
 \frac{\rd(\mu_1)_\Phi}{\rd\nu}=\ell_{\Phi,\mu_1}\frac{\rd\mu_1}{\rd\nu}=& \ell_{\Phi,\mu_1}\frac{\rd\mu_1}{\rd\nu}\mathbb{I}_{[-1,1]} & \exp(-\Phi)=\mathbb{I}_{[-1,1]}
 \\
 =&\ell_{\Phi,\mu_2}\frac{\rd\mu_1}{\rd\nu}\mathbb{I}_{[-1,1]} & \text{by \eqref{eq_intermediate1}}
 \\
 =& \ell_{\Phi,\mu_2}\frac{\rd\mu_2}{\rd\nu}\mathbb{I}_{[-1,1]} & \text{by \eqref{eq_intermediate2}}
 \\
 =& \ell_{\Phi,\mu_2}\frac{\rd\mu_2}{\rd\nu} & \exp(-\Phi)=\mathbb{I}_{[-1,1]}
 \\
 =&\frac{\rd(\mu_2)_\Phi}{\rd\nu},
\end{align*}
which implies that $(\mu_1)_\Phi=(\mu_2)_\Phi$. Since $\mu_1\neq \mu_2$, it follows that $D(\nu)\ni \mu\mapsto \mu_\Phi\in D(\nu)$ is not injective for this choice of $\Phi$.
\end{example}

\section{Bounds in the total variation metric}
\label{sectionTVbounds}

Let $\nu,\mu_i\in\mathcal{P}(\Theta)$ be such that $\mu_i\ll\nu$ for $i=1,2$.
The total variation metric is defined by
\begin{equation}
 d_\TV(\mu_1,\mu_2)\coloneqq \sup_{A\in\Sigma}\abs{\mu_1(A)-\mu_2(A)}=\frac{1}{2}\int\Abs{\frac{\rd \mu_1}{\rd\nu}-\frac{\rd\mu_2}{\rd\nu}}\rd\nu.
 \label{eq_total_variation_metric}
 \end{equation}
 The value of $d_\TV(\mu_1,\mu_2)$ does not depend on the choice of $\nu$, and $d_{\TV}(\mu_1,\mu_2)$ takes values in $[0,1]$.
 The maximal value of 1 is attained if and only if the densities $\tfrac{\rd\mu_1}{\rd\nu}$ and $\tfrac{\rd\mu_2}{\rd\nu}$ have disjoint regions of positivity, i.e. $\mu_1$ and $\mu_2$ are singular; see e.g. \cite[Section B.1]{GhosalvanderVaart2017}. 
 If $\mu_1$ and $\mu_2$ are not singular, this does not imply that either $\mu_1\ll\mu_2$ or that $\mu_2\ll\mu_1$, as the example below shows.
 \begin{example}
  \label{example_translated_uniform_measure}
  Let $\nu$ be the Lebesgue measure on $\R$, $\mu$ be the uniform measure on $[0,1]$, and for $\tau\in (-1,1)$, let $\mu_\tau$ be the uniform measure on $[\tau,1+\tau]$. 
  If $\tau \neq 0$, then $[0,1]\setminus [\tau,1+\tau]$ and $[\tau,1+\tau]\setminus[0,1]$ are both nonempty, so neither $\mu\ll\mu_\tau$ nor $\mu_\tau\ll\mu$ hold.
  On the other hand, by \eqref{eq_total_variation_metric},
  \begin{equation*}
    d_{\TV}(\mu,\mu_\tau)=\frac{1}{2}\left(\int_{[0,1]\setminus[\tau,1+\tau]}\rd \nu+\int_{[\tau,1+\tau]\setminus[0,1]}\rd \nu\right)=\frac{1}{2}\left(\abs{\tau}+\abs{\tau}\right)=\abs{\tau}.
  \end{equation*}
  Thus as $\tau\to 0$, $\mu_\tau$ converges to $\mu$ in the total variation metric.
 \end{example}

\subsection{Perturbations of the misfit}
\label{sectionTVbounds_misfit_perturbation}

The proofs of the results in this section are given in \Cref{section_TVbounds_Misfitperturbations_proofs}.

The upper and lower bounds on the total variation metric in the next result follow from the local Lipschitz continuity \eqref{eq_Lipschitz_continuity_exp_function} of the exponential function.
\begin{restatable}{proposition}{TVboundsMisfitperturbationsLipschitz}
 \label{proposition_TV_bounds_misfit_perturbations_via_Lipschitz_continuity}
 Let $\mu\in\mathcal{P}(\Theta)$ and $\Phi_i\in L^0(\Theta,\R)$ be such that $\mu_{\Phi_i}\in \mathcal{P}(\Theta)$, $i=1,2$.
\begin{enumerate}
 \item  \label{proposition_TV_bounds_misfit_perturbations_via_Lipschitz_continuity_upper_bound}  If in addition $\Phi_i\in L^1_\mu(\Theta,\R_{\geq 0})$ for $i=1,2$, then
   \begin{equation*}
    d_{\TV}(\mu_{\Phi_1},\mu_{\Phi_2}) \leq \frac{1}{2}\frac{1}{Z_{\Phi_1,\mu}\wedge Z_{\Phi_2,\mu}}\Norm{\Phi_1-\Phi_2+\log \frac{Z_{\Phi_1,\mu}}{Z_{\Phi_2,\mu}}}_{L^1_\mu}.
  \end{equation*}
\item \label{proposition_TV_bounds_misfit_perturbations_via_Lipschitz_continuity_lower_bound}
If in addition $\Phi_i\in L^\infty_\mu(\Theta,\R)$, $i=1,2$, then
  \begin{equation*}
    d_{\TV}(\mu_{\Phi_1},\mu_{\Phi_2}) \geq \frac{1}{2}\left(\frac{\exp(-\Norm{\Phi_1}_{L^\infty_\mu})}{Z_{\Phi_1,\mu}}\wedge\frac{ \exp(-\Norm{\Phi_2}_{L^\infty_\mu})}{ Z_{\Phi_2,\mu}}\right)\Norm{\Phi_1-\Phi_2+\log \frac{Z_{\Phi_1,\mu}}{Z_{\Phi_2,\mu}}}_{L^1_\mu}.
  \end{equation*}
\end{enumerate}
If $\Phi_1-\Phi_2$ is constant $\mu$-a.s., then equality holds in both the upper and lower bounds.
\end{restatable}
 
\begin{remark}[Nonnegativity hypothesis for misfits bounded from below]
\label{remark_nonnegativity_hypothesis}
 If $\mu\in\mathcal{P}(\Theta)$ and $\Phi\in L^0(\Theta,\R)$ satisfies $\essinf_\mu\Phi>-\infty$, then we may assume that $\essinf_\mu \Phi=0$. 
There is no loss of generality, since if $-\infty<\essinf_\mu \Phi<0$, then by defining $\Psi\coloneqq \Phi-\essinf_\mu\Phi$, we have $\essinf_\mu \Psi=0$.
By \Cref{lemma_noninjectivity_of_dataMisfit_maps}\ref{item_lemma_noninjectivity_of_dataMisfit_maps_equivalent_condition_for_agreement} it follows that $\log \ell_{\Psi,\mu}=\log \ell_{\Phi,\mu}$, and by \eqref{eq_posterior_function} it follows that $\mu_{\Psi}=\mu_{\Phi}$. 
\end{remark}

\begin{remark}[Bounded misfits]
 \label{remark_bounded_misfits}
 The assumption that $\Phi_i\in L^\infty_\mu(\Theta,\R)$ for $i=1,2$ that we made in \cref{proposition_TV_bounds_misfit_perturbations_via_Lipschitz_continuity_lower_bound} appears to be restrictive.
 It is not satisfied in the case where $\Theta=\R^d$ for $d\in\N$ and $\Phi_i$ is generated by a linear forward map and Gaussian observation noise, for example. 
 However, similar boundedness assumptions have appeared in the theory of nonlinear Bayesian statistical inverse problems; see e.g. \cite[equation (1.17)]{Nickl2023}. 
 The boundedness assumption can hold if the space $\Theta$ of admissible parameters is itself bounded; such boundedness conditions were shown to be important in the Gaussian regression with random design setting in \cite{Birge2004}, for example. 
 Another sufficient condition for the boundedness assumption to hold is that the prior $\mu$ has bounded support.
 We discuss this condition in \Cref{remark_priors_with_bounded_support} below.
\end{remark}

The next result obtains upper and lower bounds for the total variation and misfit perturbations by using the triangle inequality.
The upper bound was stated in \cite[Theorem 8]{Sprungk2020}.
The lower bounds are new, to the best of our knowledge. 
However, for the lower bound to be nontrivial, one requires both $\mu(\Phi_1-\Phi_2>0)$ and $\mu(\Phi_1-\Phi_2<0)$ to be positive. 
The lower bound in \Cref{proposition_TV_bounds_misfit_perturbations_via_Lipschitz_continuity}\ref{proposition_TV_bounds_misfit_perturbations_via_Lipschitz_continuity_lower_bound} that was obtained using local Lipschitz continuity of the exponential function does not require this condition.

\begin{restatable}{proposition}{TVboundsMisfitperturbationsTriangle}
\label{proposition_TV_bounds_misfit_perturbations_via_triangle_inequality}
 Let $\Phi_i\in L^0_\mu(\Theta,\R)$ be such that $\mu_{\Phi_i}\in\mathcal{P}(\Theta)$, $i=1,2$.
 \begin{enumerate}
 \item \label{item_proposition_TV_bounds_misfit_perturbations_via_triangle_inequality_upper_bound}
If $\essinf_\mu\Phi_i= 0$ for $i=1,2$, then
\begin{align}
\abs{Z_{\Phi_2,\mu}-Z_{\Phi_1,\mu}}\leq &\norm{ e^{-\Phi_1}-e^{-\Phi_2}}_{L^1_\mu}\leq\Norm{\Phi_1-\Phi_2}_{L^1_\mu}
\label{eq_Lipschitz_continuity_evidence_unnormalisedLikelihood_wrt_misfit}
\\
  d_{\TV}(\mu_{\Phi_1},\mu_{\Phi_2})\leq& \frac{1}{2}\frac{1}{Z_{\Phi_1,\mu}\vee Z_{\Phi_2,\mu}}\left(\norm{ e^{-\Phi_1}-e^{-\Phi_2}}_{L^1_\mu}+ \abs{Z_{\Phi_2,\mu}-Z_{\Phi_1,\mu}}\right),
  \label{eq_TV_upper_bound_misfit_perturbation_triangle}
 \end{align}
 If $\Phi_1=\Phi_2$ $\mu$-a.s., then equality holds in the upper bound in \eqref{eq_Lipschitz_continuity_evidence_unnormalisedLikelihood_wrt_misfit} and  \eqref{eq_TV_upper_bound_misfit_perturbation_triangle}.
 Equality holds in the lower bound in \eqref{eq_Lipschitz_continuity_evidence_unnormalisedLikelihood_wrt_misfit} if and only if $\Phi_1-\Phi_2$ $\mu$-a.s. does not change sign.

 \item 
 \label{item_proposition_TV_bounds_misfit_perturbations_via_triangle_inequality_lower_bound}
 If $\Phi_i\in L^\infty_\mu(\Theta,\R)$ for $i=1,2$, and both $\mu(\Phi_1-\Phi_2>0)$ and $\mu(\Phi_1-\Phi_2<0)$ are positive, then
 \begin{equation*}
  d_{\TV}(\mu_{\Phi_1},\mu_{\Phi_2})>
  \begin{cases}
  \frac{1}{2}\frac{\exp(-\norm{\Phi_1}_{L^\infty_\mu})\wedge\exp(- \norm{\Phi_2}_{L^\infty_\mu})}{Z_{\Phi_1,\mu}\wedge Z_{\Phi_2,\mu}}\Norm{\Phi_1-\Phi_2}_{L^1_\mu} , & Z_{\Phi_1,\mu}=Z_{\Phi_2,\mu}
  \\
  \exp(-\Norm{\Phi_1}_{L^\infty_\mu})Z_{\Phi_2,\mu}^{-1}\int_{\{\Phi_1>\Phi_2\}} \abs{\Phi_1-\Phi_2}\rd \mu, & Z_{\Phi_1,\mu}>Z_{\Phi_2,\mu},
  \\
  \exp(-\Norm{\Phi_2}_{L^\infty_\mu})Z_{\Phi_1,\mu}^{-1}\int_{\{\Phi_2>\Phi_1\}} \abs{\Phi_1-\Phi_2}\rd \mu, & Z_{\Phi_1,\mu}<Z_{\Phi_2,\mu}.
  \end{cases}
 \end{equation*}
 \item \label{item_proposition_TV_bounds_misfit_perturbations_via_triangle_inequality_lower_bound_no_sign_change}
   If $\Phi_i\in L^\infty_\mu(\Theta,\R)$ for $i=1,2$ and if either $\mu(\Phi_1-\Phi_2\geq 0)=1$ or $\mu(\Phi_1-\Phi_2\leq 0)=1$, then 
 \begin{equation*}
  d_{\TV}(\mu_{\Phi_1},\mu_{\Phi_2})\geq 
  \begin{cases}
  \frac{1}{2}\frac{\exp(-\norm{\Phi_1}_{L^\infty_\mu})\wedge\exp(- \norm{\Phi_2}_{L^\infty_\mu})}{Z_{\Phi_1,\mu}\wedge Z_{\Phi_2,\mu}}\Norm{\Phi_1-\Phi_2}_{L^1_\mu} , & Z_{\Phi_1,\mu}=Z_{\Phi_2,\mu}
  \\
  0, & \text{otherwise.}
  \end{cases}
 \end{equation*}
 \end{enumerate}
\end{restatable}

\begin{remark}
\label{remark_L1_norm_unnormalised_likelihoods_and_absolute_difference_normalisation_constants}
\Cref{item_proposition_TV_bounds_misfit_perturbations_via_triangle_inequality_upper_bound} states that $\abs{Z_{\Phi_2,\mu}-Z_{\Phi_1,\mu}}\leq \Norm{\Phi_1-\Phi_2}_{L^1_\mu}$.
In general, one cannot bound $\abs{Z_{\Phi_1,\mu}-Z_{\Phi_2,\mu}}$ from below by $\norm{\Phi_1-\Phi_2}_{L^1_\mu}$.
If such a bound were possible, then it would imply that $\Phi\mapsto Z_{\Phi,\mu}$ is injective, thus contradicting \Cref{lemma_noninjectivity_of_dataMisfit_maps}\ref{item_lemma_noninjectivity_of_dataMisfit_maps_noninjective}.
One situation in which such a lower bound is possible is if $\Phi_i\in L^\infty_\mu(\Theta,\R)$ for $i=1,2$ and either $\mu(\Phi_1-\Phi_2\geq 0)=1$ or $\mu(\Phi_1-\Phi_2\leq 0)=1$.
Under these hypotheses,
\begin{equation*}
 \abs{Z_{\Phi_1,\mu}-Z_{\Phi_2,\mu}}= \norm{e^{-\Phi_1}-e^{-\Phi_2}}_{L^1_\mu}\geq \left[\exp(-\norm{\Phi_1}_{L^\infty_\mu})\wedge \exp(-\norm{\Phi_2}_{L^\infty_\mu})\right]\norm{\Phi_1-\Phi_2}_{L^1_\mu},
\end{equation*}
where the equation and inequality follow by \Cref{lemma_L1norm_difference_of_unnormalised_likelihoods_minus_abs_diff_normalisation_constants_misfit_perturbation}\ref{item2_lemma_L1norm_difference_of_unnormalised_likelihoods_minus_abs_diff_normalisation_constants_misfit_perturbation} and \eqref{eq_Lipschitz_continuity_exp_function} respectively.
\end{remark}

\paragraph{Comparison of approaches} 

The upper bounds on $d_{\TV}(\mu_{\Phi_1},\mu_{\Phi_2})$ obtained by Lipschitz continuity and by the triangle inequality are
 $\norm{-\Phi_1-\log Z_{\Phi_1,\mu}+\Phi_2+\log Z_{\Phi_2,\mu}}_{L^1_\mu}$ and $\norm{e^{-\Phi_1}-e^{-\Phi_2}}_{L^1_\mu}+\abs{Z_{\Phi_2,\mu}-Z_{\Phi_1,\mu}}$ respectively; see  \Cref{proposition_TV_bounds_misfit_perturbations_via_Lipschitz_continuity}\ref{proposition_TV_bounds_misfit_perturbations_via_Lipschitz_continuity_upper_bound}  and \Cref{proposition_TV_bounds_misfit_perturbations_via_triangle_inequality}\ref{item_proposition_TV_bounds_misfit_perturbations_via_triangle_inequality_upper_bound}.
  Both upper bounds hold if $\essinf_\mu\Phi_i\geq 0$ for $i=1,2$.
Equality can be attained in the upper bound obtained by Lipschitz continuity if $\Phi_1-\Phi_2$ is $\mu$-a.s. constant.
For the upper bound obtained by the triangle inequality, equality is attained in the more restrictive case where $\Phi_1=\Phi_2$ $\mu$-a.s.

Similarly, for the lower bounds on $d_{\TV}(\mu_{\Phi_1},\mu_{\Phi_2})$ obtained by Lipschitz continuity and the triangle inequality in  \Cref{proposition_TV_bounds_misfit_perturbations_via_Lipschitz_continuity} and \Cref{proposition_TV_bounds_misfit_perturbations_via_triangle_inequality} respectively, equality can be attained in the lower bound in \Cref{proposition_TV_bounds_misfit_perturbations_via_Lipschitz_continuity}, but not in the lower bound in \Cref{proposition_TV_bounds_misfit_perturbations_via_triangle_inequality}.
Indeed, the lower bound in \Cref{proposition_TV_bounds_misfit_perturbations_via_Lipschitz_continuity} is zero if and only if $\Phi_1-\Phi_2$ is $\mu$-a.s. constant, which is consistent with \Cref{lemma_noninjectivity_of_dataMisfit_maps}.
On the other hand, the lower bounds in \Cref{proposition_TV_bounds_misfit_perturbations_via_triangle_inequality} can be zero even if $\Phi_1-\Phi_2$ is not $\mu$-a.s. constant: for example, if $\Phi_1>\Phi_2$ $\mu$-a.s., then $Z_{\Phi_1,\mu}<Z_{\Phi_2,\mu}$ holds by \Cref{lemma_L1norm_difference_of_unnormalised_likelihoods_minus_abs_diff_normalisation_constants_misfit_perturbation}\ref{lemma_L1norm_difference_of_unnormalised_likelihoods_minus_abs_diff_normalisation_constants_prior_perturbation_item1}, and the lower bound on $d_{\TV}(\mu_{\Phi_1},\mu_{\Phi_2})$ in this case is 0, by \Cref{proposition_TV_bounds_misfit_perturbations_via_triangle_inequality}\ref{item_proposition_TV_bounds_misfit_perturbations_via_triangle_inequality_lower_bound_no_sign_change}.

The preceding observations indicate that for bounds on the total variation metric of posteriors resulting from perturbed misfits, the triangle inequality-based approach yields bounds that are not as sharp as the bounds obtained by Lipschitz continuity.

To conclude this section on upper and lower bounds in the total variation metric for misfit perturbations, we recall the motivations for the lower bounds given in \Cref{section_contents}. 
Since the lower bound in \Cref{proposition_TV_bounds_misfit_perturbations_via_Lipschitz_continuity}\ref{proposition_TV_bounds_misfit_perturbations_via_Lipschitz_continuity_lower_bound} involves division by the evidence terms $Z_{\Phi_1,\mu}$ or $Z_{\Phi_2,\mu}$, it follows that $d_{\TV}(\mu_{\Phi_1},\mu_{\Phi_2})$ is more sensitive to the perturbation $\Phi_1-\Phi_2$ as $Z_{\Phi_1,\mu}$ and $Z_{\Phi_2,\mu}$ approach zero from above. 
  This provides additional evidence for the increasing sensitivity that was deduced using upper bounds; see \cite[Remark 9]{Sprungk2020}.
 The inequalities in \Cref{proposition_TV_bounds_misfit_perturbations_via_Lipschitz_continuity}\ref{proposition_TV_bounds_misfit_perturbations_via_Lipschitz_continuity_upper_bound} and \ref{proposition_TV_bounds_misfit_perturbations_via_Lipschitz_continuity_lower_bound} give bounds of the form \eqref{eq_prototypical_upper_bound_misfit_perturbation} and \eqref{eq_prototypical_lower_bound_misfit_perturbation} respectively.
  In particular, \Cref{proposition_TV_bounds_misfit_perturbations_via_Lipschitz_continuity} implies that if $\Phi_i\in L^\infty_\mu(\Theta,\R)$ and $Z_{\Phi_1,\mu}=Z_{\Phi_2,\mu}$, then
 \begin{equation*}
  \frac{\exp(-\Norm{\Phi_1}_{L^\infty_\mu})\wedge \exp(-\Norm{\Phi_2}_{L^\infty_\mu})}{Z_{\Phi_1,\mu}}\Norm{\Phi_1-\Phi_2}_{L^1_\mu} \leq 2 d_{\TV}(\mu_{\Phi_1},\mu_{\Phi_2})\leq \frac{1}{Z_{\Phi_1,\mu}}\Norm{\Phi_1-\Phi_2}_{L^1_\mu}.
 \end{equation*}
Thus, for a suitable fixed prior $\mu$ and level $c>0$, the misfit-to-posterior map $\Phi\mapsto \mu_\Phi$ is locally bi-Lipschitz continuous on the evidence level set $\{\Phi\in L^\infty_\mu(\Theta,\R)\ :\ Z_{\Phi,\mu}=c\}$.

\subsection{Perturbations of the prior}
\label{sectionTVbounds_prior_perturbation}

The proofs of the results in this section are given in \Cref{section_TVbounds_Priorperturbations_proofs}.

To prove bounds on the total variation metric between the posteriors, we shall use the triangle inequality, in an approach similar to that used to prove \Cref{proposition_TV_bounds_misfit_perturbations_via_triangle_inequality}.
The upper bound in  \Cref{proposition_TV_bounds_prior_perturbation_via_triangle_inequality}\ref{item_proposition_TV_bounds_prior_perturbation_via_triangle_inequality_upper_bound} implies the upper bound stated in \cite[Theorem 8]{Sprungk2020} on the total variation metric between the posteriors.
The lower bound is new, to the best of our knowledge.
While the upper bound is invariant under permutation of $\mu_1$ and $\mu_2$, the lower bound is in general not invariant under this permutation.
\begin{restatable}{proposition}{TVboundsPriorperturbationTriangle}
 \label{proposition_TV_bounds_prior_perturbation_via_triangle_inequality}
 Let $\mu_i\in\mathcal{P}(\Theta)$ and let $\Phi\in L^0(\Theta,\R)$ be such that $(\mu_i)_\Phi\in\mathcal{P}(\Theta)$ for $i=1,2$. 
 \begin{enumerate}
  \item \label{item_proposition_TV_bounds_prior_perturbation_via_triangle_inequality_upper_bound} If $\Phi\in L^0(\Theta,\R_{\geq 0})$ is such that $\essinf_{\mu_i}\Phi=0$ for some $i\in\{1,2\}$, then 
  \begin{equation*}
   d_{\TV}((\mu_1)_\Phi,(\mu_2)_\Phi)\leq \frac{1}{Z_{\Phi,\mu_1}\vee Z_{\Phi,\mu_2}}\left(d_{\TV}(\mu_1,\mu_2)
+\frac{\abs{Z_{\Phi,\mu_2}-Z_{\Phi,\mu_1}}}{2} \right),
  \end{equation*}
   and $\abs{Z_{\Phi,\mu_2}-Z_{\Phi,\mu_1}}\leq 2d_{\TV}(\mu_1,\mu_2)$.
  \item \label{item_proposition_TV_bounds_prior_perturbation_via_triangle_inequality_lower_bound}
 If $\Phi\in L^\infty_{\mu_i}(\Theta,\R)$ for $i=1,2$, then
 \begin{equation*}
  d_{\TV}((\mu_1)_\Phi,(\mu_2)_\Phi)\geq  \frac{\exp(-\norm{\Phi}_{L^\infty_{\mu_1}})\wedge \exp(-\norm{\Phi}_{L^\infty_{\mu_2}})}{Z_{\Phi,\mu_1}} \Abs{d_{\TV}(\mu_1,\mu_2)- \Abs{ \frac{Z_{\Phi,\mu_2}-Z_{\Phi,\mu_1}}{2Z_{\Phi,\mu_2}} }}.
 \end{equation*}  
 \end{enumerate}
 For both the upper bound and lower bound, equality holds if and only if $d_{\TV}(\mu_1,\mu_2)=0$.
\end{restatable}
\begin{remark}
\label{remark_TV_bounds_noninjectivity_prior_based_maps}
\Cref{item_proposition_TV_bounds_prior_perturbation_via_triangle_inequality_upper_bound} states that $\abs{Z_{\Phi,\mu_2}-Z_{\Phi,\mu_1}}\leq 2d_{\TV}(\mu_1,\mu_2)$. In general, one cannot bound $\abs{Z_{\Phi,\mu_2}-Z_{\Phi,\mu_1}}$ from below by $d_{\TV}(\mu_1,\mu_2)$. If such a bound were possible, then it would imply that $\mu\mapsto Z_{\Phi,\mu}$ is injective, which contradicts \Cref{lemma_noninjectivity_of_prior_maps}\ref{item_lemma_noninjectivity_of_prior_maps_noninjective} in the case where the domain of this map is taken to be $D(\nu)$ for some common dominating measure $\nu$.
\Cref{item_proposition_TV_bounds_prior_perturbation_via_triangle_inequality_lower_bound} implies that if $\Phi\in L^\infty_{\mu_i}(\Theta,\R)$, then $(\mu_1)_\Phi=(\mu_2)_\Phi$ implies $d_{\TV}(\mu_1,\mu_2)=\Abs{ \tfrac{Z_{\Phi,\mu_2}-Z_{\Phi,\mu_1}}{2Z_{\Phi,\mu_2}}}$. The statement $d_{\TV}(\mu_1,\mu_2)=\Abs{ \tfrac{Z_{\Phi,\mu_2}-Z_{\Phi,\mu_1}}{2Z_{\Phi,\mu_2}}}$ is weaker than the statement $\mu_1=\mu_2$, which was shown in \Cref{lemma_noninjectivity_of_prior_maps}\ref{item_lemma_noninjectivity_of_prior_maps_equivalent_condition_for_agreement} to be equivalent to $(\mu_1)_\Phi=(\mu_2)_\Phi$.
\end{remark}

It is possible to obtain bounds on $d_{\TV}((\mu_1)_\Phi,(\mu_2)_\Phi)$ using the local Lipschitz continuity of the exponential function.
\begin{restatable}{proposition}{TVboundsPriorperturbationLipschitz}
Let $\Phi\in L^0(\Theta,\R)$ and $\mu_i\in\mathcal{P}(\Theta)$, $i=1,2$, be such that $(\mu_i)_\Phi\in\mathcal{P}(\Theta)$ and $c\leq \tfrac{\rd\mu_2}{\rd\mu_1}\leq C$, $\mu_1$-a.s. for some $0<c<C<\infty$.
 \label{proposition_TV_bounds_prior_perturbations_via_Lipschitz_continuity}
 \begin{enumerate}
  \item \label{proposition_TV_bounds_prior_perturbations_via_Lipschitz_continuity_upper_bound}
   If in addition $\essinf_{\mu_1}\Phi=0$, then
  \begin{equation*}
  d_{\TV}((\mu_1)_\Phi,(\mu_2)_\Phi)\leq \frac{1}{2}\frac{1}{Z_{\Phi,\mu_1}\wedge Z_{\Phi,\mu_2}}\Norm{\frac{\rd\mu_2}{\rd \mu_1}}_{L^\infty_{\mu_1}} \Norm{\log \frac{Z_{\Phi,\mu_2}}{Z_{\Phi,\mu_1}}-\log\frac{\rd \mu_2}{\rd \mu_1} }_{L^1_{\mu_1}},
  \end{equation*}
  and $\Abs{\log \tfrac{Z_{\Phi,\mu_2}}{Z_{\Phi,\mu_1}}}\leq 2(Z_{\Phi,\mu_1}\wedge Z_{\Phi,\mu_2})^{-1}d_{\TV}(\mu_1,\mu_2)$.
  \item \label{proposition_TV_bounds_prior_perturbations_via_Lipschitz_continuity_lower_bound}
  If in addition $\Phi\in L^\infty_{\mu_1}$, then
  \begin{equation*}
   d_{\TV}((\mu_1)_\Phi,(\mu_2)_\Phi)\geq\frac{1}{2}\frac{\exp(-\norm{\Phi}_{L^\infty_{\mu_1}}) }{Z_{\Phi,\mu_1}\vee Z_{\Phi,\mu_2}}\Norm{\left(\frac{\rd\mu_2}{\rd \mu_1}\right)^{-1}}_{L^\infty_{\mu_1}}\Norm{\log \frac{Z_{\Phi,\mu_2}}{Z_{\Phi,\mu_1}}-\log\frac{\rd \mu_2}{\rd \mu_1} }_{L^1_{\mu_1}}.
  \end{equation*}
 \end{enumerate}
\end{restatable}

\paragraph{Comparison of approaches} The bounds obtained by the Lipschitz continuity approach in  \Cref{proposition_TV_bounds_prior_perturbations_via_Lipschitz_continuity} hold under the hypotheses that the Radon--Nikodym derivative $\tfrac{\rd\mu_2}{\rd\mu_1}$ is $\mu_1$-a.s. bounded away from zero and bounded from above, which implies that $\mu_1\sim\mu_2$.  
For the bounds in \Cref{proposition_TV_bounds_prior_perturbation_via_triangle_inequality} that are obtained by the triangle inequality approach, neither of these hypotheses are required.
Furthermore, the bounds in \Cref{proposition_TV_bounds_prior_perturbations_via_Lipschitz_continuity} that are obtained by the Lipschitz continuity approach for prior perturbation do not involve $d_{\TV}(\mu_1,\mu_2)$, and hence are not of the form \eqref{eq_prototypical_lower_bound_prior_perturbation} and \eqref{eq_prototypical_upper_bound_prior_perturbation}.

We conclude this section on upper and lower bounds in the total variation metric for prior perturbations by recalling the motivations for the lower bounds given in \Cref{section_contents}.  
Since the lower bound in \Cref{proposition_TV_bounds_prior_perturbation_via_triangle_inequality}\ref{proposition_TV_bounds_prior_perturbations_via_Lipschitz_continuity_lower_bound} involves division by the evidence terms $Z_{\Phi,\mu_1}$ or $Z_{\Phi,\mu_2}$, it follows that $d_{\TV}((\mu_1)_{\Phi},(\mu_2)_{\Phi})$ is more sensitive to the perturbation quantified by $d_{\TV}(\mu_1,\mu_2)$ as $Z_{\Phi,\mu_1}$ or $Z_{\Phi,\mu_2}$ approach zero from above. 
  This provides additional evidence for the increasing sensitivity that was deduced using upper bounds; see \cite[Remark 9]{Sprungk2020}.
 \Cref{proposition_TV_bounds_prior_perturbation_via_triangle_inequality} also gives bounds of the form \eqref{eq_prototypical_upper_bound_prior_perturbation} and \eqref{eq_prototypical_lower_bound_prior_perturbation} respectively.
  In particular, these bounds imply that if $\Phi\in L^\infty_{\mu_i}(\Theta,\R)$ for $i=1,2$ and $Z_{\Phi,\mu_1}=Z_{\Phi,\mu_2}$, then
 \begin{equation*}
  \frac{\exp(-\Norm{\Phi}_{L^\infty_{\mu_1}})\wedge \exp(-\Norm{\Phi}_{L^\infty_{\mu_2}})}{Z_{\Phi,\mu_1}}d_{\TV}(\mu_1,\mu_2) \leq d_{\TV}((\mu_1)_{\Phi},(\mu_2)_{\Phi})\leq \frac{1}{Z_{\Phi,\mu_1}}d_{\TV}(\mu_1,\mu_2).
 \end{equation*}
Thus, for a suitable fixed misfit $\Phi$ and level $c>0$, the prior-to-posterior map $\mu\mapsto \mu_\Phi$ is locally bi-Lipschitz continuous on the evidence level set $\{\mu\in\mathcal{P}(\Theta)\ :\ Z_{\Phi,\mu}=c\}$.

\section{Bounds in the Hellinger metric}
\label{sectionHellingerbounds}

Let $\nu,\mu_i\in\mathcal{P}(\Theta)$ be such that $\mu_i\ll\nu$ for $i=1,2$.
The Hellinger metric is defined by
\begin{equation}
 d_\Hel(\mu_1,\mu_2)\coloneqq  \left( \int \left(\sqrt{\frac{\rd\mu_1}{\rd\nu}}-\sqrt{\frac{\rd\mu_2}{\rd\nu}}\right)^2\rd \nu\right)^{1/2}.
 \label{eq_Hellinger_metric}
\end{equation}
 The value of $d_\Hel(\mu_1,\mu_2)$ does not depend on the choice of $\nu$, and $d_{\Hel}(\mu_1,\mu_2)$ as defined above takes values in $[0,\sqrt{2}]$.
 The maximal value of $\sqrt{2}$ is attained if and only if the densities $\tfrac{\rd\mu_1}{\rd\nu}$ and $\tfrac{\rd\mu_2}{\rd\nu}$ have disjoint regions of positivity, i.e. $\mu_1$ and $\mu_2$ are singular; see e.g. \cite[Section B.1]{GhosalvanderVaart2017}.
 By the inequalities 
 \begin{equation*}
  d_{\Hel}^2(\mu_1,\mu_2)\leq d_{\TV}(\mu_1,\mu_2)\leq 2d_{\Hel}(\mu_1,\mu_2),
 \end{equation*}
see e.g. \cite[Lemma B.1]{GhosalvanderVaart2017}, it follows that values of $d_{\Hel}(\mu_1,\mu_2)$ between 0 and $\sqrt{2}$ do not in general imply that either $\mu_1\ll\mu_2$ or $\mu_2\ll\mu_1$.
For example, if $d_{\TV}(\mu_1,\mu_2)<1$, then by the inequality $d_{\Hel}^2(\mu_1,\mu_2)\leq d_{\TV}(\mu_1,\mu_2)$ above, $d_{\Hel}(\mu_1,\mu_2)<1$ as well.
However, by \Cref{example_translated_uniform_measure}, $d_{\TV}(\mu_1,\mu_2)<1$ does not imply either that $\mu_1\ll\mu_2$ or $\mu_2\ll \mu_1$.

\subsection{Perturbations of the misfit}
\label{sectionHellingerbounds_misfit_perturbation}

The proofs of the statements in this section are given in \Cref{section_Hellingerbounds_Misfitperturbations_proofs}.

We first state upper and lower bounds on the Hellinger metric between posteriors resulting from different misfits using the Lipschitz continuity approach. 
\begin{restatable}{proposition}{HellingerboundsMisfitperturbationsLipschitz}
 \label{proposition_Hellinger_bounds_misfit_perturbations_via_Lipschitz_continuity}
 Let $\mu\in\mathcal{P}(\Theta)$ and $\Phi_i\in L^2_\mu(\Theta,\R)$ be such that $\mu_{\Phi_i}\in\mathcal{P}(\Theta)$ for $i=1,2$. 
 \begin{enumerate}
  \item \label{proposition_Hellinger_bounds_misfit_perturbations_via_Lipschitz_continuity_upper_bound}
  If in addition $\essinf_\mu\Phi_i\geq 0$ for $i=1,2$, then
  \begin{equation*}
  d_{\Hel}(\mu_{\Phi_1},\mu_{\Phi_2})\leq \frac{1}{2}\frac{1}{Z_{\Phi_1,\mu}^{1/2} \wedge Z_{\Phi_2,\mu}^{1/2}}\Norm{\Phi_1-\Phi_2+\log \frac{Z_{\Phi_2,\mu}}{Z_{\Phi_1,\mu}}}_{L^2_\mu}.
 \end{equation*}
 
  \item 
  \label{proposition_Hellinger_bounds_misfit_perturbations_via_Lipschitz_continuity_lower_bound}
  If in addition $\Phi_i\in L^\infty_\mu(\Theta,\R_{\geq 0})$ for $i=1,2$, then
  \begin{equation*}
  d_{\Hel}(\mu_{\Phi_1},\mu_{\Phi_2})\geq \frac{1}{2}\left(\frac{\exp(-\tfrac{1}{2}\norm{\Phi_1}_{L^\infty_\mu})}{Z_{\Phi_1,\mu}^{1/2}}\wedge \frac{\exp(-\tfrac{1}{2}\norm{\Phi_2}_{L^\infty_\mu})}{Z_{\Phi_2,\mu}^{1/2}}\right)\Norm{\Phi_1-\Phi_2+\log \frac{Z_{\Phi_2,\mu}}{Z_{\Phi_1,\mu}}}_{L^2_\mu}.
 \end{equation*}
 \end{enumerate}
 For both the upper and lower bounds, equality holds if and only if $\Phi_1-\Phi_2$ is $\mu$-a.s. constant.
\end{restatable}
We also obtain upper bounds using the triangle inequality.
\begin{restatable}{proposition}{HellingerboundsMisfitperturbationsTriangle}
 \label{proposition_Hellinger_upper_bound_misfit_perturbations_via_triangle_inequality}
 Let $\mu\in\mathcal{P}(\Theta)$ and $\Phi_i\in L^2_\mu(\Theta,\R)$, $i=1,2$. 
 Then
 \begin{equation*}
d_{\Hel}(\mu_{\Phi_1},\mu_{\Phi_2})\leq \frac{1}{Z_{\Phi_1,\mu}^{1/2}\vee Z_{\Phi_2,\mu}^{1/2}}\left( \Norm{e^{-\Phi_1/2}-e^{-\Phi_2/2}}_{L^2_\mu}+\Abs{Z_{\Phi_2,\mu}^{1/2}-Z_{\Phi_1,\mu}^{1/2}}\right),
 \end{equation*}
 where equality holds if and only if $\Phi_1=\Phi_2$ $\mu$-a.s.
  It also holds that $ \abs{Z_{\Phi_2,\mu}^{1/2}-Z_{\Phi_1,\mu}^{1/2}}\leq  \Norm{e^{-\Phi_1/2}-e^{-\Phi_2/2}}_{L^2_\mu}$, where equality holds if and only if $\Phi_1-\Phi_2$ is $\mu$-a.s. constant.
 In particular, if $\Phi_i\in L^2_\mu(\Theta,\R)$ and $\essinf_\mu\Phi_i\geq 0$ for $i=1,2$, then
 \begin{equation*}
d_{\Hel}(\mu_{\Phi_1},\mu_{\Phi_2})\leq\frac{2\norm{e^{-\Phi_1/2}-e^{-\Phi_2/2}}_{L^2_\mu}}{Z_{\Phi_1,\mu}^{1/2}\vee Z_{\Phi_2,\mu}^{1/2}}\leq  \frac{(2\Norm{\Phi_1-\Phi_2}_{L^1_\mu}^{1/2})\wedge \Norm{\Phi_1-\Phi_2}_{L^2_\mu}}{Z_{\Phi_1,\mu}^{1/2}\vee Z_{\Phi_2,\mu}^{1/2}}.
 \end{equation*}
\end{restatable}
If $\essinf_\mu\Phi_i\geq 0$ for $i=1,2$, then $Z_{\Phi_i,\mu}\leq 1$, and thus the second upper bound on $d_{\Hel}(\mu_{\Phi_1},\mu_{\Phi_2})$ in  \Cref{proposition_Hellinger_upper_bound_misfit_perturbations_via_triangle_inequality} above implies the upper bound on $d_{\Hel}(\mu_{\Phi_1},\mu_{\Phi_2})$ in \cite[Theorem 5]{Sprungk2020}, which is
\begin{equation*}
 d_{\Hel}(\mu_{\Phi_1},\mu_{\Phi_2})\leq \frac{1}{Z_{\Phi_1,\mu}\wedge Z_{\Phi_2,\mu}}\Norm{\Phi_1-\Phi_2}_{L^2_\mu}.
\end{equation*}
We shall not use the triangle inequality approach to prove lower bounds on $d_{\Hel}(\mu_{\Phi_1},\mu_{\Phi_2})$ because these lower bounds require additional assumptions on $\Phi_1-\Phi_2$ and are not sharp; see the comparison of these approaches at the end of \Cref{sectionTVbounds_misfit_perturbation}.

\subsection{Perturbations of the prior}
\label{sectionHellingerbounds_prior_perturbation}

The proofs of the statements in this section are given in \Cref{section_Hellingerbounds_Priorperturbations_proofs}.

We use the triangle inequality approach to obtain upper and lower bounds for the Hellinger metric between posteriors resulting from different priors.
We do not use the Lipschitz continuity approach, because this approach requires stronger hypotheses on $\mu_1$ and $\mu_2$, e.g. that $\log \tfrac{\rd\mu_1}{\rd\mu_2}$ be bounded; see the comparison of approaches at the end of \Cref{sectionTVbounds_prior_perturbation}. 
\begin{restatable}{proposition}{HellingerboundsPriorperturbationsTriangle}
 \label{proposition_Hellinger_bounds_prior_perturbation_via_triangle_inequality}
 Let $\Phi\in L^0(\Theta,\R)$ and $\mu_i\in\mathcal{P}(\Theta)$ be such that $(\mu_i)_\Phi\in \mathcal{P}(\Theta)$ for $i=1,2$.
 \begin{enumerate}
\item \label{proposition_Hellinger_bound_prior_perturbation_upper_bound}

If in addition $\essinf_{\mu_i}\Phi=0$ for $i=1,2$, then
\begin{equation*}
 d_{\Hel}((\mu_1)_\Phi,(\mu_2)_\Phi)\leq 
\frac{d_{\Hel}(\mu_1,\mu_2)+\abs{Z_{\Phi,\mu_2}^{1/2}-Z_{\Phi,\mu_1}^{1/2}}}{Z_{\Phi,\mu_1}^{1/2}\vee Z_{\Phi,\mu_2}^{1/2}}.
\end{equation*}
In addition, $\abs{Z_{\Phi,\mu_2}^{1/2}-Z_{\Phi,\mu_1}^{1/2}}\leq d_{\Hel}(\mu_1,\mu_2)$.
For both inequalities, equality holds if and only if $\mu_1=\mu_2$. 

\item \label{proposition_Hellinger_bound_prior_perturbation_lower_bound}
If in addition $\Phi\in L^\infty_{\mu_i}(\Theta,\R_{\geq 0})$ for $i=1,2$, then
\begin{equation*}
 d_{\Hel}((\mu_1)_\Phi,(\mu_2)_\Phi)\geq \frac{\exp(-\tfrac{1}{2}\norm{\Phi}_{L^\infty_{\mu_1}})\wedge \exp(-\tfrac{1}{2}\norm{\Phi}_{L^\infty_{\mu_2}})}{Z_{\Phi,\mu_1}^{1/2}}\Abs{d_{\Hel}(\mu_1,\mu_2)-\Abs{\frac{Z_{\Phi,\mu_2}^{1/2}-Z_{\Phi,\mu_1}^{1/2}}{Z_{\Phi,\mu_2}^{1/2} }}}.
\end{equation*}
Equality holds if and only if $\mu_1=\mu_2$.
\end{enumerate}
\end{restatable}

In \cite[Theorem 7]{Sprungk2020}, the following inequalities are obtained under the hypotheses that $\Phi\in L^0(\Theta,\R_{\geq 0})$:
\begin{equation*}
 d_{\Hel}((\mu_1)_\Phi,(\mu_2)_\Phi)\leq \frac{2}{Z_{\Phi,\mu_1}\wedge Z_{\Phi,\mu_2}}d_{\Hel}(\mu_1,\mu_2),\quad \Abs{ Z_{\Phi,\mu_1}- Z_{\Phi,\mu_2}}\leq 2d_{\Hel}(\mu_1,\mu_2).
\end{equation*}
If $\Phi$ is nonnegative, then by \eqref{eq_normalisationConstant_function} $Z_{\Phi,\mu_i}\leq 1$ for $i=1,2$.
This implies that $(Z_{\Phi,\mu_1}^{1/2}\vee Z_{\Phi,\mu_2}^{1/2})^{-1}\leq (Z_{\Phi,\mu_1}\wedge Z_{\Phi,\mu_2})^{-1}$ and that $Z_{\Phi,\mu_1}^{1/2}+Z_{\Phi,\mu_2}^{1/2}\leq 2$.
The latter inequality implies that 
\begin{equation*}
 \Abs{ Z_{\Phi,\mu_1}- Z_{\Phi,\mu_2}}=\Abs{ Z_{\Phi,\mu_1}^{1/2}- Z_{\Phi,\mu_2}^{1/2}}\Abs{Z_{\Phi,\mu_1}^{1/2}+Z_{\Phi,\mu_2}^{1/2}}\leq 2 \Abs{ Z_{\Phi,\mu_1}^{1/2}- Z_{\Phi,\mu_2}^{1/2}}.
\end{equation*}
Thus, the first and second inequalities in \Cref{proposition_Hellinger_bounds_prior_perturbation_via_triangle_inequality}\ref{proposition_Hellinger_bound_prior_perturbation_upper_bound} imply the first and second inequalities in \cite[Theorem 7]{Sprungk2020} respectively.

The upper and lower bounds in \Cref{proposition_Hellinger_bounds_misfit_perturbations_via_Lipschitz_continuity} and \Cref{proposition_Hellinger_bounds_prior_perturbation_via_triangle_inequality} show that more concentrated posteriors are more sensitive to perturbations in the misfit and the prior respectively, when the sensitivity is quantified using the Hellinger metric.
These upper and lower bounds are examples of the prototypical bounds \eqref{eq_prototypical_upper_bounds} and \eqref{eq_prototypical_lower_bounds} respectively.
If the evidence is conserved, then the misfit-to-posterior and prior-to-posterior maps are locally bi-Lipschitz.
By \Cref{proposition_Hellinger_bounds_misfit_perturbations_via_Lipschitz_continuity}, if $\Phi_i\in L^\infty_\mu(\Theta,\R)$ for $i=1,2$ and $Z_{\Phi_1,\mu}=Z_{\Phi_2,\mu}$, then
\begin{equation*}
 \frac{\exp(-\tfrac{1}{2}\Norm{\Phi_1}_{L^\infty_\mu})\wedge \exp(-\tfrac{1}{2}\Norm{\Phi_1}_{L^\infty_\mu})}{Z_{\Phi_1,\mu}^{1/2}}\Norm{\Phi_1-\Phi_2}_{L^2_\mu} \leq 2 d_{\Hel}(\mu_{\Phi_1},\mu_{\Phi_2}) \leq \frac{1}{Z_{\Phi_1,\mu}^{1/2}}\Norm{\Phi_1-\Phi_2}_{L^2_\mu}.
\end{equation*}
Thus for an arbitrary suitable $\mu\in\mathcal{P}(\Theta)$ and $c>0$, the misfit-to-posterior map $\Phi\mapsto \mu_\Phi$ is locally bi-Lipschitz on the evidence level set $\{\Phi\in L^\infty_\mu(\Theta,\R)\ :\ Z_{\Phi,\mu}=c\}$.
Similarly, by \Cref{proposition_Hellinger_bounds_prior_perturbation_via_triangle_inequality}, if $\Phi\in L^\infty_{\mu_i}(\Theta,\R)$ for $i=1,2$ and $Z_{\Phi,\mu,1}=Z_{\Phi,\mu_2}$, then
\begin{equation*}
 \frac{\exp(-\tfrac{1}{2}\Norm{\Phi}_{L^\infty_{\mu_1}})\wedge \exp(-\tfrac{1}{2}\Norm{\Phi}_{L^\infty_{\mu_2}})}{Z_{\Phi,\mu_1}^{1/2}}d_{\Hel}(\mu_1,\mu_2) \leq d_{\Hel}((\mu_1)_{\Phi},(\mu_2)_{\Phi}) \leq \frac{d_{\Hel}(\mu_1,\mu_2)}{Z_{\Phi,\mu_1}^{1/2}}.
\end{equation*}
Thus for an arbitrary suitable $\Phi$ and for every $c>0$, the prior-to-posterior map $\mu\mapsto \mu_\Phi$ is locally bi-Lipschitz on the evidence level set $\{\mu\in \mathcal{P}(\Theta)\ :\  Z_{\Phi,\mu}=c\}$.

\section{Bounds in the Kullback--Leibler divergence}
\label{sectionKLbounds}

Let $\mu_i\in\mathcal{P}(\Theta)$, $i=1,2$.
The Kullback--Leibler divergence is defined by
\begin{equation}
  d_\KL(\mu_1\Vert\mu_2)\coloneqq 
 \begin{cases} 
 \int \log \frac{\rd\mu_1}{\rd\mu_2}\rd \mu_1, & \mu_1\ll \mu_2                                  
\\
+\infty, & \text{otherwise.}
 \end{cases}
 \label{eq_Kullback_Leibler_divergence}
\end{equation}
The Kullback--Leibler divergence does not satisfy the triangle inequality.
For this reason, we cannot combine bounds for misfit perturbations with bounds for prior perturbations to address the case where both the misfit and prior are jointly perturbed.
Instead, we directly prove bounds for the case where both the misfit and prior are jointly perturbed.
By setting $\mu_1=\mu_2$ (respectively, $\Phi_1=\Phi_2$) in the bounds below, one obtains the bounds in the cases where only the misfit is perturbed and the prior is held fixed (resp. the prior is perturbed and the misfit is held fixed).
\begin{restatable}{proposition}{KLboundsJointperturbations}
 \label{proposition_KL_bounds_misfit_and_prior_perturbations}
 Let $\mu_i\in\mathcal{P}(\Theta)$ and $\Phi_i\in L^0(\Theta,\R)$ be such that $(\mu_i)_{\Phi_i}\in\mathcal{P}(\Theta)$ for $i=1,2$.
 \begin{enumerate}
  \item \label{proposition_KL_bounds_misfit_and_prior_perturbations_upper_bounds}
  If $\mu_1\ll\mu_2$, $\Phi_i\in L^1_{\mu_1}$ for $i=1,2$, and $\essinf_{\mu_1}\Phi_1=0$, then the following upper bounds hold:
  \begin{align*}
 &d_{\KL}((\mu_1)_{\Phi_1}\Vert (\mu_2)_{\Phi_2})
\\
\leq & \frac{1}{Z_{\Phi_1,\mu_1}} \Norm{\Phi_2-\Phi_1}_{L^1_{\mu_1}} + \Abs{\log \frac{Z_{\Phi_2,\mu_2}}{Z_{\Phi_1,\mu_1}}}+\frac{1}{Z_{\Phi_1,\mu_1}} \left( d_{\KL}(\mu_1\Vert \mu_2)+\sqrt{2d_{\KL}(\mu_1\Vert \mu_2)}\right) .
\end{align*}
If in addition $\essinf_{\mu_2}\Phi_i=0$ for $i=1,2$, then
\begin{align*}
 \Abs{\log \frac{Z_{\Phi_2,\mu_2}}{Z_{\Phi_1,\mu_1}}}\leq  \frac{\norm{\Phi_1-\Phi_2}_{L^1_{\mu_2}}}{Z_{\Phi_1,\mu_2}\wedge Z_{\Phi_2,\mu_2}}+ \frac{\sqrt{2d_{\KL}(\mu_1\Vert \mu_2)}}{Z_{\Phi_1,\mu_2}\wedge Z_{\Phi_2,\mu_2}}.
\end{align*}
  \item \label{proposition_KL_bounds_misfit_and_prior_perturbations_lower_bounds}
  If $(\mu_1)_{\Phi_1}\ll (\mu_2)_{\Phi_2}$ and  $\Phi_1\in L^\infty_{\mu_1}$, then
    \begin{align*}
       &\frac{\exp(-\norm{\Phi_1}_{L^\infty_{\mu_1}})}{Z_{\Phi_1,\mu_1}} \Abs{\int  \Phi_2-\Phi_1~\rd \mu_1+\log\frac{ Z_{\Phi_2,\mu_2}}{Z_{\Phi_1,\mu_1}} +d_{\KL}(\mu_1\Vert \mu_2)} 
          \\
          \leq & \frac{\exp(-\norm{\Phi_1}_{L^\infty_{\mu_1}})}{Z_{\Phi_1,\mu_1}} \left(\int \Abs{ \Phi_2-\Phi_1}\rd \mu_1+\Abs{\log\frac{ Z_{\Phi_2,\mu_2}}{Z_{\Phi_1,\mu_1}}}+\int\Abs{\log\frac{\rd\mu_1}{\rd\mu_2}}\rd\mu_1\right)
          \\
   \leq & d_{\KL}((\mu_1)_{\Phi_1}\Vert (\mu_2)_{\Phi_2})+ \sqrt{2d_{\KL}((\mu_1)_{\Phi_1}\Vert (\mu_2)_{\Phi_2})}.
  \end{align*}
  \end{enumerate}
\end{restatable}
See \Cref{section_KLbounds_proofs} for the proof of \Cref{proposition_KL_bounds_misfit_and_prior_perturbations}.

We now compare the bounds above with the upper bounds on the Kullback--Leibler divergence between posteriors in \cite{Sprungk2020}.
Consider first the case where only the misfit is perturbed.
If $\essinf_\mu\Phi_i=0$ for $i=1,2$, then the main statement of \cite[Theorem 11]{Sprungk2020} yields 
\begin{equation*}
 d_{\KL}(\mu_{\Phi_1}\Vert \mu_{\Phi_2})\leq \frac{2}{Z_{\Phi_1,\mu}\wedge Z_{\Phi_2,\mu}}\norm{\Phi_1-\Phi_2}_{L^1_\mu},
\end{equation*}
while applying \Cref{proposition_KL_bounds_misfit_and_prior_perturbations}\ref{proposition_KL_bounds_misfit_and_prior_perturbations_upper_bounds} with $\mu_1\leftarrow \mu$ and $\mu_2\leftarrow \mu$ yields
\begin{equation*}
 d_{\KL}(\mu_{\Phi_1}\Vert \mu_{\Phi_2})\leq \frac{\norm{\Phi_2-\Phi_1}_{L^1_\mu}}{Z_{\Phi_1,\mu}}+\frac{\norm{\Phi_2-\Phi_1}_{L^1_\mu}}{Z_{\Phi_1,\mu}\wedge Z_{\Phi_2,\mu}},
\end{equation*}
which is a tighter bound. For the upper bound on the Kullback--Leibler divergence when only the prior is perturbed, \cite[Theorem 12]{Sprungk2020} yields the following inequality, under the hypotheses that $\mu_1\sim\mu_2$ and $\Phi\in L^0(\Theta,\R_{\geq 0})$:
\begin{equation*}
 d_{\KL}((\mu_1)_{\Phi}\Vert (\mu_2)_{\Phi}) \leq \frac{d_{\KL}(\mu_1\Vert \mu_2)+d_{\KL}(\mu_2\Vert \mu_1)}{Z_{\Phi,\mu_1}\wedge Z_{\Phi,\mu_2}}.
\end{equation*}
Applying \Cref{proposition_KL_bounds_misfit_and_prior_perturbations}\ref{proposition_KL_bounds_misfit_and_prior_perturbations_upper_bounds} with $\Phi_1\leftarrow \Phi$ and $\Phi_2\leftarrow \Phi$ yields
\begin{equation*}
 d_{\KL}((\mu_1)_{\Phi}\Vert (\mu_2)_{\Phi}) \leq \frac{d_{\KL}(\mu_1\Vert \mu_2)+\sqrt{2d_{\KL}(\mu_1\Vert \mu_2)}}{Z_{\Phi,\mu_1}}+\frac{\sqrt{2d_{\KL}(\mu_1\Vert \mu_2)}}{Z_{\Phi,\mu_1}\wedge Z_{\Phi,\mu_2}},
\end{equation*}
under a weaker hypothesis on $\mu_1$ and $\mu_2$, namely that only $\mu_1\ll \mu_2$.
The preceding upper bound is tighter than the upper bound from \cite[Theorem 12]{Sprungk2020} if $2\sqrt{2d_{\KL}(\mu_1\Vert \mu_2)}\leq d_{\KL}(\mu_2\Vert \mu_1)$. 
This holds trivially if $\mu_1\ll\mu_2$ holds but $\mu_2\ll\mu_1$ does not hold.
The main advantage of the upper bound in  \Cref{proposition_KL_bounds_misfit_and_prior_perturbations}\ref{proposition_KL_bounds_misfit_and_prior_perturbations_upper_bounds} is that it holds under the hypothesis $\mu_1\ll\mu_2$, which is weaker than the hypothesis $\mu_1\sim\mu_2$ in \cite[Theorem 12]{Sprungk2020}.

The upper and lower bounds in \Cref{proposition_KL_bounds_misfit_and_prior_perturbations} yield examples of the prototypical upper and lower bounds from \eqref{eq_prototypical_upper_bounds} and \eqref{eq_prototypical_lower_bounds}, depending on the value of $d_{\KL}((\mu_1)_{\Phi_1}\Vert (\mu_2)_{\Phi_2})$.
We use the inequalities
\begin{equation}
\label{eq_basic_inequalities}
 x+\sqrt{2x}\leq \begin{cases}
                  2x, & x\geq 2,
                  \\
                  2\sqrt{2x}, & 0\leq x<2
                 \end{cases}
                 \Longleftrightarrow x\geq \begin{cases}
                                            \tfrac{1}{2}(x+\sqrt{2x}), & x\geq 2,
                                            \\
                                            \tfrac{1}{8}(x+\sqrt{2x})^2, & 0\leq x<2,
                                           \end{cases}
\end{equation}
which can be verified by direct calculation.
For the misfit-to-posterior map, let $\mu_1\leftarrow\mu$, $\mu_2\leftarrow\mu$ for some $\mu\in\mathcal{P}(\Theta)$, and $\Phi_i\in L^\infty_\mu(\Theta,\R)$ for $i=1,2$; then by \Cref{proposition_KL_bounds_misfit_and_prior_perturbations} and the statement on the right of \eqref{eq_basic_inequalities},
\begin{align*}
\frac{1}{8}\left(\frac{\exp(-\norm{\Phi_1}_{L^\infty_{\mu}})}{Z_{\Phi_1,\mu}} \left(\Norm{ \Phi_2-\Phi_1}_{L^1_{\mu}}+\Abs{\log\frac{ Z_{\Phi_2,\mu}}{Z_{\Phi_1,\mu}}}\right)\right)^{\gamma}\leq& d_{\KL}(\mu_{\Phi_1}\Vert \mu_{\Phi_2})
\\
\leq & \frac{1}{Z_{\Phi_1,\mu}} \Norm{\Phi_2-\Phi_1}_{L^1_{\mu}} + \Abs{\log \frac{Z_{\Phi_2,\mu}}{Z_{\Phi_1,\mu}}} ,
\end{align*}
where $\gamma=2$ if $0\leq d_{\KL}((\mu_1)_{\Phi_1}\Vert (\mu_2)_{\Phi_2})<2$, and $\gamma=1$ otherwise.
For the prior-to-posterior map, let $\mu_1,\mu_2\in\mathcal{P}(\Theta)$, $\Phi_1\leftarrow\Phi$, $\Phi_2\leftarrow \Phi$, and $\Phi\in L^\infty_{\mu_i}(\Theta,\R)$ for $i=1,2$; then 
\begin{align*}
\frac{1}{8}\left(\frac{\exp(-\norm{\Phi}_{L^\infty_{\mu_1}})}{Z_{\Phi,\mu_1}} \Abs{\log\frac{ Z_{\Phi,\mu_2}}{Z_{\Phi,\mu_1}}+d_{\KL}(\mu_1\Vert \mu_2)}\right)^{\gamma}\leq& d_{\KL}((\mu_1)_{\Phi}\Vert (\mu_2)_{\Phi})
\\
\leq & \Abs{\log \frac{Z_{\Phi,\mu_2}}{Z_{\Phi,\mu_1}}} + 2\sqrt{2} \left(d_{\KL}(\mu_1\Vert \mu_2)\right)^\xi,
\end{align*}
where $\gamma$ is the same exponent as above, and $\xi=1$ if $d_{\KL}(\mu_1\Vert \mu_2)\geq 2$ and $\xi=1/2$ otherwise.
From these bounds, one cannot conclude that neither the misfit-to-posterior nor the prior-to-posterior maps are locally bi-Lipschitz continuous in the Kullback--Leibler topology on evidence level sets.
However, the lower bounds show that more concentrated posteriors are more sensitive to perturbations in the misfit and in the prior.

\section{Bounds in the 1-Wasserstein metric}
\label{sectionW1bounds}

Let $d$ be a metric on $\Theta$, such that $(\Theta,d)$ is a Polish space. 
Given $p\in [1,\infty)$ the corresponding $p$-Wasserstein metric on $\mathcal{P}(\Theta)$ and Wasserstein space of order $p$ are defined by 
\begin{subequations}
\begin{align}
 \Was_p(\mu,\nu)\coloneqq& \inf_{\pi\in\Pi(\mu,\nu)}\left(\int_{\Theta\times\Theta} d(x,y)^p\rd \pi( x,y) \right)^{1/p},
 \label{eq_Wasserstein_metric}
 \\
 \abs{\mu}_{\mathcal{P}_p}^{p}\coloneqq&\int d(x_0,x)^p\mu(\rd x),\quad \mathcal{P}_p(\Theta)\coloneqq \left\{\mu\in\mathcal{P}(\Theta)\ :\ \abs{\mu}_{\mathcal{P}_p}<\infty \right\}, 
\label{eq_Wasserstein_space}
\end{align} 
\end{subequations}
where $\Pi(\mu,\nu)\subset \mathcal{P}(\Theta\times \Theta)$ is the set of all couplings of $\mu$ and $\nu$, and $x_0$ in \eqref{eq_Wasserstein_space} is arbitrary. 
We emphasise that the definition of $\abs{\cdot}_{\mathcal{P}_p}$---and thus also of $\mathcal{P}_p$---does not depend on the choice of $x_0$; see e.g. \cite[Definition 6.1, Definition 6.4]{Villani2009} and \cite[Definition 4.3]{Ambrosio2021}.

In inference problems, the measures used for inference may be mutually singular, for example if the measures are empirical distributions supported on finite, disjoint subsets of the parameter space.
In this case the total variation metric, Hellinger metric, or Kullback--Leibler divergence of any two such measures will assume maximal values and thus be uninformative.
The $p$-Wasserstein metrics do not exhibit this property, and are thus advantageous in such settings. 
The first instances of 1-Wasserstein bounds on posteriors appear in \cite[Section 5]{Sprungk2020}.

Given a metric space $(X,\mathbf{d})$ and a function $f:X\to\R$, define the Lipschitz constant of $f$ and its supremum norm by 
\begin{equation}
\label{eq_Lipschitz_constant_and_supremum_norm}
 \norm{f}_{\Lip(X,\mathbf{d})}\coloneqq \sup_{x\neq y} \tfrac{\abs{f(x)-f(y)}}{\mathbf{d}(x,y)},\quad \norm{f}_{\infty,X}\coloneqq \sup_{x\in X}\abs{f(x)}.
\end{equation}
Recall that $\supp{\mu}$ denotes the support of a measure $\mu$.
For an arbitrary $\mu\in\mathcal{M}(\Theta)$, define the radius of the support of $\mu$,
 \begin{equation}
  \label{eq_radius_parameter}
  R(\mu)\coloneqq \diam{\supp{\mu}}=\sup_{x,y\in\supp{\mu}}d(x,y),
 \end{equation}
 so that $R(\mu)$ is finite if and only if $\supp{\mu}$ is bounded.
  
In this section, we focus on the 1-Wasserstein metric, and we shall use the Kantorovich--Rubinstein duality formulas applied to $\mu,\nu\in\mathcal{P}_1(\Theta)$:
\begin{subequations}
 \begin{align}
 \label{eq_Kantorovich_Rubinstein_duality_formula}
\Was_1(\mu,\nu)=&\sup_{\Norm{f}_{\Lip(\supp{\mu+\nu},d)}\leq 1}\left\{\int f \rd \mu-\int f \rd \nu\right\}
\\
 \label{eq_Kantorovich_Rubinstein_duality_formula_fixed_x0}
\Was_1(\mu,\nu)=&\sup_{\Norm{f}_{\Lip(\supp{\mu+\nu},d)}\leq 1, f(x_0)=0}\left\{\int f \rd \mu-\int f \rd \nu\right\},
\end{align}
\end{subequations}
where \eqref{eq_Kantorovich_Rubinstein_duality_formula_fixed_x0} is valid for any fixed $x_0$, and follows from \eqref{eq_Kantorovich_Rubinstein_duality_formula}: if $f(\cdot)$ satisfies $\norm{f}_{\Lip}\leq 1$ then $g(\cdot)\coloneqq f(\cdot)-f(x_0)$ satisfies $\norm{g}_{\Lip}\leq 1$ and $g(x_0)=0$.
Some auxiliary results about Wasserstein metrics are collected in \Cref{section_W1bounds_proofs}.

\subsection{Perturbations of the misfit}
\label{sectionW1bounds_misfit_perturbations}

\begin{restatable}{proposition}{WassersteinboundsMisfitperturbations}
\label{proposition_W1_bounds_misfit_perturbations}
 Let $\mu\in\mathcal{P}(\Theta)$, and $\Phi_i\in L^0(\Theta,\R)$ be such that $\mu_{\Phi_i}\in \mathcal{P}_1(\Theta)$ for $i=1,2$.
 \begin{enumerate}
\item \label{proposition_W1_bounds_misfit_perturbations_upper_bound}
If $R(\mu)$ is finite and $\essinf_\mu\Phi_i=0$ for $i=1,2$, then
 \begin{equation*}
    \Was_1(\mu_{\Phi_1},\mu_{\Phi_2})\leq \frac{R(\mu)}{Z_{\Phi_1,\mu}\vee Z_{\Phi_2,\mu}}\left(\Norm{ e^{-\Phi_1}-e^{-\Phi_2}}_{L^1_\mu}+\Abs{Z_{\Phi_1,\mu}-Z_{\Phi_2,\mu}}\right),
  \end{equation*}
   where equality holds if $\Phi_1=\Phi_2$ $\mu$-a.s., but not if $\Phi_1-\Phi_2=c$ $\mu$-a.s. for some $c\neq 0$.
   If in addition $\Phi_i\in L^1_\mu(\Theta,\R)$ for $i=1,2$, then
   \begin{equation*}
    \Was_1(\mu_{\Phi_1},\mu_{\Phi_2})\leq \frac{R(\mu)}{Z_{\Phi_1,\mu}\vee Z_{\Phi_2,\mu}}\left(\Norm{\Phi_1-\Phi_2}_{L^1_\mu}+\Abs{Z_{\Phi_1,\mu}-Z_{\Phi_2,\mu}}\right).
   \end{equation*}
  \item \label{proposition_W1_bounds_misfit_perturbations_lower_bound}
  It holds that
  \begin{equation*}
   \Was_1(\mu_{\Phi_1},\mu_{\Phi_2})=\frac{1}{Z_{\Phi_1,\mu}}\sup_{\norm{f}_{\Lip}\leq 1}\Abs{ \int f\left(  e^{-\Phi_1}-e^{-\Phi_2}\right)\rd \mu+\int f \rd \mu_{\Phi_2}(Z_{\Phi_2,\mu}-Z_{\Phi_1,\mu}) },
  \end{equation*}
  and the supremum is attained.
  If $Z_{\Phi_1,\mu}=Z_{\Phi_2,\mu}$ and $0<\norm{e^{-\Phi_1}-e^{-\Phi_2}}_{\Lip(\supp{\mu},d)}<\infty$, then
    \begin{equation*}
   \Was_1(\mu_{\Phi_1},\mu_{\Phi_2})\geq \frac{1}{Z_{\Phi_1,\mu}}\frac{\Norm{e^{-\Phi_1}-e^{-\Phi_2}}_{L^2_\mu}^{2}}{\norm{e^{-\Phi_1}-e^{-\Phi_2}}_{\Lip(\supp{\mu},d)}},
   \end{equation*}
   and if in addition $\Phi_i\in L^\infty_\mu(\Theta,\R)$ for $i=1,2$, then
   \begin{equation*}
   \Was_1(\mu_{\Phi_1},\mu_{\Phi_2})\geq  \frac{\exp(-2\norm{\Phi_1}_{L^\infty_\mu})\wedge \exp(-2\norm{\Phi_2}_{L^\infty_\mu})}{Z_{\Phi_1,\mu}}\frac{\Norm{\Phi_1-\Phi_2}_{L^2_\mu}^{2}}{\norm{e^{-\Phi_1}-e^{-\Phi_2}}_{\Lip(\supp{\mu},d)}}.
  \end{equation*}  
  \end{enumerate}
\end{restatable}
See \Cref{section_W1bounds_misfit_perturbations_proofs} for the proof of \Cref{proposition_W1_bounds_misfit_perturbations}.

We now compare the upper bound above to the upper bound on the 1-Wasserstein metric in \cite[Theorem 14]{Sprungk2020}: under the hypothesis that $\essinf_\mu\Phi_i=0$ and $\Phi_i\in L^2_\mu(\Theta,\R_{\geq 0})$ for $i=1,2$,
\begin{equation*}
 \Was_1(\mu_{\Phi_1},\mu_{\Phi_2})\leq\frac{1}{Z_{\Phi_2,\mu}} \left(\Abs{\mu_{\Phi}}_{\mathcal{P}_1}\norm{\Phi_1-\Phi_2}_{L^1_\mu}+\Abs{\mu}_{\mathcal{P}_2}\Norm{\Phi_1-\Phi_2}_{L^2_\mu}\right),
\end{equation*}
where we point out that $\abs{\cdot}_{\mathcal{P}_p}$ in \cite[p. 15]{Sprungk2020} is defined as $\inf_{x_0}\left(\int d(x,x_0)^p\mu(\rd x)\right)^{1/p}$, and hence different from the definition in \eqref{eq_Wasserstein_space}.
If $\Phi_i\in L^1_\mu(\Theta,\R_{\geq 0})$ for $i=1,2$, then the upper bound in \cref{proposition_W1_bounds_misfit_perturbations_upper_bound} implies 
  \begin{equation*}
   \Was_1(\mu_{\Phi_1},\mu_{\Phi_2})\leq 2\frac{R(\mu)}{Z_{\Phi_1,\mu}\vee Z_{\Phi_2,\mu}}\norm{\Phi_1-\Phi_2}_{L^1_\mu},
 \end{equation*}
where we used the inequality $\abs{Z_{\Phi_1,\mu}-Z_{\Phi_2,\mu}}\leq \norm{\Phi_1-\Phi_2}_{L^1_\mu}$ in \eqref{eq_Lipschitz_continuity_evidence_unnormalisedLikelihood_wrt_misfit}. 
If $R(\mu)$ is finite, then it follows from the definitions \eqref{eq_radius_parameter} and \eqref{eq_Wasserstein_space} of $R(\cdot)$ and $\abs{\cdot}_{\mathcal{P}_p}$ that $\abs{\mu}_{\mathcal{P}_p}\leq R(\mu)$ for every $p\in [1,\infty)$.
Thus, the upper bound obtained by \cref{proposition_W1_bounds_misfit_perturbations_upper_bound} has the advantage that it only requires $\Phi_1$ and $\Phi_2$ to be only integrable with respect to $\mu$, and not square-integrable with respect to $\mu$.

\begin{remark}[Priors with bounded support]
\label{remark_priors_with_bounded_support}
In the context of nonparametric Bayesian inverse problems, priors with bounded supports have been used to prove Bernstein--von Mises theorems, via a so-called `localisation approach'.
This approach involves restricting the prior to a bounded subset and then renormalising; see e.g. \cite[Section 4.1]{Monard2021} or \cite[Section 4.1.1.1]{Nickl2023}.
The hypothesis that a prior on a possibly unbounded metric space has bounded support is not stronger than the hypothesis that a prior is supported on a bounded metric space.
The latter hypothesis was used in \cite[Theorem 15]{Sprungk2020}, for example.
\end{remark}

\begin{remark}[Lipschitz continuous likelihoods]
\label{remark_Lipschitz_continuity}
In \Cref{proposition_W1_bounds_misfit_perturbations}\ref{proposition_W1_bounds_misfit_perturbations_lower_bound}, a lower bound on $\Was_1(\mu_{\Phi_1},\mu_{\Phi_2})$ is obtained under the hypothesis that the difference $e^{-\Phi_1}-e^{-\Phi_2}$ of unnormalised likelihoods is nonconstant and Lipschitz continuous.
This hypothesis holds if $e^{-\Phi_i}$ is Lipschitz on $\supp{\mu}$ for $i=1,2$ and $e^{-\Phi_1}-e^{-\Phi_2}$ is nonconstant.
By \Cref{lemma_local_Lipschitz_continuity_unnormalised_likelihoods}, $e^{-\Phi}$ is Lipschitz on $\supp{\mu}$ if $\Phi_i$ is Lipschitz on $\supp{\mu}$ and $\inf e^{\Phi}>0$, for example. 
The assumption of local Lipschitz continuity of misfits is an important part of the well-posedness theory for Bayesian inverse problems; see e.g. \cite[Assumption 2.6(iii)]{Stuart2010}. 
\end{remark}

\subsection{Perturbations of the prior}
\label{sectionW1bounds_prior_perturbations}

\begin{restatable}{proposition}{WassersteinUpperboundPriorperturbations}
\label{proposition_W1_upper_bound_prior_perturbations}
Let $\mu_i\in\mathcal{P}_1(\Theta)$ for $i=1,2$, $\Phi \in L^0(\Theta,\R)$ satisfy $e^{-\Phi}\in \Lip(\supp{\mu_1+\mu_2},d)$, and $(\mu_i)_\Phi\in\mathcal{P}_1(\Theta)$ for $i=1,2$.
\begin{enumerate}
\item \label{proposition_W1_upper_bound_prior_perturbations_normalisation_constant}
It holds that $\abs{Z_{\Phi,\mu_1}-Z_{\Phi,\mu_2}}\leq \norm{e^{-\Phi}}_{\Lip(\supp{\mu_1+\mu_2},d)} \Was_1(\mu_1,\mu_2)$. 
In addition, there exists a 1-Lipschitz function $f_\opt\in\Lip(\supp{\mu_1+\mu_2},d)$ such that 
\begin{equation*}
 \Was_1((\mu_1)_\Phi,(\mu_2)_\Phi)=\Abs{\int f_\opt \rd (\mu_1)_\Phi-\int f_\opt \rd (\mu_2)_\Phi}.
\end{equation*}

\item \label{proposition_W1_upper_bound_prior_perturbations_zero_Lipschitz_constant_case}
If $\norm{f_\opt e^{-\Phi}}_{\Lip(\supp{\mu_1+\mu_2},d)}=0$, then there exists $\lambda\in\R$ such that $f_\opt =\lambda e^{\Phi}$ on $\supp{\mu_1+\mu_2}$, and 
\begin{equation*}
 \Was_1((\mu_1)_\Phi,(\mu_2)_\Phi)=\abs{\lambda}\frac{\abs{Z_{\Phi,\mu_1}-Z_{\Phi,\mu_2}}}{Z_{\Phi,\mu_1}Z_{\Phi,\mu_2}}\leq \abs{\lambda}\frac{\norm{e^{-\Phi}}_{\Lip(\supp{\mu_1+\mu_2},d)}}{Z_{\Phi,\mu_1}Z_{\Phi,\mu_2}} \Was_1(\mu_1,\mu_2).
\end{equation*}
\item \label{proposition_W1_upper_bound_prior_perturbations_nonzero_Lipschitz_constant_case}
If $0<\norm{f_\opt e^{-\Phi}}_{\Lip(\supp{\mu_1+\mu_2},d)}<\infty$, $\essinf_{\mu_i}\Phi=0$ for $i=1,2$, and $R(\mu_1+\mu_2)<\infty$, then 
 \begin{align*}
  \Was_1((\mu_1)_\Phi,(\mu_2)_\Phi) \leq &
   \frac{1+\norm{e^{-\Phi}}_{\Lip(\supp{\mu_1+\mu_2},d)} R(\mu_1+\mu_2)}{Z_{\Phi,\mu_1}\vee Z_{\Phi,\mu_2}} \Was_1(\mu_1,\mu_2)
   \\
   &+\frac{\abs{Z_{\Phi,\mu_1}-Z_{\Phi,\mu_2}}}{Z_{\Phi,\mu_1}\vee Z_{\Phi,\mu_2}} R(\mu_1+\mu_2)
 \\
 \leq &\left(1+2\norm{e^{-\Phi}}_{\Lip(\supp{\mu_1+\mu_2},d)} R(\mu_1+\mu_2)\right) \frac{\Was_1(\mu_1,\mu_2)}{Z_{\Phi,\mu_1}\vee Z_{\Phi,\mu_2}}.
  \end{align*}
\end{enumerate}
  \end{restatable}
  See \Cref{section_W1bounds_prior_perturbations_proofs} for the proof of \Cref{proposition_W1_upper_bound_prior_perturbations}.
  
The statement in \cref{proposition_W1_upper_bound_prior_perturbations_normalisation_constant} concerning the existence of $f_\opt$ does not appear in \cite[Section 5]{Sprungk2020}.
In \cref{proposition_W1_upper_bound_prior_perturbations_zero_Lipschitz_constant_case} and \cref{proposition_W1_upper_bound_prior_perturbations_nonzero_Lipschitz_constant_case}, we consider the case where $\norm{f_\opt e^{-\Phi}}_{\Lip(\supp{\mu_1+\mu_2},d)}$ vanishes and the case where it is positive and finite respectively, in order to simplify the associated upper bounds and their proofs. 
The case where $\norm{f_\opt e^{-\Phi}}_{\Lip(\supp{\mu_1+\mu_2},d)}$ vanishes is not explicitly considered in \cite[Section 5]{Sprungk2020}.

The main statement of \cite[Theorem 15]{Sprungk2020} is that if $(\Theta,d)$ is a bounded metric space, $e^{-\Phi}:\Theta\to [0,1]$, and $\norm{e^{-\Phi}}_{\Lip(\Theta,d)}<\infty$, then for $\abs{\mu}_{\mathcal{P}_1}\coloneqq\inf_{x_0\in\Theta}\int d(x,x_0)\rd \mu(x)$,
 \begin{equation}
 \label{eq_W1_bound_from_Sprungk2020}
  \Was_1((\mu_1)_\Phi,(\mu_2)_\Phi)\leq \frac{1+\norm{e^{-\Phi}}_{\Lip(\Theta,d)}\diam{\Theta}}{Z_{\Phi,\mu_2}}\left(1+\norm{e^{-\Phi}}_{\Lip(\Theta,d)}\frac{\Abs{\mu_1}_{\mathcal{P}_1}}{Z_{\Phi,\mu_1}}\right)\Was_1(\mu_1,\mu_2).
 \end{equation}
 The hypothesis that $(\Theta,d)$ is bounded implies that $\abs{\mu}_{\mathcal{P}_1}\leq \diam{\Theta}$.
 By \Cref{corollary_attainment_of_supremum_in_Kantorovich_Duality}, the same hypothesis also implies that the supremum in the dual representation \eqref{eq_Kantorovich_Rubinstein_duality_formula} for $\Was_1((\mu_1)_\Phi,(\mu_2)_\Phi)$ is attained at some $f_\opt$.
 If the hypotheses of \cite[Theorem 15]{Sprungk2020} hold and $\norm{f_\opt e^{-\Phi}}_{\Lip(\Theta,d)}>0$, then by \Cref{proposition_W1_upper_bound_prior_perturbations}\ref{proposition_W1_upper_bound_prior_perturbations_nonzero_Lipschitz_constant_case},
 \begin{equation}
 \label{eq_W1_bound_for_comparison}
  \Was_1((\mu_1)_\Phi,(\mu_2)_\Phi)\leq \left(1+2\norm{e^{-\Phi}}_{\Lip(\Theta,d)}\diam{\Theta}\right)\frac{\Was_1(\mu_1,\mu_2)}{Z_{\Phi,\mu_1}\vee Z_{\Phi,\mu_2}}.
 \end{equation}
By \eqref{eq_normalisationConstant_function}, $\essinf_{\mu_i}\Phi=0$ implies that $Z_{\Phi,\mu_i}\leq 1$ for $i=1,2$.
Thus, when $(\Theta,d)$ is bounded, then \Cref{proposition_W1_upper_bound_prior_perturbations}\ref{proposition_W1_upper_bound_prior_perturbations_nonzero_Lipschitz_constant_case} can provide a tighter upper bound on $\Was_1((\mu_1)_\Phi,(\mu_2)_\Phi)$ than the bound in \cite[Theorem 15]{Sprungk2020}.
Below, we show a sufficient condition for which the prefactor of $\Was_1(\mu_1,\mu_2)$ in \eqref{eq_W1_bound_from_Sprungk2020} is larger than the prefactor of $\Was_1(\mu_1,\mu_2)$ in \eqref{eq_W1_bound_for_comparison}:
\begin{align*}
& & \diam{\Theta}\leq & \left(1+\norm{e^{-\Phi}}_{\Lip(\Theta,d)}\diam{\Theta}\right)\frac{\abs{\mu_1}_{\mathcal{P}_1}}{Z_{\Phi,\mu_1}}
\\
\Longrightarrow & & \norm{e^{-\Phi}}_{\Lip(\Theta,d) }\diam{\Theta}\leq &  \left(1+\norm{e^{-\Phi}}_{\Lip(\Theta,d)}\diam{\Theta}\right)\norm{e^{-\Phi}}_{\Lip(\Theta,d)}\frac{\abs{\mu_1}_{\mathcal{P}_1}}{Z_{\Phi,\mu_1}}
\\
\Longleftrightarrow & & 1+2\norm{e^{-\Phi}}_{\Lip(\Theta,d) }\diam{\Theta}\leq &  \left(1+\norm{e^{-\Phi}}_{\Lip(\Theta,d)}\diam{\Theta}\right)\left(1+\norm{e^{-\Phi}}_{\Lip(\Theta,d)}\frac{\abs{\mu_1}_{\mathcal{P}_1}}{Z_{\Phi,\mu_1}}\right)
\\
\Longrightarrow & & \frac{1+2\norm{e^{-\Phi}}_{\Lip(\Theta,d)}\diam{\Theta}}{Z_{\Phi,\mu_1}\vee Z_{\Phi,\mu_2}} \leq &  \frac{1+\norm{e^{-\Phi}}_{\Lip(\Theta,d)}\diam{\Theta}}{Z_{\Phi,\mu_2}}\left(1+\norm{e^{-\Phi}}_{\Lip(\Theta,d)}\frac{\abs{\mu_1}_{\mathcal{P}_1}}{Z_{\Phi,\mu_1}}\right).
\end{align*}
We conclude this section with a lower bound for the 1-Wasserstein metric between posteriors resulting from different priors.
\begin{restatable}{proposition}{WassersteinLowerboundPriorperturbations}
 \label{proposition_W1_lower_bound_prior_perturbations}
 Let $\mu_i\in\mathcal{P}_1(\Theta)$ and $\Phi\in L^0(\Theta,\R)$ be such that $(\mu_i)_\Phi\in\mathcal{P}_1(\Theta)$ for $i=1,2$.
 Then there exists a 1-Lipschitz function $g_\opt$ on $\Theta$ such that
 \begin{equation*}
  \Was_1(\mu_1,\mu_2)=Z_{\Phi,\mu_1} \Abs{\int g_\opt ~e^\Phi \rd (\mu_1)_\Phi-\int g_\opt ~ e^\Phi \rd (\mu_2)_\Phi+\int g_\opt \rd \mu_2\left(\frac{1}{Z_{\Phi,\mu_2}}-\frac{1}{Z_{\Phi,\mu_1}}\right)}.
 \end{equation*}
In particular, if $Z_{\Phi,\mu_1}=Z_{\Phi,\mu_2}$, then 
 \begin{equation*}
  \Was_1(\mu_1,\mu_2)\leq Z_{\Phi,\mu_1}\norm{g_\opt e^\Phi}_{\Lip(\supp{\mu_1+\mu_2},d)}\Was_1((\mu_1)_\Phi,(\mu_2)_\Phi),
 \end{equation*}
 where equality holds if $\norm{g_\opt e^\Phi}_{\Lip(\supp{\mu_1+\mu_2},d)}=0$.
\end{restatable}
We conclude this section on 1-Wasserstein bounds with some observations.
Let $\mu\in\mathcal{P}(\Theta)$ and $R(\mu)$ be finite. If $\Phi_i\in L^\infty_\mu(\Theta,\R)$ and $\essinf_\mu\Phi_i=0$ for $i=1,2$, then by \Cref{proposition_W1_bounds_misfit_perturbations}, 
\begin{align*}
 &\frac{1}{Z_{\Phi_1,\mu}} \Abs{\int f\left(e^{-\Phi_1}-e^{-\Phi_2}\right)\rd\mu+\int f \rd\mu_{\Phi_2}(Z_{\Phi_2,\mu}-Z_{\Phi_1,\mu})} 
 \\
 &\leq \Was_1(\mu_{\Phi_1},\mu_{\Phi_2})\leq \frac{R(\mu)}{Z_{\Phi_1,\mu}\vee Z_{\Phi_2,\mu}}\left(\Norm{\Phi_1-\Phi_2}_{L^1_\mu}+\Abs{Z_{\Phi_1,\mu}-Z_{\Phi_2,\mu}}\right)
\end{align*}
for any 1-Lipschitz function $f$.
We thus obtain bounds of the type \eqref{eq_prototypical_lower_bound_misfit_perturbation} and \eqref{eq_prototypical_upper_bound_misfit_perturbation}, and also show that more concentrated posteriors are increasingly sensitive in the 1-Wasserstein metric with respect to perturbations in the misfit.
Under the additional hypothesis of conservation of evidence, i.e. under the hypothesis that $Z_{\Phi_1,\mu}=Z_{\Phi_2,\mu}$, we have
\begin{align*}
 &\frac{\exp(-2\norm{\Phi_1}_{L^\infty_\mu})\wedge \exp(-2\norm{\Phi_2}_{L^\infty_\mu})}{Z_{\Phi_1,\mu}} \frac{\norm{\Phi_1-\Phi_2}_{L^2_\mu}^{2} }{\norm{e^{-\Phi_1}-e^{-\Phi_2}}_{\Lip(\supp{\mu},d)}} 
 \\
 &\leq \Was_1(\mu_{\Phi_1},\mu_{\Phi_2})\leq \frac{R(\mu)}{Z_{\Phi_1,\mu}\vee Z_{\Phi_2,\mu}}\Norm{\Phi_1-\Phi_2}_{L^1_\mu}.
\end{align*}
This shows that the bounds in \Cref{proposition_W1_bounds_misfit_perturbations} do not suffice to show local bi-Lipschitz continuity of the misfit-to-posterior map, since the lower bound above involves the ratio $\norm{\Phi_1-\Phi_2}_{L^2_\mu}^{2} /\norm{e^{-\Phi_1}-e^{-\Phi_2}}_{\Lip(\supp{\mu},d)}$.

For the prior-to-posterior map, we consider for simplicity the setting where $\mu_i,(\mu_i)_\Phi\in\mathcal{P}_1(\Theta)$, $\essinf_{\mu_i}\Phi=0$ for $i=1,2$, $R(\mu_1+\mu_2)$ is finite, and $Z_{\Phi,\mu_1}=Z_{\Phi,\mu_2}$.
By \Cref{proposition_W1_upper_bound_prior_perturbations}\ref{proposition_W1_upper_bound_prior_perturbations_nonzero_Lipschitz_constant_case} and \Cref{proposition_W1_lower_bound_prior_perturbations},
\begin{align*}
  \frac{\Was_1(\mu_1,\mu_2)}{ Z_{\Phi,\mu_1}\norm{g_\opt e^\Phi}_{\Lip(\supp{\mu_1+\mu_2},d)}} \leq& \Was_1((\mu_1)_\Phi,(\mu_2)_\Phi)
  \\
  \leq& 
   \frac{1+\norm{e^{-\Phi}}_{\Lip(\supp{\mu_1+\mu_2},d)} R(\mu_1+\mu_2)}{Z_{\Phi,\mu_1}} \Was_1(\mu_1,\mu_2).
\end{align*}
Thus, under the stated hypotheses, the prior-to-posterior map is locally bi-Lipschitz continuous, and more concentrated posteriors are more sensitive to perturbations in the prior. 

\section{Continuity in Wasserstein metrics}
\label{section_continuity_in_Wp_metrics}

Recall the definitions \eqref{eq_Wasserstein_metric} and \eqref{eq_Wasserstein_space} of the $p$-Wasserstein metric $\Was_p(\mu,\nu)$ on the $p$-Wasserstein space $\mathcal{P}_p(\Theta)$.
The proofs of the results in this section are given in \Cref{section_continuity_in_Wp_metrics_proofs}.

We now state a result regarding the continuity of the misfit-to-posterior map in the $p$-Wasserstein distance.
\begin{restatable}{lemma}{LemmaWassersteinContinuityMisfitToPosteriorMap}
\label{lemma_Wasserstein_continuity_misfit_to_posterior_map}
Let $p\in[1,\infty)$, $\mu \in \mathcal{P}_p (\Theta)$, $\Phi$ and $(\Phi_n)_n$ be such that $\essinf_\mu\Phi=0$, $\essinf_\mu\Phi_n=0$ and $\mu_{\Phi_n}\in\mathcal{P}_p(\Theta)$ for every $n\in\N$, and $\Phi_n\to \Phi$ $\mu$-a.s. Then $Z_{\Phi_n,\mu}\to Z_{\Phi,\mu}$ and $\Was_p(\mu_{\Phi_n},\mu_{\Phi})\to 0$.
\end{restatable}

Given \Cref{lemma_Wasserstein_continuity_misfit_to_posterior_map}, it is natural to ask about the continuity of the inverse of the misfit-to-posterior map.
By \Cref{lemma_noninjectivity_of_dataMisfit_maps}\ref{item_lemma_noninjectivity_of_dataMisfit_maps_noninjective}, the misfit-to-evidence map is not injective on $L^p_\mu$ for any $\mu\in\mathcal{P}(\Theta)$ if $\supp{\mu}$ is uncountable or if $\card{\supp{\mu}}\geq 3$.
Thus, the misfit-to-posterior map is not injective on $L^p_\mu$ under the same conditions.
Nevertheless, we can describe the continuity of the inverse of the misfit-to-posterior map under certain conditions.
\begin{restatable}{lemma}{LemmaWassersteinContinuityPosteriorToMisfitMap}
 \label{lemma_Wasserstein_continuity_posterior_to_misfit_map}
 Let $p\in[1,\infty)$, $\mu\in\mathcal{P}_p(\Theta)$, and $\Phi,\Phi_n\in C(\Theta,\R_{\geq 0})$ for every $n\in\N$. Then $\mu_\Phi$ and $(\mu_{\Phi_n})_n$ belong to $\mathcal{P}_p(\Theta)$. If $\Was_p(\mu_{\Phi_n},\mu_\Phi)\to 0$, then there exists a subsequence $(\ell_{\Phi_{n_k},\mu})_{k\in\N}$ of $(\ell_{\Phi_n,\mu})_{n\in\N}$ such that $(\ell_{\Phi_{n_k},\mu})_{k\in\N}$ converges to $\ell_{\Phi,\mu}$ in $L^2_\mu(\Theta,\R)$.
 If in addition $\Phi,\Phi_n\in C_b(\Theta,\R_{\geq 0})$ are uniformly bounded and $Z_{\Phi_n,\mu}=Z_{\Phi,\mu}$ for every $n\in\N$, then there exists a subsequence $(\Phi_{n_k})_{k\in\N}$ of $(\Phi_n)_{n\in\N}$ such that $(\Phi_{n_k})_{k\in\N}$ converges to $\Phi$ in $L^2_\mu(\Theta,\R)$.
\end{restatable}

Continuity of the prior-to-posterior map in the $p$-Wasserstein metric was shown in \cite[Lemma 16]{Sprungk2020}. Below we show that under certain conditions, the inverse of this map is also continuous in the $p$-Wasserstein metric.
\begin{restatable}{lemma}{LemmaWassersteinContinuityPosteriorToPriorMap}
\label{lemma_Wasserstein_continuity_posterior_to_prior_map}
Let $p\in[1,\infty)$, $\mu_n, \mu \in \mathcal{P}_p(\Theta)$ for every $n\in\N$, and $\Phi\in C_b(\Theta,\R_{\geq 0})$. If $\Was_p((\mu_n)_\Phi,\mu_{\Phi})\to 0$, then $Z_{\Phi,\mu_n}\to Z_{\Phi,\mu}$ and $\Was_p(\mu_n,\mu)\to 0$.
\end{restatable}

\section{Conclusion}
\label{sec_conclusion}

In this paper, we have revisited the bounds and ideas relating to local Lipschitz stability results for the misfit-to-posterior and prior-to-posterior maps that were presented in \cite{Sprungk2020}.
In addition, we investigated lower bounds on the metrics or divergences between the priors. The motivation for doing so is to better understand the extent to which one can consider the upper bounds to be sharp with respect to the magnitude of the perturbations, in the sense that up to prefactors, the same function of the perturbations appears in both the upper and lower bounds.
In particular, for the misfit-to-posterior map, we seek one function $\mathcal{F}_M$ such that we may set $\mathcal{U}_M\leftarrow \mathcal{F}_M$ and $\mathcal{L}_M\leftarrow\mathcal{F}_M$ in \eqref{eq_prototypical_upper_bound_misfit_perturbation} and \eqref{eq_prototypical_lower_bound_misfit_perturbation}.
Similarly, for the prior-to-posterior map, we seek one function $\mathcal{F}_P$ such that we may set $\mathcal{U}_P\leftarrow \mathcal{F}_P$ and $\mathcal{L}_P\leftarrow\mathcal{F}_P$ in \eqref{eq_prototypical_upper_bound_prior_perturbation} and \eqref{eq_prototypical_lower_bound_prior_perturbation}.

The bounds that we have obtained above show that for the misfit-to-posterior map, the goal of finding sharp upper bounds is possible when using the total variation metric and the Hellinger metric; in these cases $\mathcal{F}_M$ is an $L^1$ and $L^2$ norm of the difference of the log-likelihoods respectively. We leave the problem of determining sharp upper bounds for the Kullback--Leibler divergence and 1-Wasserstein metric for future work.

For the prior-to-posterior map, we found that if the evidence is conserved, i.e. if $Z_{\Phi,\mu_1}=Z_{\Phi,\mu_2}$, then $\mathcal{F}_P$ is given by the same metric that is used to measure the difference in the posterior, for the case of the total variation metric, the Hellinger metric, and the 1-Wasserstein metric. 
We leave the problem of determining an analogous result for the Kullback--Leibler divergence for future work.

Our results complement those of \cite{Sprungk2020} concerning the increasing sensitivity of posteriors to misfit or prior perturbations as posteriors become more concentrated.
In particular, while the latter show that posteriors can become more sensitive with respect to misfit or prior perturbations as posteriors become more concentrated, our results---specifically, our lower bounds---prove that posteriors in fact must become more sensitive to misfit or prior perturbations as they become more concentrated. 

\section*{Acknowledgements}

The research of NC has received funding from the ERC under the European Union’s Horizon 2020 Research and Innovation Programme --- Grant Agreement ID \href{https://cordis.europa.eu/project/id/818473}{818473}. The research of HCL has been partially funded by the Deutsche Forschungsgemeinschaft (DFG) --- Project-ID \href{https://gepris.dfg.de/gepris/projekt/318763901}{318763901} --- SFB1294.

\bibliographystyle{amsplain}
\bibliography{biblio}

\appendix

\section{Auxiliary results}
\label{section_auxiliary_results}

\begin{lemma}
 \label{lemma_esssup_essinf_different_measures_different_functions}
\begin{enumerate}
 \item \label{item_esssup_essinf_different_measures}
Let $\mu,\nu\in\mathcal{P}(\Theta)$ with $\mu\ll \nu$. Then for every $f\in L^0(\Theta,\R)$, $\esssup_{\mu}f \leq \esssup_\nu f$ and $\essinf_{\mu} f\geq \essinf_\nu f$.
\item \label{item_esssup_essinf_different_functions}
Let $\Phi,\Psi\in L^0(\Theta,\R)$. Then $\esssup_\mu(\Phi+\Psi)\leq \esssup_\mu\Phi+\esssup_\mu \Psi$ and $\essinf_\mu(\Phi+\Psi)\geq \essinf_\mu \Phi+\essinf_\mu\Psi$.
\end{enumerate}
\end{lemma}
\begin{proof}[Proof of \Cref{lemma_esssup_essinf_different_measures_different_functions}]
 Proof of \cref{item_esssup_essinf_different_measures}: let $f\in L^0(\Theta,\R)$ and $\nu\in\mathcal{P}(\Theta)$. Then
 \begin{equation*}
  \esssup_\nu f\coloneqq \inf\{M\in\R\ :\ \nu(f>M)=0 \},\quad \essinf_\nu f\coloneqq \sup\{M\in\R\ :\ \nu(f<M)=0 \}.
 \end{equation*}
 In particular, $\nu(f>\esssup_\nu f)=0$ and $\nu(f<\essinf_\nu f)=0$.
Thus, if $\mu\ll \nu$, then $\mu(f>\esssup_\nu f)=0$ and $\mu(f<\essinf_\nu f)=0$. 
By definition of the infimum and supremum, $\mu(f>\esssup_\nu f)=0$ and $\mu(f<\essinf_\nu f)=0$ imply $\esssup_{\mu}f \leq \esssup_\nu f$ and $\essinf_{\mu} f\geq \essinf_\nu f$ respectively. 
This proves \Cref{item_esssup_essinf_different_measures}.

Proof of \cref{item_esssup_essinf_different_functions}: by definition, $\esssup_\mu (\Phi+\Psi)=\inf\{M\in\R\ :\ \mu(\Phi+\Psi>M)=0\}$ and $\mu(\Phi+\Psi>\esssup_\mu(\Phi+\Psi))=0$. 
Since $\mu(\Phi>\esssup_\mu\Phi)=0$ and $\mu(\Psi>\esssup_\mu\Psi)=0$, it follows that $\mu(\Phi+\Psi>\esssup_\mu\Phi+\esssup_\mu\Psi)=0$. Thus, by definition of the infimum, $\esssup_\mu(\Phi+\Psi)\leq \esssup_\mu\Phi+\esssup_\mu \Psi$. A similar argument shows that $\essinf_\mu(\Phi+\Psi)\geq \essinf_\mu \Phi+\essinf_\mu\Psi$.
\end{proof}

Recall that $(\Theta,\Sigma)$ denotes the measurable space of admissible parameters.
\begin{lemma}[Pointwise bounds on roots of likelihoods by local Lipschitz continuity]
\label{lemma_preliminary_bounds_on_pth_root_likelihoods}
Let $p>0$, $\mu\in\mathcal{P}(\Theta)$, and $A\in\Sigma$ satisfy $\mu(A)>0$.
\begin{enumerate}
 \item \label{item_lemma_preliminary_bounds_on_pth_root_likelihoods_upper_bound}
 If $\Phi_i\in L^0(\Theta,\R)$ and $\essinf_\mu \Phi_i \mathbb{I}_A>-\infty$ for $i=1,2$, then for $\mu$-a.e. $\theta\in A$,
 \begin{align*}
   &\abs{\ell_{\Phi_1,\mu}^{1/p}-\ell_{\Phi_2,\mu}^{1/p}}(\theta)
   \\
   \leq&  \left(\frac{\exp(-\tfrac{1}{p}\essinf_\mu\Phi_1\mathbb{I}_A)}{Z_{\Phi_1,\mu}^{1/p}}\vee \frac{\exp(-\tfrac{1}{p}\essinf_\mu\Phi_2\mathbb{I}_A)}{Z_{\Phi_2,\mu}^{1/p}}\right)\frac{\Abs{\log\ell_{\Phi_1,\mu}-\log\ell_{\Phi_2,\mu}}}{p}(\theta).
 \end{align*}

\item \label{item_lemma_preliminary_bounds_on_pth_root_likelihoods_lower_bound}
If $\Phi_i\in L^0(\Theta,\R)$ satisfies $\esssup_\mu \Phi_i \mathbb{I}_A<\infty$ for $i=1,2$, then for $\mu$-a.e. $\theta\in A$,
\begin{align*}
  &\abs{\ell_{\Phi_1,\mu}^{1/p}-\ell_{\Phi_2,\mu}^{1/p}}(\theta)
  \\
  \geq &\left( \frac{\exp(-\tfrac{1}{p}\esssup_\mu\Phi_1\mathbb{I}_A)}{Z_{\Phi_1,\mu}^{1/p}}\wedge \frac{\exp(-\tfrac{1}{p}\esssup_\mu\Phi_2\mathbb{I}_A)}{Z_{\Phi_2,\mu}^{1/p}}\right)\frac{\Abs{\log\ell_{\Phi_1,\mu}-\log\ell_{\Phi_2,\mu}}}{p}(\theta).
 \end{align*} 
\end{enumerate}
 \end{lemma}
\begin{proof}
\Cref{item_lemma_preliminary_bounds_on_pth_root_likelihoods_upper_bound}: 
By the upper bound in the local Lipschitz continuity statement \eqref{eq_Lipschitz_continuity_exp_function}, 
\begin{equation*}
 \Abs{ \exp\biggr(\frac{\log\ell_{\Phi_1,\mu}}{p}\biggr) -  \exp\biggr(\frac{\log\ell_{\Phi_2,\mu}}{p}\biggr) }  \leq \left(\ell_{\Phi_1,\mu}^{1/p}\vee \ell_{\Phi_2,\mu}^{1/p}\right)  \frac{\Abs{\log\ell_{\Phi_1,\mu}-\log\ell_{\Phi_2,\mu}}}{p}.
\end{equation*}
By \eqref{eq_likelihood_function}, it holds for $\mu$-a.e. $\theta\in A$ that
\begin{equation*}
 (\ell_{\Phi_1,\mu}\vee \ell_{\Phi_2,\mu})(\theta)
 \leq \frac{\exp(-\essinf_\mu\Phi_1\mathbb{I}_A)}{Z_{\Phi_1,\mu}}\vee \frac{\exp(-\essinf_\mu\Phi_2\mathbb{I}_A)}{Z_{\Phi_2,\mu}}.
\end{equation*}
\Cref{item_lemma_preliminary_bounds_on_pth_root_likelihoods_lower_bound}: By the lower bound in the local Lipschitz continuity statement  \eqref{eq_Lipschitz_continuity_exp_function},
\begin{equation*}
\Abs{ \exp\biggr(\frac{\log\ell_{\Phi_1,\mu}}{p}\biggr) -  \exp\biggr(\frac{\log\ell_{\Phi_2,\mu}}{p}\biggr) } 
 \geq \left(\ell_{\Phi_1,\mu}^{1/p}\wedge \ell_{\Phi_2,\mu}^{1/p}\right)  \frac{\Abs{\log\ell_{\Phi_1,\mu}-\log\ell_{\Phi_2,\mu}}}{p},
\end{equation*}
By \eqref{eq_likelihood_function}, it holds for $\mu$-a.e. $\theta\in A$ that
\begin{equation*}
 (\ell_{\Phi_1,\mu}\wedge \ell_{\Phi_2,\mu})(\theta)
 \geq 
 \frac{\exp(-\esssup_\mu\Phi_1\mathbb{I}_A)}{Z_{\Phi_1,\mu}}\wedge \frac{\exp(-\esssup_\mu\Phi_2\mathbb{I}_A)}{Z_{\Phi_2,\mu}}.
\end{equation*}
\end{proof}

\begin{lemma}
 \label{lemma_L1norm_difference_of_unnormalised_likelihoods_minus_abs_diff_normalisation_constants_misfit_perturbation}
For $i=1,2$, let $\Phi_i\in L^0(\Theta,\R)$ be such that $\ell_{\Phi_i,\mu}\in L^1_\mu(\Theta,\R_{\geq 0})$ exists, for $i=1,2$.
\begin{enumerate}
\item \label{item1_lemma_L1norm_difference_of_unnormalised_likelihoods_minus_abs_diff_normalisation_constants_misfit_perturbation} If in addition $\Phi_1\geq \Phi_2$ $\mu$-a.s. (respectively, $\Phi_1>\Phi_2$ $\mu$-a.s.), then $Z_{\Phi_1,\mu}\leq Z_{\Phi_2,\mu}$ (respectively, $Z_{\Phi_1,\mu}< Z_{\Phi_2,\mu}$). 

\item \label{item2_lemma_L1norm_difference_of_unnormalised_likelihoods_minus_abs_diff_normalisation_constants_misfit_perturbation} 
The difference $\Norm{e^{-\Phi_1}-e^{-\Phi_2}}_{L^1_\mu}- \abs{Z_{\Phi_1,\mu}-Z_{\Phi_2,\mu}}$ is nonnegative. It is strictly positive if and only if
\begin{equation}
\label{eq_hypothesis_difference_of_misfits_changes_sign_mu_almost_surely}
 \mu(\Phi_1-\Phi_2>0), \mu(\Phi_1-\Phi_2<0) \in (0,1).
\end{equation}

\item \label{item3_lemma_L1norm_difference_of_unnormalised_likelihoods_minus_abs_diff_normalisation_constants_misfit_perturbation}
If $Z_{\Phi_1,\mu}\neq Z_{\Phi_2,\mu}$, then 
\begin{equation*}
\frac{\Norm{e^{-\Phi_1}-e^{-\Phi_2}}_{L^1_\mu}- \abs{Z_{\Phi_1,\mu}-Z_{\Phi_2,\mu}}}{2}=
 \begin{cases}
 \int_{\{\Phi_1>\Phi_2\}} \Abs{e^{-\Phi_1}-e^{-\Phi_2}}\rd \mu, & Z_{\Phi_1,\mu}>Z_{\Phi_2,\mu}
  \\
\int_{\{\Phi_1<\Phi_2\}} \Abs{e^{-\Phi_1}-e^{-\Phi_2}}\rd \mu, & Z_{\Phi_1,\mu}<Z_{\Phi_2,\mu}.
 \end{cases}
\end{equation*}
\end{enumerate}
\end{lemma}
\begin{proof}[Proof of \Cref{lemma_L1norm_difference_of_unnormalised_likelihoods_minus_abs_diff_normalisation_constants_misfit_perturbation}] 

\Cref{item1_lemma_L1norm_difference_of_unnormalised_likelihoods_minus_abs_diff_normalisation_constants_misfit_perturbation}: Suppose $\Phi_1\geq \Phi_2$ $\mu$-a.s. By \eqref{eq_normalisationConstant_function},
\begin{align*}
 &Z_{\Phi_1,\mu}=\int e^{-\Phi_1}\rd \mu\leq \int e^{-\Phi_2}\rd \mu=Z_{\Phi_2,\mu},
\end{align*}
If $\Phi_1>\Phi_2$ $\mu$-a.s., then $Z_{\Phi_1,\mu}<Z_{\Phi_2,\mu}$, by replacing the inequality above with a strict inequality.

\Cref{item2_lemma_L1norm_difference_of_unnormalised_likelihoods_minus_abs_diff_normalisation_constants_misfit_perturbation}: 
By Jensen's inequality applied to the absolute value function,
\begin{equation*}
 \abs{Z_{\Phi_1,\mu}-Z_{\Phi_2,\mu}}=\Abs{\int e^{-\Phi_1}-e^{-\Phi_2}\rd \mu}\leq \int \Abs{e^{-\Phi_1}-e^{-\Phi_2}}\rd \mu=\Norm{e^{-\Phi_1}-e^{-\Phi_2}}_{L^1_\mu},
\end{equation*}
and equality holds if and only if $e^{-\Phi_1}-e^{-\Phi_2}$ $\mu$-a.s. does not change sign. In turn, $e^{-\Phi_1}-e^{-\Phi_2}$ $\mu$-a.s. does not change sign if and only if either $\Phi_1\geq \Phi_2$ $\mu$-a.s. or $\Phi_2\geq \Phi_1$ $\mu$-a.s. The negation of the latter statement is \eqref{eq_hypothesis_difference_of_misfits_changes_sign_mu_almost_surely}.

\Cref{item3_lemma_L1norm_difference_of_unnormalised_likelihoods_minus_abs_diff_normalisation_constants_misfit_perturbation}: Suppose that $Z_{\Phi_1,\mu}>Z_{\Phi_2,\mu}$, so that $\Abs{Z_{\Phi_1,\mu}-Z_{\Phi_2,\mu}} =Z_{\Phi_1,\mu}-Z_{\Phi_2,\mu}$.
Since
\begin{equation*}
 \{e^{-\Phi_1}-e^{-\Phi_2}> 0\}=\{\Phi_1<\Phi_2\},\quad \{e^{-\Phi_1}-e^{-\Phi_2}< 0\}=\{\Phi_1> \Phi_2\}
\end{equation*}
and since
\begin{align*}
 \Norm{e^{-\Phi_1}-e^{-\Phi_2}}_{L^1_\mu}=&\int_{\{\Phi_1< \Phi_2\}}e^{-\Phi_1}-e^{-\Phi_2}\rd \mu- \int_{\{\Phi_1>\Phi_2\}}e^{-\Phi_1}-e^{-\Phi_2}\rd \mu
 \\
 Z_{\Phi_1,\mu}-Z_{\Phi_2,\mu}=&\int_{\{\Phi_1<\Phi_2\}}e^{-\Phi_1}-e^{-\Phi_2}\rd \mu+\int_{\{\Phi_1>\Phi_2\}}e^{-\Phi_1}-e^{-\Phi_2}\rd \mu,
\end{align*}
it follows that 
\begin{equation*}
 \Norm{e^{-\Phi_1}-e^{-\Phi_2}}_{L^1_\mu}-(Z_{\Phi_1,\mu}-Z_{\Phi_2,\mu})
=-2\int_{\{\Phi_1>\Phi_2\}}e^{-\Phi_1}-e^{-\Phi_2}\rd \mu =2\int_{\{\Phi_1>\Phi_2\}} \abs{e^{-\Phi_1}-e^{-\Phi_2}}\rd \mu.
 \end{equation*}
 If $Z_{\Phi_2,\mu}>Z_{\Phi_1,\mu}$, then switching $\Phi_1$ and $\Phi_2$ and applying the analogous argument yields the corresponding conclusion.
 \end{proof}

 \begin{lemma}
 \label{lemma_L1norm_difference_of_unnormalised_likelihoods_minus_abs_diff_normalisation_constants_prior_perturbation}
 Let $\mu_1,\mu_2\in\mathcal{P}(\Theta)$ and let $\nu\in\mathcal{P}(\Theta)$ satisfy $\mu_i\ll\nu$ for $i=1,2$. 
\begin{enumerate}
 \item \label{lemma_L1norm_difference_of_unnormalised_likelihoods_minus_abs_diff_normalisation_constants_prior_perturbation_item1}
 It holds that
 \begin{equation*}
  d_{\TV}(\mu_1,\mu_2)=\int_{\{\tfrac{\rd\mu_1}{\rd\nu}>\tfrac{\rd\mu_2}{\rd\nu}\}} \frac{\rd\mu_1}{\rd\nu}-\frac{\rd\mu_2}{\rd\nu} \rd \nu=\int_{\{\tfrac{\rd\mu_1}{\rd\nu}<\tfrac{\rd\mu_2}{\rd\nu}\}} \frac{\rd\mu_2}{\rd\nu}-\frac{\rd\mu_1}{\rd\nu} \rd \nu.
 \end{equation*}
 In particular, $d_{\TV}(\mu_1,\mu_2)>0$ if and only if $\nu(\tfrac{\rd\mu_1}{\rd\nu}>\tfrac{\rd\mu_2}{\rd\nu})>0$, and $d_{\TV}(\mu_1,\mu_2)>0$ if and only if $\nu(\tfrac{\rd\mu_1}{\rd\nu}<\tfrac{\rd\mu_2}{\rd\nu})>0$ .
 \item \label{lemma_L1norm_difference_of_unnormalised_likelihoods_minus_abs_diff_normalisation_constants_prior_perturbation_item2}
 It holds that
 \begin{equation*}
  0\leq \int \exp(-\Phi)\Abs{\frac{\rd\mu_2}{\rd\nu}-\frac{\rd\mu_1}{\rd\nu}}\rd \nu-\Abs{Z_{\Phi,\mu_2}-Z_{\Phi,\mu_1}}.
 \end{equation*}
The inequality is strict if and only if $d_{\TV}(\mu_1,\mu_2)>0$.
\end{enumerate}
 \end{lemma}
 \begin{proof}[Proof of \Cref{lemma_L1norm_difference_of_unnormalised_likelihoods_minus_abs_diff_normalisation_constants_prior_perturbation}]
 Proof of \cref{lemma_L1norm_difference_of_unnormalised_likelihoods_minus_abs_diff_normalisation_constants_prior_perturbation_item1}:
 Let $A\coloneqq \{\theta\in\Theta:\tfrac{\rd\mu_1}{\rd\nu}> \tfrac{\rd\mu_2}{\rd\nu}\}$ and $B\coloneqq \{\theta\in\Theta:\tfrac{\rd\mu_1}{\rd\nu}<\tfrac{\rd\mu_2}{\rd\nu}\}$.
 Observe that
\begin{equation*}
 d_{\TV}(\mu_1,\mu_2)=\frac{1}{2}\int_{A\cup B} \Abs{\frac{\rd\mu_1}{\rd\nu}-\frac{\rd\mu_2}{\rd\nu}  }\rd \nu=\frac{1}{2}\left(\int_A \frac{\rd\mu_1}{\rd\nu}-\frac{\rd\mu_2}{\rd\nu} \rd \nu-\int_B \frac{\rd\mu_1}{\rd\nu}-\frac{\rd\mu_2}{\rd\nu} \rd \nu\right)
\end{equation*}
and that
\begin{equation*}
 \int_A \frac{\rd\mu_1}{\rd\nu}-\frac{\rd\mu_2}{\rd\nu} \rd \nu+\int_B \frac{\rd\mu_1}{\rd\nu}-\frac{\rd\mu_2}{\rd\nu} \rd \nu=\int\frac{\rd\mu_1}{\rd\nu}-\frac{\rd\mu_2}{\rd\nu} \rd \nu=\mu_1(\Theta)-\mu_2(\Theta)=0.
\end{equation*}
Thus, $\int_A \frac{\rd\mu_1}{\rd\nu}-\frac{\rd\mu_2}{\rd\nu} \rd \nu=-\int_B \frac{\rd\mu_1}{\rd\nu}-\frac{\rd\mu_2}{\rd\nu} \rd \nu$.
Combining this equation with the equation above for $d_{\TV}(\mu_1,\mu_2)$ yields the desired conclusion.

Proof of \cref{lemma_L1norm_difference_of_unnormalised_likelihoods_minus_abs_diff_normalisation_constants_prior_perturbation_item2}: Since $Z_{\Phi,\mu_i}=\int \exp(-\Phi) \rd \mu_i$ for $i=1,2$, Jensen's inequality yields
\begin{equation*}
 \Abs{Z_{\Phi,\mu_1}-Z_{\Phi,\mu_2}}=\Abs{\int \exp(-\Phi)\left(\frac{\rd\mu_1}{\rd\nu}-\frac{\rd\mu_2}{\rd\nu}\right)\rd \nu}\leq \int \exp(-\Phi)\Abs{\frac{\rd\mu_1}{\rd\nu}-\frac{\rd\mu_2}{\rd\nu}}\rd \nu,
\end{equation*}
and the inequality is an equality if and only if $\tfrac{\rd\mu_1}{\rd\nu}-\tfrac{\rd\mu_2}{\rd\nu}$ does not change sign, $\nu$-a.s.
 \end{proof}
\begin{lemma}[Local Lipschitz continuity of unnormalised likelihoods]
 \label{lemma_local_Lipschitz_continuity_unnormalised_likelihoods}
 Let $(\Theta,d)$ be a metric space and $M\subset \Theta$ be nonempty.
 Then
 \begin{equation*}
\inf_{M}e^{\Phi}\norm{e^{-\Phi}}_{\Lip(M,d)}\leq \norm{\Phi}_{\Lip(M,d)}\leq \norm{e^{-\Phi}}_{\Lip(M,d)} \left(\sup_{M}e^{\Phi}\right).
 \end{equation*}
In particular, if $\inf_M e^{\Phi}$ and $\sup_M e^{\Phi}$ are strictly positive and finite, then $\norm{e^{-\Phi}}_{\Lip(M,d)}$ is finite if and only if $\norm{\Phi}_{\Lip(M,d)}$ is finite.
\end{lemma}
\begin{proof}
By the lower bound in \eqref{eq_Lipschitz_continuity_exp_function} with $x\leftarrow -\Phi(x)$ and $y\leftarrow -\Phi(y)$, and by the definition \eqref{eq_Lipschitz_constant_and_supremum_norm} of the Lipschitz constant, we have for $x,y\in M$ that
 \begin{equation*}
  \left(e^{-\Phi(x)}\wedge e^{-\Phi(y)}\right)\Abs{\Phi(x)-\Phi(y)}\leq \Abs{e^{-\Phi(x)}-e^{-\Phi(y)}}\leq \norm{e^{-\Phi}}_{\Lip(M,d)}d(x,y).
 \end{equation*}
Thus
\begin{equation*}
 \sup_{x,y\in M,x\neq y} \frac{\Abs{\Phi(x)-\Phi(y)}}{d(x,y)}\leq \norm{e^{-\Phi}}_{\Lip(M,d)} \sup_{x,y\in M}\left[\left(e^{-\Phi(x)}\wedge e^{-\Phi(y)}\right)^{-1}\right],
\end{equation*}
i.e. $\norm{\Phi}_{\Lip(M,d)}\leq \norm{e^{-\Phi}}_{\Lip(M,d)} \left(\sup_{M}e^{\Phi}\right)$.
By the upper bound in \eqref{eq_Lipschitz_continuity_exp_function} with $x\leftarrow -\Phi(x)$ and $y\leftarrow -\Phi(y)$, and by the definition of the Lipschitz constant, we have for $x,y\in M$ that
\begin{equation*}
  \Abs{e^{-\Phi(x)}-e^{-\Phi(y)}}\leq \left(e^{-\Phi(x)}\vee e^{-\Phi(y)}\right)\Abs{\Phi(x)-\Phi(y)}\leq \left(e^{-\Phi(x)}\vee e^{-\Phi(y)}\right)\norm{\Phi}_{\Lip(M,d)}d(x,y)
 \end{equation*}
and thus 
\begin{equation*}
\sup_{x,y\in M, x\neq y}\frac{\Abs{e^{-\Phi(x)}-e^{-\Phi(y)}}}{d(x,y)}\leq \sup_{x,y\in M} \left(e^{-\Phi(x)}\vee e^{-\Phi(y)}\right)\norm{\Phi}_{\Lip(M,d)}.
\end{equation*}
This implies $\norm{e^{-\Phi}}_{\Lip(M,d)}\leq \norm{\Phi}_{\Lip(M,d)}\left( \sup_M e^{-\Phi}\right)$. Since $\sup(f^{-1})=(\inf f)^{-1}$ for any function $f$ taking values in $\R_{\geq 0}$, the preceding inequality implies
\begin{equation*}
 \norm{e^{-\Phi}}_{\Lip(M,d)}\left( \inf_M e^{\Phi}\right)=\norm{e^{-\Phi}}_{\Lip(M,d)}\left( \sup_M e^{-\Phi}\right)^{-1}\leq \norm{\Phi}_{\Lip(M,d)},
\end{equation*}
which completes the proof.
\end{proof}

\section{Proofs of noninjectivity}
\label{section_noninjectivity_proofs}

\LemmaNoninjectivityMisfitMaps*
\begin{proof}[Proof of \Cref{lemma_noninjectivity_of_dataMisfit_maps}]
 Proof of \cref{item_lemma_noninjectivity_of_dataMisfit_maps_noninjective}: It suffices to consider the case where $\Phi\mapsto Z_{\Phi,\mu}$ has the domain $L^\infty_\mu(\Theta,\R)$, since $L^\infty_\mu(\Theta,\R)\subseteq L^p_\mu(\Theta,\R)$ for every $p\in[1,\infty]$.
 
 If $\supp{\mu}$ is uncountable, then $L^\infty_\mu(\Theta,\R)$ is an infinite-dimensional space. 
 If $\supp{\mu}=(x_n)_{n}$ is countable, then every $f\in L^\infty(\Theta,\R)$ is uniquely determined by the collection $(f(x_n))_{n}$. In particular, if $\card{\supp{\mu}}\geq 2$, then $L^\infty_\mu(\Theta,\R)$ has dimension at least 2. 
 On the other hand, the codomain $\R_{>0}$ of the map is one-dimensional. Thus, the map $\Phi\mapsto Z_{\Phi,\mu}$ cannot be injective.

 Proof of \cref{item_lemma_noninjectivity_of_dataMisfit_maps_noninvariance_under_translations}: by \eqref{eq_normalisationConstant_function},
 \begin{equation*}
  Z_{\Phi+c,\mu}=\int \exp(-c-\Phi)\rd \mu =\exp(-c)\int \exp(-\Phi)\rd \mu=\exp(-c)Z_{\Phi,\mu}.
 \end{equation*}

 Proof of \cref{item_lemma_noninjectivity_of_dataMisfit_maps_equivalent_condition_for_agreement}: Let $\theta\in\Theta$ be arbitrary. By \eqref{eq_likelihood_function}, we have the following equivalence of statements: 
 \begin{align*}
   & & -(\Phi_1(\theta)+\log Z_{\Phi_1,\mu})=\log \ell_{\Phi_1,\mu}(\theta)=&\log \ell_{\Phi_2,\mu}(\theta)=-(\Phi_2(\theta)+\log Z_{\Phi_2,\mu}) & & \text{}
  \\
  \Longleftrightarrow & & \Phi_1(\theta)-\Phi_2(\theta) =& \log Z_{\Phi_2,\mu}-\log Z_{\Phi_1,\mu}.
 \end{align*}
\end{proof}

\LemmaNoninjectivityPriorMaps*
\begin{proof}[Proof of \Cref{lemma_noninjectivity_of_prior_maps}]
 Proof of \cref{item_lemma_noninjectivity_of_prior_maps_noninjective}: The set $D(\nu)$ is in one-to-one correspondence with the set $\mathcal{F}(\nu)\coloneqq \{f\in L^1_\nu(\Theta,\R_{\geq 0}) : \norm{f}_{L^1_\nu}=1\}$, because for every $\mu\in D(\nu)$ there exists a unique $\tfrac{\rd\mu}{\rd\nu}\in \mathcal{F}(\nu)$ by the Radon--Nikodym theorem, and because every $g\in\mathcal{F}(\nu)$ defines some $\mu_g\in D(\nu)$ via $\tfrac{\rd \mu_g}{\rd\nu}\coloneqq g$. 
 
 If $\supp{\nu}$ is uncountable or countably infinite, then $\mathcal{F}(\nu)$ is an infinite-dimensional space. 
 If $\supp{\nu}=(x_n)_{n=1}^N$ for some $N\in\N$, then every $f\in\mathcal{F}(\nu)$ is represented by $(f(x_n))_{n=1}^{N}\in\R_{> 0}^{N}$. 
 Given the constraint $\sum_{n=1}^{N} f(x_n)=\norm{f}_{L^1_\nu}=1$, it follows that every $f\in\mathcal{F}(\nu)$ is uniquely defined by any $N-1$ elements of $(f(x_n))_{n=1}^{N}$.  
 In particular, if $\card{\supp{\nu}}=N$, then every $f\in\mathcal{F}(\nu)$ is determined by $N-1$ scalars $(f(x_n))_{n=1}^{N-1}$. This shows that $\mathcal{F}(\nu)$ has dimension $N-1$.
 Since the codomain $\R_{>0}$ of the map $D(\nu)\ni \mu\mapsto Z_{\Phi,\mu}$ has dimension 1, the map $\mathcal{F}(\nu)\ni \frac{\rd\mu}{\rd\nu} \mapsto Z_{\Phi,\mu}\in\R_{>0}$ cannot be injective when $N-1\geq 2$.
 Given the one-to-one correspondence between $D(\nu)$ and $\mathcal{F}(\nu)$, it follows that the map $D(\nu)\ni \mu\mapsto Z_{\Phi,\mu}\in\R_{>0}$ is not injective.

 Proof of \cref{item_lemma_noninjectivity_of_prior_maps_equivalent_condition_for_agreement}: By the hypothesis that $Z_{\Phi,\mu}\in\R_{>0}$ for every $\mu\in D(\nu)$ and $\mu_i\in D(\nu)$ for $i=1,2$, it follows from \eqref{eq_posterior_function} that $(\mu_1)_\Phi\in\mathcal{P}(\Theta)$ for $i=1,2$.
 By \eqref{eq_posterior_function},
   \begin{equation}
  \label{eq_radon_nikodym_derivative_different_priors}
  \frac{\rd (\mu_i)_\Phi}{\rd \nu}=\frac{\exp(-\Phi)}{Z_{\Phi,\mu_i}}\frac{\rd \mu_i}{\rd \nu},\quad i=1,2.
 \end{equation}
 By the hypothesis that $\exp(-\Phi)\in\R_{>0}$ $\nu$-a.s., it follows that $\tfrac{\rd (\mu_i)_\Phi}{\rd \nu}>0$ if and only if $\tfrac{\rd \mu_i}{\rd \nu}>0$.
 
 If $\mu_1=\mu_2$ as probability measures, then $Z_{\Phi,\mu_1}=Z_{\Phi,\mu_2}$ by \eqref{eq_normalisationConstant_function} and $\tfrac{\rd\mu_1}{\rd\nu}=\tfrac{\rd\mu_2}{\rd\nu}$, and thus $(\mu_1)_\Phi=(\mu_2)_\Phi$ by \eqref{eq_radon_nikodym_derivative_different_priors}.  
 
 For the converse implication, suppose that $(\mu_1)_\Phi=(\mu_2)_\Phi$. We have
\begin{align*}
 (\mu_1)_\Phi=(\mu_2)_\Phi \Longleftrightarrow & \nu\left(\frac{\exp(-\Phi)}{Z_{\Phi,\mu_1}}\frac{\rd \mu_1}{\rd \nu}=\frac{\exp(-\Phi)}{Z_{\Phi,\mu_2}}\frac{\rd \mu_2}{\rd \nu}\right)=1 & \text{by \eqref{eq_radon_nikodym_derivative_different_priors}}
 \\
 \Longleftrightarrow &  \nu\left(\frac{1}{Z_{\Phi,\mu_1}}\frac{\rd \mu_1}{\rd \nu}=\frac{1}{Z_{\Phi,\mu_2}}\frac{\rd \mu_2}{\rd \nu}\right)=1 & \nu(\exp(-\Phi)\in\R_{>0})=1
 \\
 \Longrightarrow& \frac{1}{Z_{\Phi,\mu_1}}\int \frac{\rd \mu_1}{\rd \nu}\rd \nu=\frac{1}{Z_{\Phi,\mu_2}}\int \frac{\rd \mu_2}{\rd \nu}\rd\nu.
\end{align*}
This proves that $(\mu_1)_\Phi=(\mu_2)_\Phi$ implies $Z_{\Phi,\mu_1}=Z_{\Phi,\mu_2}$. By \eqref{eq_radon_nikodym_derivative_different_priors} and the hypothesis that $\exp(-\Phi)\in\R_{>0}$ $\nu$-a.s., it follows that 
\begin{equation*}
 \nu\left(\frac{\rd\mu_1}{\rd\nu}=\frac{\rd\mu_2}{\rd\nu}\right)=1 \Longleftrightarrow \mu_1=\mu_2.
\end{equation*}
This completes the proof of \Cref{lemma_noninjectivity_of_prior_maps}.
\end{proof}

\section{Proofs of total variation bounds}
\label{section_TVbounds_proofs}

\subsection{Proofs of total variation bounds for misfit perturbations}
\label{section_TVbounds_Misfitperturbations_proofs}

\TVboundsMisfitperturbationsLipschitz*
\begin{proof}[Proof of \Cref{proposition_TV_bounds_misfit_perturbations_via_Lipschitz_continuity}]
For $i=1,2$, if $\Phi_i\in L^1_\mu(\Theta,\R_{\geq 0})$, then $\esssup_\mu\exp(-\Phi_i)\leq 1$. Thus,
\begin{align*}
 &d_{\TV}(\mu_{\Phi_1},\mu_{\Phi_2})=\frac{1}{2}\Norm{\ell_{\Phi_1,\mu}-\ell_{\Phi_2,\mu}}_{L^1_\mu} & \text{by \eqref{eq_total_variation_metric}}
 \\
 \leq &\frac{1}{2}\left(\frac{1}{Z_{\Phi_1,\mu}}\vee \frac{1}{Z_{\Phi_2,\mu}}\right)\Norm{\log \ell_{\Phi_1,\mu}-\log\ell_{\Phi_2,\mu}}_{L^1_\mu} & \text{by \Cref{lemma_preliminary_bounds_on_pth_root_likelihoods}\ref{item_lemma_preliminary_bounds_on_pth_root_likelihoods_upper_bound}.}
\end{align*}
If $\Phi_i\in L^\infty_\mu(\Theta,\R_{\geq 0})$, then $\essinf_\mu\exp(-\Phi_i)=\exp(-\Norm{\Phi_i}_{L^\infty_\mu})$ for $i=1,2$. 
Thus,
\begin{align*}
 &d_{\TV}(\mu_{\Phi_1},\mu_{\Phi_2})=\frac{1}{2}\Norm{\ell_{\Phi_1,\mu}-\ell_{\Phi_2,\mu}}_{L^1_\mu} & \text{by \eqref{eq_total_variation_metric}}
 \\
 \geq &\frac{1}{2}\left(\frac{\exp(-\Norm{\Phi_1}_{L^\infty_\mu})}{Z_{\Phi_1,\mu}}\wedge \frac{\exp(-\Norm{\Phi_2}_{L^\infty_\mu})}{Z_{\Phi_2,\mu}}\right)\Norm{\log \ell_{\Phi_1,\mu}-\log\ell_{\Phi_2,\mu}}_{L^1_\mu} & \text{by \Cref{lemma_preliminary_bounds_on_pth_root_likelihoods}\ref{item_lemma_preliminary_bounds_on_pth_root_likelihoods_lower_bound}.}
\end{align*}
By \Cref{lemma_noninjectivity_of_dataMisfit_maps}\ref{item_lemma_noninjectivity_of_dataMisfit_maps_equivalent_condition_for_agreement}, $\Phi_1-\Phi_2$ is $\mu$-a.s. constant if and only if $\log\ell_{\Phi_1,\mu}=\log\ell_{\Phi_2,\mu}$ $\mu$-a.s.
By \eqref{eq_posterior_function} and \eqref{eq_total_variation_metric}, $\log\ell_{\Phi_1,\mu}=\log\ell_{\Phi_2,\mu}$ $\mu$-a.s. if and only if $d_{\TV}(\mu_{\Phi_1},\mu_{\Phi_2})=0$.
\end{proof}

\TVboundsMisfitperturbationsTriangle*
\begin{proof}[Proof of \Cref{proposition_TV_bounds_misfit_perturbations_via_triangle_inequality}]
We have
\begin{align*}
  &2d_{\TV}(\mu_{\Phi_1},\mu_{\Phi_2})
  \\
  = &\Norm{\ell_{\Phi_1,\mu}-\ell_{\Phi_2,\mu}}_{L^1_\mu} & \text{by \eqref{eq_total_variation_metric}}
  \\
  =& \Norm{\frac{e^{-\Phi_1}-e^{-\Phi_2}}{Z_{\Phi_1,\mu}}+\frac{e^{-\Phi_2}}{Z_{\Phi_1,\mu}}-\frac{e^{-\Phi_2}}{Z_{\Phi_2,\mu}}}_{L^1_\mu} & \text{by \eqref{eq_likelihood_function}}
  \\
  =& \frac{1}{Z_{\Phi_1,\mu}}\Norm{e^{-\Phi_1}-e^{-\Phi_2}+\ell_{\Phi_2,\mu} (Z_{\Phi_2,\mu}-Z_{\Phi_1,\mu})}_{L^1_\mu}.
  \end{align*}
  By \eqref{eq_posterior_function}, $\norm{\ell_{\Phi_2,\mu}}_{L^1_\mu}=1$, and $\norm{\ell_{\Phi_2,\mu}(Z_{\Phi_2,\mu}-Z_{\Phi_1,\mu})}_{L^1_\mu}=\abs{Z_{\Phi_2,\mu}-Z_{\Phi_1,\mu}}$.
  Thus, by the reverse triangle inequality and the triangle inequality,
  \begin{subequations}
   \begin{align}
  \frac{\Abs{\norm{ e^{-\Phi_1}-e^{-\Phi_2}}_{L^1_\mu}- \abs{Z_{\Phi_2,\mu}-Z_{\Phi_1,\mu}}}}{Z_{\Phi_1,\mu}}  
  \leq &2d_{\TV}(\mu_{\Phi_1},\mu_{\Phi_2})
  \label{eq_TV_lower_bound_misfit_perturbation}
  \\
  \leq &\frac{\norm{ e^{-\Phi_1}-e^{-\Phi_2}}_{L^1_\mu}+ \abs{Z_{\Phi_2,\mu}-Z_{\Phi_1,\mu}}}{Z_{\Phi_1,\mu}}.
  \label{eq_TV_upper_bound_misfit_perturbation}
  \end{align}
  \end{subequations}
  By switching $\Phi_1$ and $\Phi_2$ in \eqref{eq_TV_lower_bound_misfit_perturbation} and \eqref{eq_TV_upper_bound_misfit_perturbation} and considering the resulting inequalities, we can replace the denominator of the left-hand side of \eqref{eq_TV_lower_bound_misfit_perturbation} and the right-hand side by $(Z_{\Phi_1,\mu}\wedge Z_{\Phi_2,\mu})^{-1}$ and $(Z_{\Phi_1,\mu}\vee Z_{\Phi_2,\mu})^{-1}$ respectively:
  \begin{align*}
  \frac{\Abs{\norm{ e^{-\Phi_1}-e^{-\Phi_2}}_{L^1_\mu}- \abs{Z_{\Phi_2,\mu}-Z_{\Phi_1,\mu}}}}{Z_{\Phi_1,\mu}\wedge Z_{\Phi_2,\mu}}  
  \leq &2d_{\TV}(\mu_{\Phi_1},\mu_{\Phi_2})
  \\
  \leq &\frac{\norm{ e^{-\Phi_1}-e^{-\Phi_2}}_{L^1_\mu}+ \abs{Z_{\Phi_2,\mu}-Z_{\Phi_1,\mu}}}{Z_{\Phi_1,\mu}\vee Z_{\Phi_2,\mu}}.
  \end{align*}
Recall that the triangle inequality implies the reverse triangle inequality and vice versa.
Recall also that equality holds in the triangle inequality $\norm{u-v}\leq \norm{u}+\norm{v}$ if and only if the vectors $u$ and $v$ are collinear.
Thus, equality holds in \eqref{eq_TV_lower_bound_misfit_perturbation} if and only if there exists some $\lambda$ such that $\mu$-a.s. $e^{-\Phi_1}-e^{-\Phi_2}=\lambda e^{-\Phi_2}$, which in turn is equivalent to the existence of some $c$ such that $\Phi_1=\Phi_2+c$ $\mu$-a.s.

   By \eqref{eq_Lipschitz_continuity_exp_function} and the hypothesis that $\essinf_\mu \Phi_i=0$ for $i=1,2$,
 \begin{equation*}
  \Norm{ e^{-\Phi_1}-e^{-\Phi_2}}_{L^1_\mu}\leq \left[\exp(-\essinf_\mu\Phi_1)\vee \exp(-\essinf_\mu\Phi_2)\right] \Norm{\Phi_1-\Phi_2}_{L^1_\mu}=\Norm{\Phi_1-\Phi_2}_{L^1_\mu}.
 \end{equation*}
 The inequality $\abs{Z_{\Phi_2,\mu}-Z_{\Phi_1,\mu}}\leq \norm{ e^{-\Phi_1}-e^{-\Phi_2}}_{L^1_\mu}$ and the equality condition follow from \Cref{lemma_L1norm_difference_of_unnormalised_likelihoods_minus_abs_diff_normalisation_constants_misfit_perturbation}\ref{item2_lemma_L1norm_difference_of_unnormalised_likelihoods_minus_abs_diff_normalisation_constants_misfit_perturbation}, which does not require $\essinf_\mu\Phi_i=0$ for $i=1,2$.
  This proves \cref{item_proposition_TV_bounds_misfit_perturbations_via_triangle_inequality_upper_bound}.
  
  To prove \cref{item_proposition_TV_bounds_misfit_perturbations_via_triangle_inequality_lower_bound}, note that the hypothesis that $\mu(\Phi_1-\Phi_2>0)$ and $\mu(\Phi_1-\Phi_2<0)$ are both positive implies the strict positivity of $\norm{ e^{-\Phi_1}-e^{-\Phi_2}}_{L^1_\mu}- \abs{Z_{\Phi_2,\mu}-Z_{\Phi_1,\mu}}$, by \Cref{lemma_L1norm_difference_of_unnormalised_likelihoods_minus_abs_diff_normalisation_constants_misfit_perturbation}\ref{item2_lemma_L1norm_difference_of_unnormalised_likelihoods_minus_abs_diff_normalisation_constants_misfit_perturbation}.
Recall from the proof of \cref{item_proposition_TV_bounds_misfit_perturbations_via_triangle_inequality_upper_bound} that equality holds in \eqref{eq_TV_lower_bound_misfit_perturbation} if and only if $\Phi_1-\Phi_2$ is $\mu$-a.s. constant.
If $\mu(\Phi_1-\Phi_2>0)$ and $\mu(\Phi_1-\Phi_2<0)$ are both positive, then the inequality in \eqref{eq_TV_lower_bound_misfit_perturbation} is strict.
Now suppose $Z_{\Phi_1,\mu}=Z_{\Phi_2,\mu}$. Then by the lower bound in \eqref{eq_Lipschitz_continuity_exp_function} and the hypothesis that $\Phi_i\in L^\infty_\mu(\Theta,\R)$ for $i=1,2$, we have
\begin{align}
\Norm{e^{-\Phi_1}-e^{-\Phi_2}}_{L^1_\mu}\geq &\int [e^{-\Phi_1}\wedge e^{-\Phi_2}]\Abs{\Phi_1-\Phi_2}\rd \mu 
\nonumber
\\
\geq& \left[ \exp(-\norm{\Phi_1}_{L^\infty_\mu})\wedge \exp(-\norm{\Phi_2}_{L^\infty_\mu})\right]\Norm{\Phi_1-\Phi_2}_{L^1_\mu}.
\label{eq_lower_bound_L1_norm_metric_unnormalised_likelihoods}
\end{align}
Since the inequality in \eqref{eq_TV_lower_bound_misfit_perturbation} is strict, we obtain
\begin{equation*}
d_{\TV}(\mu_{\Phi_1},\mu_{\Phi_2})>\frac{1}{2}\frac{\exp(-\norm{\Phi_1}_{L^\infty_\mu})\wedge \exp(-\norm{\Phi_2}_{L^\infty_\mu})}{Z_{\Phi_1,\mu}\wedge Z_{\Phi_2,\mu}}\Norm{\Phi_1-\Phi_2}_{L^1_\mu}.
\end{equation*}
Now suppose $Z_{\Phi_1,\mu}> Z_{\Phi_2,\mu}$. 
Then
\begin{align*}
 &\frac{1}{2}\left(\Norm{e^{-\Phi_1}-e^{-\Phi_2}}_{L^1_\mu}-\Abs{Z_{\Phi_2,\mu}-Z_{\Phi_1,\mu}}\right)
 \\
 =& \int_{\{\Phi_1>\Phi_2\}} \abs{e^{-\Phi_1}-e^{-\Phi_2}}\rd \mu & \text{by \Cref{lemma_L1norm_difference_of_unnormalised_likelihoods_minus_abs_diff_normalisation_constants_misfit_perturbation}\ref{item3_lemma_L1norm_difference_of_unnormalised_likelihoods_minus_abs_diff_normalisation_constants_misfit_perturbation}}
 \\
 \geq &  \int_{\{\Phi_1>\Phi_2\}} (e^{-\Phi_1}\wedge e^{-\Phi_2})\abs{\Phi_1-\Phi_2}\rd \mu &\text{by \eqref{eq_Lipschitz_continuity_exp_function}}
 \\
 \geq &  \exp(-\Norm{\Phi_1}_{L^\infty_\mu})\int_{\{\Phi_1>\Phi_2\}} \abs{\Phi_1-\Phi_2}\rd \mu & \Phi_1\in L^\infty_\mu,
 \end{align*}
and by \eqref{eq_TV_lower_bound_misfit_perturbation} we obtain
\begin{equation*}
 d_{\TV}(\mu_{\Phi_1},\mu_{\Phi_2})> \frac{\exp(-\Norm{\Phi_1}_{L^\infty_\mu})}{Z_{\Phi_2,\mu}}\int_{\{\Phi_1>\Phi_2\}} \abs{\Phi_1-\Phi_2}\rd \mu ,
\end{equation*}
which yields the lower bound in the second case of \cref{item_proposition_TV_bounds_misfit_perturbations_via_triangle_inequality_lower_bound}.
By switching the roles of $\Phi_1$ and $\Phi_2$ in the preceding argument, we obtain the lower bound in the third case of \cref{item_proposition_TV_bounds_misfit_perturbations_via_triangle_inequality_lower_bound}.

To prove \cref{item_proposition_TV_bounds_misfit_perturbations_via_triangle_inequality_lower_bound_no_sign_change}, note that we can use the same proof as in \cref{item_proposition_TV_bounds_misfit_perturbations_via_triangle_inequality_lower_bound} for the case where $Z_{\Phi_1,\mu}=Z_{\Phi_2,\mu}$. 
For the case where $Z_{\Phi_1,\mu}\neq Z_{\Phi_2,\mu}$, we use \eqref{eq_TV_lower_bound_misfit_perturbation} with \Cref{lemma_L1norm_difference_of_unnormalised_likelihoods_minus_abs_diff_normalisation_constants_misfit_perturbation}\ref{item2_lemma_L1norm_difference_of_unnormalised_likelihoods_minus_abs_diff_normalisation_constants_misfit_perturbation}.
\end{proof}

\subsection{Proofs of total variation bounds for prior perturbations}
\label{section_TVbounds_Priorperturbations_proofs}

\TVboundsPriorperturbationTriangle*
\begin{proof}[Proof of \Cref{proposition_TV_bounds_prior_perturbation_via_triangle_inequality}]
Let $\nu\in\mathcal{P}(\Theta)$ be such that $\mu_i\ll \nu$ for $i=1,2$, and let $\Phi\in L^0(\Theta,\R)$ be such that $(\mu_i)_\Phi\in\mathcal{P}(\Theta)$ for $i=1,2$, for $(\mu_i)_\Phi$ defined by \eqref{eq_posterior_function}.
Then
\begin{align}
 &2d_{\TV}((\mu_1)_\Phi,(\mu_2)_\Phi)
 \nonumber
 \\
 =&\int\Abs{\frac{\rd (\mu_1)_\Phi}{\rd \nu}-\frac{\rd (\mu_2)_\Phi}{\rd \nu}} \rd \nu & \text{by \eqref{eq_total_variation_metric}}
 \nonumber
 \\
 =& \int \Abs{ \frac{\exp(-\Phi)}{Z_{\Phi,\mu_1}}\frac{\rd \mu_1}{\rd \nu}-\frac{\exp(-\Phi)}{Z_{\Phi,\mu_2}}\frac{\rd \mu_2}{\rd \nu}}\rd \nu
 \nonumber & \text{by \eqref{eq_radon_nikodym_derivative_different_priors}}
 \\
 =& \int \Abs{ \frac{\exp(-\Phi)}{Z_{\Phi,\mu_1}}\left(\frac{\rd \mu_1}{\rd \nu}-\frac{\rd \mu_2}{\rd \nu}\right)+\left(\frac{\exp(-\Phi)}{Z_{\Phi,\mu_1}}-\frac{\exp(-\Phi)}{Z_{\Phi,\mu_2}}\right)
 \frac{\rd \mu_2}{\rd \nu}}\rd \nu
 \nonumber
\\
 =& \frac{1}{Z_{\Phi,\mu_1}}\int \exp(-\Phi) \Abs{ \left(\frac{\rd \mu_1}{\rd \nu}-\frac{\rd \mu_2}{\rd \nu}\right)+\frac{Z_{\Phi,\mu_2}-Z_{\Phi,\mu_1}}{Z_{\Phi,\mu_2}} \frac{\rd \mu_2}{\rd \nu}}\rd \nu .
 \label{eq_TV_metric_different_priors}
 \end{align}

 Proof of \cref{item_proposition_TV_bounds_prior_perturbation_via_triangle_inequality_upper_bound}: 
 Let $\nu=\lambda \mu_1+(1-\lambda)\mu_2$ for $0<\lambda <1$. Then $\mu_i\ll\nu$ for $i=1,2$ and $\essinf_\nu\Phi\leq \essinf_{\mu_i}\Phi$ for $i=1,2$ by \Cref{lemma_esssup_essinf_different_measures_different_functions}.
 The definition of the essential infimum as the largest lower bound and the hypotheses that $\Phi\in L^0(\Theta,\R_{\geq 0})$ and $\essinf_{\mu_i}\Phi=0$ for some $i\in\{1,2\}$ together imply that $\essinf_\nu \Phi=0$.
 We have
\begin{align*}
  &2d_{\TV}((\mu_1)_\Phi,(\mu_2)_\Phi)
 \\
 =& \frac{1}{Z_{\Phi,\mu_1}}\int \exp(-\Phi) \Abs{ \left(\frac{\rd \mu_1}{\rd \nu}-\frac{\rd \mu_2}{\rd \nu}\right)+\frac{Z_{\Phi,\mu_2}-Z_{\Phi,\mu_1}}{Z_{\Phi,\mu_2}} \frac{\rd \mu_2}{\rd \nu}}\rd \nu  & \text{by \eqref{eq_TV_metric_different_priors}}
 \\
 \leq & \frac{1}{Z_{\Phi,\mu_1}}\left(\int \exp(-\Phi) \Abs{ \left(\frac{\rd \mu_1}{\rd \nu}-\frac{\rd \mu_2}{\rd \nu}\right)}\rd \nu
+\abs{Z_{\Phi,\mu_2}-Z_{\Phi,\mu_1}} \int \frac{\exp(-\Phi)}{Z_{\Phi,\mu_2}} \frac{\rd \mu_2}{\rd \nu}\rd \nu  \right)
 \\
 =& \frac{1}{Z_{\Phi,\mu_1}}\left(\int \exp(-\Phi) \Abs{ \left(\frac{\rd \mu_1}{\rd \nu}-\frac{\rd \mu_2}{\rd \nu}\right)}\rd \nu
+\abs{Z_{\Phi,\mu_2}-Z_{\Phi,\mu_1}} \right) 
\\
 \leq &\frac{1}{Z_{\Phi,\mu_1}}\left(2d_{\TV}(\mu_1,\mu_2)
+\abs{Z_{\Phi,\mu_2}-Z_{\Phi,\mu_1}} \right).
\end{align*}
 The first inequality above follows from the triangle inequality. 
 The second equation follows since $(\mu_2)_\Phi\in\mathcal{P}(\Theta)$.
 The second inequality follows from $\essinf_\nu\Phi=0$ and the definition \eqref{eq_total_variation_metric} of the total variation metric.
 Switching $\mu_1$ and $\mu_2$ in the argument above and using that $Z_{\Phi,\mu_1}^{-1}\wedge Z_{\Phi,\mu_2}^{-1}=(Z_{\Phi,\mu_1}\vee Z_{\Phi,\mu_2})^{-1}$ implies the first bound in \cref{item_proposition_TV_bounds_prior_perturbation_via_triangle_inequality_upper_bound}.
 
To prove the bound on $\abs{Z_{\Phi,\mu_2}-Z_{\Phi,\mu_1}}$, we have for $\nu$ as above that
 \begin{equation*}
  \abs{Z_{\Phi,\mu_2}-Z_{\Phi,\mu_1}}=\Abs{\int e^{-\Phi}\rd \mu_1-\int e^{-\Phi}\rd \mu_2 }\leq \int e^{-\Phi} \Abs{ \frac{\rd \mu_1}{\rd \nu}-\frac{\rd \mu_2}{\rd \nu}}\rd\nu\leq 2d_{\TV}(\mu_1,\mu_2).
 \end{equation*}
 By \Cref{lemma_L1norm_difference_of_unnormalised_likelihoods_minus_abs_diff_normalisation_constants_prior_perturbation}\ref{lemma_L1norm_difference_of_unnormalised_likelihoods_minus_abs_diff_normalisation_constants_prior_perturbation_item2}, the left inequality is strict if and only if $d_{\TV}(\mu_1,\mu_2)>0$. 
 The right inequality is an equality if and only if $\Phi$ is $\nu$-a.s. constant.
 If $\Phi$ is $\nu$-a.s. constant, then $(\mu_i)_\Phi=\mu$ by \eqref{eq_posterior_function} and $d_{\TV}((\mu_1)_\Phi,(\mu_2)_\Phi)= d_{\TV}(\mu_1,\mu_2)$.
 
 Proof of \cref{item_proposition_TV_bounds_prior_perturbation_via_triangle_inequality_lower_bound}: If $\nu \coloneqq \lambda \mu_1+(1-\lambda)\mu_2$ for $0<\lambda <1$, then $\mu_i\ll\nu$ for $i=1,2$ and in addition $\norm{\Phi}_{L^\infty_\nu}=\norm{\Phi}_{L^\infty_{\mu_1}}\vee \norm{\Phi}_{L^\infty_{\mu_2}}$, by \Cref{lemma_esssup_essinf_different_measures_different_functions}. 
 Thus, by the reverse triangle inequality,
\begin{align*}
  &2d_{\TV}((\mu_1)_\Phi,(\mu_2)_\Phi)
 \\
 \geq & \frac{\exp(-\norm{\Phi}_{L^\infty_\nu})}{Z_{\Phi,\mu_1}}  \int\Abs{ \left(\frac{\rd \mu_1}{\rd \nu}-\frac{\rd \mu_2}{\rd \nu}\right)+\frac{Z_{\Phi,\mu_2}-Z_{\Phi,\mu_1}}{Z_{\Phi,\mu_2}} \frac{\rd \mu_2}{\rd \nu}}\rd \nu  & \text{by \eqref{eq_TV_metric_different_priors}, $\Phi\in L^\infty_\nu(\Theta,\R)$}
 \\
 \geq &\frac{\exp(-\norm{\Phi}_{L^\infty_\nu})}{Z_{\Phi,\mu_1}} \Abs{\int\Abs{\frac{\rd \mu_1}{\rd \nu}-\frac{\rd \mu_2}{\rd \nu}}\rd \nu- \int\Abs{ \frac{Z_{\Phi,\mu_2}-Z_{\Phi,\mu_1}}{Z_{\Phi,\mu_2}} \frac{\rd \mu_2}{\rd \nu}}\rd \nu}
 \\
 =& \frac{\exp(-\norm{\Phi}_{L^\infty_\nu})}{Z_{\Phi,\mu_1}} \Abs{2d_{\TV}(\mu_1,\mu_2)- \Abs{ \frac{Z_{\Phi,\mu_2}-Z_{\Phi,\mu_1}}{Z_{\Phi,\mu_2}} }} & \text{by \eqref{eq_total_variation_metric}, $\mu_2\in\mathcal{P}(\Theta)$}.
\end{align*}
Equality holds in the application of the triangle inequality, and thus also in the reverse triangle inequality, if and only if there exists some constant $\xi$ such that $\tfrac{\rd\mu_1}{\rd\nu}-\tfrac{\rd\mu_2}{\rd\nu}=\xi \tfrac{\rd\mu_2}{\rd\nu}$ $\nu$-a.s.
This condition holds if and only if $\tfrac{\rd\mu_1}{\rd\nu}-\tfrac{\rd\mu_2}{\rd\nu}$ does not change sign $\nu$-a.s. 
By \Cref{lemma_L1norm_difference_of_unnormalised_likelihoods_minus_abs_diff_normalisation_constants_prior_perturbation}\ref{lemma_L1norm_difference_of_unnormalised_likelihoods_minus_abs_diff_normalisation_constants_prior_perturbation_item1},  $\tfrac{\rd\mu_1}{\rd\nu}-\tfrac{\rd\mu_2}{\rd\nu}$ does not change sign $\nu$-a.s. if and only if $d_{\TV}(\mu_1,\mu_2)=0$.
This proves that if 
\begin{equation*}
 2d_{\TV}((\mu_1)_\Phi,(\mu_2)_\Phi)=\frac{\exp(-\norm{\Phi}_{L^\infty_\nu})}{Z_{\Phi,\mu_1}} \Abs{2d_{\TV}(\mu_1,\mu_2)- \Abs{ \frac{Z_{\Phi,\mu_2}-Z_{\Phi,\mu_1}}{Z_{\Phi,\mu_2}} }}
\end{equation*}
then $d_{\TV}(\mu_1,\mu_2)=0$.
On the other hand, if $d_{\TV}(\mu_1,\mu_2)=0$, then 
\begin{equation*}
 2d_{\TV}((\mu_1)_\Phi,(\mu_2)_\Phi)=0=\Abs{2d_{\TV}(\mu_1,\mu_2)- \Abs{ \frac{Z_{\Phi,\mu_2}-Z_{\Phi,\mu_1}}{Z_{\Phi,\mu_2}} }}.
\end{equation*}
\end{proof}

\TVboundsPriorperturbationLipschitz*
\begin{proof}[Proof of \Cref{proposition_TV_bounds_prior_perturbations_via_Lipschitz_continuity}]
By the hypothesis that $\tfrac{\rd\mu_2}{\rd\mu_1}$ is bounded from above and bounded away from zero, it follows that $\mu_1\sim\mu_2$.
Thus, we can set $\nu\leftarrow \mu_1$ in \eqref{eq_total_variation_metric}, to obtain
\begin{equation*}
  2d_{\TV}((\mu_1)_\Phi,(\mu_2)_\Phi)= \int\Abs{\frac{\rd (\mu_1)_\Phi}{\rd \mu_1}-\frac{\rd (\mu_2)_\Phi}{\rd \mu_1}} \rd \mu_1 .
\end{equation*}
Since $c\leq \tfrac{\rd\mu_2}{\rd\mu_1}\leq C$ $\mu_1$-a.s., it follows that $-\infty<\log c\leq \log \tfrac{\rd\mu_2}{\rd\mu_1}\leq \log C<\infty$ $\mu_1$-a.s.
Hence $\log \tfrac{\rd\mu_2}{\rd\mu_1}\in L^\infty_{\mu_1}$.

Proof of \cref{proposition_TV_bounds_prior_perturbations_via_Lipschitz_continuity_upper_bound}: By \eqref{eq_Lipschitz_continuity_exp_function}, \eqref{eq_posterior_function}, and the hypothesis that $\essinf_{\mu_1}\Phi=0$,
\begin{align*}
&\int\Abs{\frac{\rd (\mu_1)_\Phi}{\rd \mu_1}-\frac{\rd (\mu_2)_\Phi}{\rd \mu_1}} \rd \mu_1 
  \\
  \leq & \int \left(\frac{\rd (\mu_1)_\Phi}{\rd \mu_1}\vee \frac{\rd (\mu_2)_\Phi}{\rd \mu_1} \right)\Abs{\log \frac{\rd (\mu_1)_\Phi}{\rd \mu_1}-\log\frac{\rd (\mu_2)_\Phi}{\rd \mu_1}} \rd \mu_1
  \\
  =& \int \left[\frac{\exp(-\Phi)}{Z_{\Phi,\mu_1}}\vee \left(\frac{\exp(-\Phi)}{Z_{\Phi,\mu_2}}\frac{\rd\mu_2}{\rd \mu_1} \right)\right]\Abs{-\log Z_{\Phi,\mu_1}+\log Z_{\Phi,\mu_2}-\log\frac{\rd \mu_2}{\rd \mu_1}} \rd \mu_1 
  \\
  \leq & \left[\frac{1}{Z_{\Phi,\mu_1}}\vee \left(\frac{1}{Z_{\Phi,\mu_2}}\Norm{\frac{\rd\mu_2}{\rd \mu_1}}_{L^\infty_{\mu_1}} \right)\right]\Norm{\log \frac{Z_{\Phi,\mu_2}}{Z_{\Phi,\mu_1}}-\log\frac{\rd \mu_2}{\rd \mu_1} }_{L^1_{\mu_1}}.
\end{align*}
If $\norm{\frac{\rd\mu_2}{\rd \mu_1}}_{L^\infty_{\mu_1}}< 1$, then $\mu_2(\Theta)<1$. Thus, $\norm{\frac{\rd\mu_2}{\rd \mu_1}}_{L^\infty_{\mu_1}}\geq  1$, and we obtain the upper bound on $d_{\TV}((\mu_1)_\Phi,(\mu_2)_\Phi)$.
By \eqref{eq_Lipschitz_continuity_log_function} and the upper bound on $\abs{Z_{\Phi,\mu_2}-Z_{\Phi,\mu_1}}$ in \Cref{proposition_TV_bounds_prior_perturbation_via_triangle_inequality}\ref{item_proposition_TV_bounds_prior_perturbation_via_triangle_inequality_upper_bound},
\begin{equation*}
\Abs{\log Z_{\Phi,\mu_2}-\log Z_{\Phi,\mu_1}}\leq \frac{\Abs{Z_{\Phi,\mu_1}-Z_{\Phi,\mu_2}}}{Z_{\Phi,\mu_1}\wedge Z_{\Phi,\mu_2}}\leq\frac{2d_{\TV}(\mu_1,\mu_2)}{Z_{\Phi,\mu_1}\wedge Z_{\Phi,\mu_2}},
\end{equation*}
which yields the upper bound on $\Abs{\log \tfrac{Z_{\Phi,\mu_2}}{Z_{\Phi,\mu_1}}}$.

Proof of \cref{proposition_TV_bounds_prior_perturbations_via_Lipschitz_continuity_lower_bound}: By \eqref{eq_Lipschitz_continuity_exp_function}, \eqref{eq_posterior_function}, and the hypothesis that $\Phi\in L^\infty_{\mu_1}$,
\begin{align*}
&\int\Abs{\frac{\rd (\mu_1)_\Phi}{\rd \mu_1}-\frac{\rd (\mu_2)_\Phi}{\rd \mu_1}} \rd \mu_1 
  \\
  \geq & \int \left(\frac{\rd (\mu_1)_\Phi}{\rd \mu_1}\wedge \frac{\rd (\mu_2)_\Phi}{\rd \mu_1} \right)\Abs{\log \frac{\rd (\mu_1)_\Phi}{\rd \mu_1}-\log\frac{\rd (\mu_2)_\Phi}{\rd \mu_1}} \rd \mu_1 
  \\
  =& \int \left[\frac{\exp(-\Phi)}{Z_{\Phi,\mu_1}}\wedge \left(\frac{\exp(-\Phi)}{Z_{\Phi,\mu_2}}\frac{\rd\mu_2}{\rd \mu_1} \right)\right]\Abs{-\log Z_{\Phi,\mu_1}+\log\frac{\rd \mu_1}{\rd \mu_1}+\log Z_{\Phi,\mu_2}-\log\frac{\rd \mu_2}{\rd \mu_1}} \rd \mu_1 
  \\
  =& \exp(-\norm{\Phi}_{L^\infty_{\mu_1}}) \left[\frac{1}{Z_{\Phi,\mu_1}}\wedge \left(\frac{1}{Z_{\Phi,\mu_2}}\Norm{\frac{\rd\mu_1}{\rd \mu_2}}_{L^\infty_{\mu_1}} \right)\right]\Norm{\log \frac{Z_{\Phi,\mu_2}}{Z_{\Phi,\mu_1}}-\log\frac{\rd \mu_2}{\rd \mu_1} }_{L^1_{\mu_1}}.
\end{align*}
If $\norm{\frac{\rd\mu_1}{\rd \mu_2}}_{L^\infty_{\mu_1}}>1$, then $\norm{\tfrac{\rd\mu_2}{\rd\mu_1}}_{L^\infty_{\mu_1}}<1$, which in turn implies $\mu_2(\Theta)<1$. 
Thus, $\norm{\frac{\rd\mu_1}{\rd \mu_2}}_{L^\infty_{\mu_1}}\leq 1$.
This yields the lower bound.
\end{proof}

\section{Proofs of Hellinger bounds}
\label{section_Hellingerbounds_proofs}

\subsection{Proofs of Hellinger bounds for misfit perturbations}
\label{section_Hellingerbounds_Misfitperturbations_proofs}

\HellingerboundsMisfitperturbationsLipschitz*
\begin{proof}[Proof of \Cref{proposition_Hellinger_bounds_misfit_perturbations_via_Lipschitz_continuity}]
Proof of \cref{proposition_Hellinger_bounds_misfit_perturbations_via_Lipschitz_continuity_upper_bound}:
 \begin{align*}
  &d_{\Hel}(\mu_{\Phi_1},\mu_{\Phi_2})
  \\
  = &\Norm{\ell_{\Phi_1,\mu}^{1/2}-\ell_{\Phi_2,\mu}^{1/2}}_{L^2_\mu} & \text{by \eqref{eq_Hellinger_metric} and  \eqref{eq_posterior_function}}
  \\
  \leq &  \tfrac{1}{2}\Norm{ \left(\ell_{\Phi_1,\mu}^{1/2}\vee \ell_{\Phi_2,\mu}^{1/2}\right)\abs{\log \ell_{\Phi_1,\mu}-\log\ell_{\Phi_2,\mu}}}_{L^2_\mu} & \text{by upper bound in \eqref{eq_Lipschitz_continuity_exp_function}}
  \\
  \leq & \frac{1}{2}\left(\esssup_\mu \ell_{\Phi_1,\mu}^{1/2}\vee \esssup_\mu \ell_{\Phi_2,\mu}^{1/2}\right)\norm{\log \ell_{\Phi_1,\mu}-\log\ell_{\Phi_2,\mu}}_{L^2_\mu}
  \\
  \leq & \frac{1}{2}\left(\frac{1}{Z_{\Phi_1,\mu}^{1/2}}\vee\frac{1}{Z_{\Phi_2,\mu}^{1/2}}\right)\norm{\log \ell_{\Phi_1,\mu}-\log\ell_{\Phi_2,\mu}}_{L^2_\mu} & \essinf_\mu\Phi_i\geq 0,\ i=1,2.
 \end{align*}

Proof of \cref{proposition_Hellinger_bounds_misfit_perturbations_via_Lipschitz_continuity_lower_bound}: 
 \begin{align*}
  &d_{\Hel}(\mu_{\Phi_1},\mu_{\Phi_2})
  \\
  = &\Norm{\ell_{\Phi_1,\mu}^{1/2}-\ell_{\Phi_2,\mu}^{1/2}}_{L^2_\mu} & \text{by \eqref{eq_Hellinger_metric} and  \eqref{eq_posterior_function}}
  \\
  \geq & \tfrac{1}{2}\Norm{ \left(\ell_{\Phi_1,\mu}^{1/2}\wedge \ell_{\Phi_2,\mu}^{1/2}\right)\abs{\log \ell_{\Phi_1,\mu}-\log\ell_{\Phi_2,\mu}}}_{L^2_\mu} & \text{by lower bound in \eqref{eq_Lipschitz_continuity_exp_function}}
  \\
  \geq & \frac{1}{2}\left(\essinf_\mu \ell_{\Phi_1,\mu}^{1/2}\wedge \essinf_\mu \ell_{\Phi_2,\mu}^{1/2}\right)\norm{\log \ell_{\Phi_1,\mu}-\log\ell_{\Phi_2,\mu}}_{L^2_\mu}
  \\
  \geq & \frac{1}{2}\left(\frac{\exp(-\norm{\Phi_1}_{L^\infty_\mu})}{Z_{\Phi_1,\mu}^{1/2}}\wedge\frac{\exp(-\norm{\Phi_2}_{L^\infty_\mu})}{Z_{\Phi_2,\mu}^{1/2}}\right)\norm{\log \ell_{\Phi_1,\mu}-\log\ell_{\Phi_2,\mu}}_{L^2_\mu} & \Phi_1,\Phi_2\in L^\infty_\mu(\Theta,\R).
 \end{align*}
  By \Cref{lemma_noninjectivity_of_dataMisfit_maps}\ref{item_lemma_noninjectivity_of_dataMisfit_maps_equivalent_condition_for_agreement}, $\Phi_1-\Phi_2$ is $\mu$-a.s. constant if and only if $\log\ell_{\Phi_1,\mu}=\log\ell_{\Phi_2,\mu}$ $\mu$-a.s.
  By \eqref{eq_posterior_function} and \eqref{eq_Hellinger_metric}, $\log\ell_{\Phi_1,\mu}=\log\ell_{\Phi_2,\mu}$ $\mu$-a.s. if and only if $d_{\Hel}(\mu_{\Phi_1},\mu_{\Phi_2})=0$.
\end{proof}

\HellingerboundsMisfitperturbationsTriangle*
\begin{proof}[Proof of \Cref{proposition_Hellinger_upper_bound_misfit_perturbations_via_triangle_inequality}]
We have
 \begin{align*}
  &d_{\Hel}(\mu_{\Phi_1},\mu_{\Phi_2})
  \\
  = &\Norm{\ell_{\Phi_1,\mu}^{1/2}-\ell_{\Phi_2,\mu}^{1/2}}_{L^2_\mu} & \text{by \eqref{eq_Hellinger_metric}}
  \\
  =& \Norm{\frac{e^{-\Phi_1/2}-e^{-\Phi_2/2}}{Z_{\Phi_1,\mu}^{1/2}}+\frac{e^{-\Phi_2/2}}{Z_{\Phi_1,\mu}^{1/2}}-\frac{e^{-\Phi_2/2}}{Z_{\Phi_2,\mu}^{1/2}}}_{L^2_\mu} & \text{by \eqref{eq_likelihood_function}}
  \\
  \leq & \Norm{\frac{e^{-\Phi_1/2}-e^{-\Phi_2/2}}{Z_{\Phi_1,\mu}^{1/2}}}_{L^2_\mu}+\Norm{\ell_{\Phi_2,\mu}^{1/2}\left(\frac{Z_{\Phi_2,\mu}^{1/2}-Z_{\Phi_1,\mu}^{1/2}}{Z_{\Phi_1,\mu}^{1/2}}\right)}_{L^2_\mu}
  \\
  =&\frac{1}{Z_{\Phi_1,\mu}^{1/2}}\left( \Norm{e^{-\Phi_1/2}-e^{-\Phi_2/2}}_{L^2_\mu}+\Abs{Z_{\Phi_2,\mu}^{1/2}-Z_{\Phi_1,\mu}^{1/2}}\right) & \text{$\norm{\ell_{\Phi_2,\mu}^{1/2}}_{L^2_\mu}=1$ by \eqref{eq_posterior_function}.}
  \end{align*}
  The inequality above follows from the triangle inequality.
  Thus, equality holds if and only if there exists some $\lambda$ such that $e^{-\Phi_1/2}-e^{-\Phi_2/2}=\lambda e^{-\Phi_2/2}$ $\mu$-a.s., which is equivalent to $\Phi_1-\Phi_2$ being $\mu$-a.s. constant.
By switching $\Phi_1$ and $\Phi_2$ in the reasoning above, we obtain the same inequality, except that division by $Z_{\Phi_1,\mu}^{1/2}$ is replaced with division by $Z_{\Phi_2,\mu}^{1/2}$. 
Since both these inequalities are true, we obtain the first inequality in the conclusion of the proposition.
  
  By \eqref{eq_normalisationConstant_function} and the Cauchy--Schwarz inequality,
  \begin{equation*}
  \Abs{Z_{\Phi_2,\mu}^{1/2}-Z_{\Phi_1,\mu}^{1/2}}^2=Z_{\Phi_1,\mu}-2Z_{\Phi_1,\mu}^{1/2} Z_{\Phi_2,\mu}^{1/2}+Z_{\Phi_2,\mu}\leq  \Norm{e^{-\Phi_1/2}-e^{-\Phi_2/2}}_{L^2_\mu}^2,
  \end{equation*}
  which proves the bound on $\abs{Z_{\Phi_2,\mu}^{1/2}-Z_{\Phi_1,\mu}^{1/2}}$.
  This proves the second inequality in the conclusion of the proposition.
  Equality holds in the application of the Cauchy--Schwarz inequality if and only if there exists $\lambda\in\R$ such that $e^{-\Phi_1/2}=\lambda e^{-\Phi_2/2}$ $\mu$-a.s., i.e. if and only if $\Phi_1-\Phi_2$ is $\mu$-a.s. constant.
  
  The inequalities we have proven thus far imply that
  \begin{equation}
    \label{eq_bound00}
    d_{\Hel}(\mu_{\Phi_1},\mu_{\Phi_2})\leq \frac{2}{Z_{\Phi_1,\mu}^{1/2}\vee Z_{\Phi_2,\mu}^{1/2}}\Norm{e^{-\Phi_1/2}-e^{-\Phi_2/2}}_{L^2_\mu},
  \end{equation}
  where equality holds if and only if $\Phi_1=\Phi_2$ $\mu$-a.s.
  We now bound $\norm{e^{-\Phi_1/2}-e^{-\Phi_2/2}}_{L^2_\mu}$.
  If $\Phi_i\in L^2_\mu(\Theta,\R)$ and $\essinf_\mu \Phi_i\geq 0$ for $i=1,2$, then by \eqref{eq_Lipschitz_continuity_exp_function}
 \begin{equation*}
   \Norm{e^{-\Phi_1/2}-e^{-\Phi_2/2}}_{L^2_\mu}^2\leq \int(e^{-\Phi_1}\vee e^{-\Phi_2})\frac{\abs{\Phi_1-\Phi_2}^2}{4}\rd \mu \leq \frac{\Norm{\Phi_1-\Phi_2}_{L^2_\mu}^{2}}{4}.
 \end{equation*}
 Alternatively, using the inequality
 \begin{equation*}
 (\sqrt{a}-\sqrt{b})^2=a-2\sqrt{ab}+b\leq a-2(a\wedge b) +b=\abs{a-b},\quad a,b\geq 0,
\end{equation*}
it follows from \eqref{eq_Lipschitz_continuity_exp_function} and $\essinf_\mu\Phi_i\geq 0$, $i=1,2$, that 
 \begin{align*}
   \Norm{e^{-\Phi_1/2}-e^{-\Phi_2/2}}_{L^2_\mu}^2\leq&\Norm{e^{-\Phi_1}-e^{-\Phi_2}}_{L^1_\mu} \leq  \int (e^{-\Phi_1}\vee e^{-\Phi_2})\Abs{\Phi_1-\Phi_2}\rd \mu 
   \\
   \leq & \Norm{\Phi_1-\Phi_2}_{L^1_\mu}.
 \end{align*}
We conclude that 
\begin{equation*}
 \Norm{e^{-\Phi_1/2}-e^{-\Phi_2/2}}_{L^2_\mu}\leq \Norm{\Phi_1-\Phi_2}_{L^1_\mu}^{1/2}\wedge \frac{\Norm{\Phi_1-\Phi_2}_{L^2_\mu}}{2}.
\end{equation*}
Substituting the above bound into \eqref{eq_bound00} yields
\begin{equation*}
  d_{\Hel}(\mu_{\Phi_1},\mu_{\Phi_2})\leq \frac{(2\Norm{\Phi_1-\Phi_2}_{L^1_\mu}^{1/2})\wedge \Norm{\Phi_1-\Phi_2}_{L^2_\mu}}{Z_{\Phi_1,\mu}^{1/2}\vee Z_{\Phi_2,\mu}^{1/2}},
\end{equation*}
as desired.
\end{proof}

\subsection{Proofs of Hellinger bounds for prior perturbations}
\label{section_Hellingerbounds_Priorperturbations_proofs}

\HellingerboundsPriorperturbationsTriangle*
\begin{proof}[Proof of \Cref{proposition_Hellinger_bounds_prior_perturbation_via_triangle_inequality}]
 
Let $\Phi\in L^0(\Theta,\R)$ and $\nu\in\mathcal{P}(\Theta)$ satisfy $\mu_i\ll \nu$ and $(\mu_i)_{\Phi}\in\mathcal{P}(\Theta)$ for $i=1,2$.
 Then
 \begin{align}
  &d_{\Hel}((\mu_1)_\Phi,(\mu_2)_\Phi)
     \nonumber
     \\
  =& \Norm{\sqrt{\frac{\rd(\mu_1)_\Phi}{\rd\nu}}-\sqrt{\frac{\rd(\mu_2)_\Phi}{\rd\nu}}}_{L^2_\nu} & \text{by \eqref{eq_Hellinger_metric}}
     \nonumber
     \\
   =& \Norm{\sqrt{\frac{e^{-\Phi}}{Z_{\Phi,\mu_1}}\frac{\rd\mu_1}{\rd\nu}} -\sqrt{\frac{e^{-\Phi}}{Z_{\Phi,\mu_2}}\frac{\rd\mu_2}{\rd\nu}}}_{L^2_\nu} & \text{by \eqref{eq_radon_nikodym_derivative_different_priors}}
   \nonumber
   \\
   =&\Norm{\sqrt{\frac{e^{-\Phi}}{Z_{\Phi,\mu_1}}}\left(\sqrt{\frac{\rd\mu_1}{\rd\nu}}-\sqrt{\frac{\rd\mu_2}{\rd\nu}}\right)+\sqrt{\frac{e^{-\Phi}}{Z_{\Phi,\mu_1}}\frac{\rd\mu_2}{\rd\nu}} -\sqrt{\frac{e^{-\Phi}}{Z_{\Phi,\mu_2}}\frac{\rd\mu_2}{\rd\nu}}}_{L^2_\nu}
   \nonumber
   \\
   =&\Norm{\sqrt{\frac{e^{-\Phi}}{Z_{\Phi,\mu_1}}}\left(\sqrt{\frac{\rd\mu_1}{\rd\nu}}-\sqrt{\frac{\rd\mu_2}{\rd\nu}}\right)+\sqrt{\frac{e^{-\Phi}}{Z_{\Phi,\mu_2}}\frac{\rd\mu_2}{\rd\nu}}\left(\sqrt{\frac{Z_{\Phi,\mu_2}}{Z_{\Phi,\mu_1}}} -\sqrt{\frac{Z_{\Phi,\mu_1}}{Z_{\Phi,\mu_1}}}\right)}_{L^2_\nu}
   \nonumber
   \\
   =&\Norm{\sqrt{\frac{e^{-\Phi}}{Z_{\Phi,\mu_1}}}\left(\sqrt{\frac{\rd\mu_1}{\rd\nu}}-\sqrt{\frac{\rd\mu_2}{\rd\nu}}\right)+\sqrt{\ell_{\Phi,\mu_2}\frac{\rd\mu_2}{\rd\nu}} \frac{Z_{\Phi,\mu_2}^{1/2}-Z_{\Phi,\mu_1}^{1/2}}{Z_{\Phi,\mu_1}^{1/2} }}_{L^2_\nu}
   &\text{by \eqref{eq_likelihood_function}.} 
   \label{eq_Hellinger_metric_different_priors}
\end{align}

Proof of \cref{proposition_Hellinger_bound_prior_perturbation_upper_bound}: By \eqref{eq_likelihood_function}, $\norm{ (\ell_{\Phi,\mu_2}\tfrac{\rd\mu_2}{\rd\nu})^{1/2}}_{L^2_\nu}=1$.
If $\nu=\lambda \mu_1+(1-\lambda)\mu_2$ for $0<\lambda<1$, then by \Cref{lemma_esssup_essinf_different_measures_different_functions}, $\essinf_\nu \Phi=\essinf_{\mu_1}\Phi \wedge \essinf_{\mu_2}\Phi$.
Thus, by applying the triangle inequality to \eqref{eq_Hellinger_metric_different_priors}, 
\begin{align*}
 d_{\Hel}((\mu_1)_\Phi,(\mu_2)_\Phi)\leq & \Norm{\sqrt{\frac{e^{-\Phi}}{Z_{\Phi,\mu_1}}}\left(\sqrt{\frac{\rd\mu_1}{\rd\nu}}-\sqrt{\frac{\rd\mu_2}{\rd\nu}}\right)}_{L^2_\nu}+\Abs{\frac{Z_{\Phi,\mu_2}^{1/2}-Z_{\Phi,\mu_1}^{1/2}}{Z_{\Phi,\mu_1}^{1/2} }}
 \\
 \leq & \frac{\esssup_\nu \exp(-\tfrac{1}{2}\Phi)}{Z_{\Phi,\mu_1}^{1/2}}\Norm{\sqrt{\frac{\rd\mu_1}{\rd\nu}}-\sqrt{\frac{\rd\mu_2}{\rd\nu}}}_{L^2_\nu}+\Abs{\frac{Z_{\Phi,\mu_2}^{1/2}-Z_{\Phi,\mu_1}^{1/2}}{Z_{\Phi,\mu_1}^{1/2} }}
 \\
 \leq & \frac{d_{\Hel}(\mu_1,\mu_2)+\abs{Z_{\Phi,\mu_2}^{1/2}-Z_{\Phi,\mu_1}^{1/2}}}{Z_{\Phi,\mu_1}^{1/2}}.
\end{align*}
In the application of the triangle inequality, equality holds if and only if the vectors lie on the same ray, i.e. if and only if there exists $\xi\geq 0$ such that 
\begin{equation*}
 \sqrt{\frac{e^{-\Phi}}{Z_{\Phi,\mu_1}}}\left(\sqrt{\frac{\rd\mu_1}{\rd\nu}}-\sqrt{\frac{\rd\mu_2}{\rd\nu}}\right)=\xi \sqrt{\ell_{\Phi,\mu_2}\frac{\rd\mu_2}{\rd\nu}},\quad \text{$\nu$-a.s.}
\end{equation*}
A necessary condition for the preceding statement to hold is that $\sqrt{\tfrac{\rd\mu_1}{\rd\nu}}-\sqrt{\tfrac{\rd\mu_2}{\rd\nu}}$ is nonnegative, $\nu$-a.s.
By \Cref{lemma_L1norm_difference_of_unnormalised_likelihoods_minus_abs_diff_normalisation_constants_prior_perturbation}\ref{lemma_L1norm_difference_of_unnormalised_likelihoods_minus_abs_diff_normalisation_constants_prior_perturbation_item1}, $\sqrt{\tfrac{\rd\mu_1}{\rd\nu}}-\sqrt{\tfrac{\rd\mu_2}{\rd\nu}}$ is nonnegative, $\nu$-a.s. if and only if $d_{\TV}(\mu_1,\mu_2)=0$, i.e. $\mu_1=\mu_2$.
Conversely, if $\mu_1=\mu_2$, then $Z_{\Phi,\mu_2}=Z_{\Phi,\mu_1}$ and $(\mu_1)_\Phi=(\mu_2)_\Phi$.

By switching $\mu_1$ and $\mu_2$ in the upper bound for $d_{\Hel}((\mu_1)_\Phi,(\mu_2)_\Phi)$ above, we obtain another upper bound for $d_{\Hel}((\mu_1)_\Phi,(\mu_2)_\Phi)$ that differs only in that $Z_{\Phi_,\mu_1}^{1/2}$ is replaced with $Z_{\Phi_,\mu_2}^{1/2}$.
Combining these two inequalities and using that $Z_{\Phi,\mu_1}^{-1/2}\wedge Z_{\Phi,\mu_2}^{-1/2}=(Z_{\Phi,\mu_1}^{1/2}\vee Z_{\Phi,\mu_2}^{1/2})^{-1}$, we obtain the first inequality in \cref{proposition_Hellinger_bound_prior_perturbation_upper_bound}.

To obtain the second inequality in \cref{proposition_Hellinger_bound_prior_perturbation_upper_bound}, we use the Cauchy--Schwarz inequality:
   \begin{align*}
    \Abs{Z_{\Phi,\mu_2}^{1/2}-Z_{\Phi,\mu_1}^{1/2}}^2=&Z_{\Phi,\mu_2}-2Z_{\Phi,\mu_2}^{1/2}Z_{\Phi,\mu_1}^{1/2}+Z_{\Phi,\mu_1}
    \\
    \leq & \int \exp(-\Phi) \frac{\rd \mu_2}{\rd\nu}\rd \nu-2 \int \exp(-\Phi) \sqrt{\frac{\rd \mu_2}{\rd \nu}\frac{\rd\mu_1}{\rd\nu}}\rd \nu +\int \exp(-\Phi)\frac{\rd \mu_1}{\rd\nu}\rd \nu
    \\
    =& \int \exp(-\Phi)\left(\sqrt{\frac{\rd\mu_2}{\rd\nu}}-\sqrt{\frac{\rd\mu_1}{\rd\nu}} \right)^2 \rd \nu
    \\
    \leq& d_{\Hel}(\mu_2,\mu_1)^2.
   \end{align*}
   For the last inequality above, we used the hypothesis that $\essinf_{\mu_i}\Phi=0$ for $i=1,2$ and \Cref{lemma_esssup_essinf_different_measures_different_functions}\ref{item_esssup_essinf_different_measures} to conclude that $\esssup_\nu\exp(-\Phi)=1$ if $\nu=\lambda \mu_1+(1-\lambda)\mu_2$ for $0<\lambda<1$.
   In the above application of the Cauchy--Schwarz inequality, equality holds if and only if there exists $\zeta\in\R$ such that $\sqrt{\tfrac{\rd\mu_1}{\rd\nu}}=\zeta\sqrt{\tfrac{\rd\mu_2}{\rd\nu}}$ $\mu$-a.s. Since $\mu_1,\mu_2\in\mathcal{P}(\Theta)$, if such a $\lambda$ exists then $\lambda =1$ must hold, which is equivalent to $\mu_1=\mu_2$.
   This completes the proof of \cref{proposition_Hellinger_bound_prior_perturbation_upper_bound}.
   
   Proof of \cref{proposition_Hellinger_bound_prior_perturbation_lower_bound}: 
   Since $\mu_2\in\mathcal{P}(\Theta)$, we have $\norm{ \tfrac{\rd\mu_2}{\rd\nu}^{1/2}}_{L^2_\nu}=1$.
   By applying the reverse triangle inequality, we have
 \begin{align*}
    &d_{\Hel}((\mu_1)_\Phi,(\mu_2)_\Phi)
     \nonumber
     \\
  =& \Norm{\sqrt{\frac{\rd(\mu_1)_\Phi}{\rd\nu}}-\sqrt{\frac{\rd(\mu_2)_\Phi}{\rd\nu}}}_{L^2_\nu}= \Norm{\sqrt{\frac{e^{-\Phi}}{Z_{\Phi,\mu_1}}\frac{\rd\mu_1}{\rd\nu}} -\sqrt{\frac{e^{-\Phi}}{Z_{\Phi,\mu_2}}\frac{\rd\mu_2}{\rd\nu}}}_{L^2_\nu} & \text{by \eqref{eq_Hellinger_metric}, \eqref{eq_radon_nikodym_derivative_different_priors}}
     \nonumber
     \\
   =&\Norm{\sqrt{\frac{e^{-\Phi}}{Z_{\Phi,\mu_1}}}\left(\sqrt{\frac{\rd\mu_1}{\rd\nu}}-\sqrt{\frac{\rd\mu_2}{\rd\nu}}\right)+\sqrt{\frac{e^{-\Phi}}{Z_{\Phi,\mu_1}}\frac{\rd\mu_2}{\rd\nu}} -\sqrt{\frac{e^{-\Phi}}{Z_{\Phi,\mu_2}}\frac{\rd\mu_2}{\rd\nu}}}_{L^2_\nu}
   \nonumber
  \\
   =&\Norm{\sqrt{\frac{e^{-\Phi}}{Z_{\Phi,\mu_1}}}\left(\sqrt{\frac{\rd\mu_1}{\rd\nu}}-\sqrt{\frac{\rd\mu_2}{\rd\nu}}\right)+\sqrt{\frac{e^{-\Phi}}{Z_{\Phi,\mu_1}}\frac{\rd\mu_2}{\rd\nu}} \frac{Z_{\Phi,\mu_2}^{1/2}-Z_{\Phi,\mu_1}^{1/2}}{Z_{\Phi,\mu_2}^{1/2} }}_{L^2_\nu} &
   \\
   \geq & \frac{\exp(-\tfrac{1}{2}\esssup_\nu\Phi)}{Z_{\Phi,\mu_1}^{1/2}}\Norm{\left(\sqrt{\frac{\rd\mu_1}{\rd\nu}}-\sqrt{\frac{\rd\mu_2}{\rd\nu}}\right)+\sqrt{\frac{\rd\mu_2}{\rd\nu}} \frac{Z_{\Phi,\mu_2}^{1/2}-Z_{\Phi,\mu_1}^{1/2}}{Z_{\Phi,\mu_2}^{1/2} }}_{L^2_\nu}
   \\
   \geq & \frac{\exp(-\tfrac{1}{2}\esssup_\nu\Phi)}{Z_{\Phi,\mu_1}^{1/2}}\Abs{\Norm{\sqrt{\frac{\rd\mu_1}{\rd\nu}}-\sqrt{\frac{\rd\mu_2}{\rd\nu}}}_{L^2_\nu}-\Norm{ \sqrt{\frac{\rd\mu_2}{\rd\nu}}}_{L^2_\nu}\Abs{\frac{Z_{\Phi,\mu_2}^{1/2}-Z_{\Phi,\mu_1}^{1/2}}{Z_{\Phi,\mu_2}^{1/2} }}}_{L^2_\nu}
   \\
   =&\frac{\exp(-\tfrac{1}{2}\esssup_\nu\Phi)}{Z_{\Phi,\mu_1}^{1/2}}\Abs{d_{\Hel}(\mu_1,\mu_2)-\Abs{\frac{Z_{\Phi,\mu_2}^{1/2}-Z_{\Phi,\mu_1}^{1/2}}{Z_{\Phi,\mu_2}^{1/2} }}},
 \end{align*}
 where the penultimate inequality follows from the reverse triangle inequality, and the last equation follows from \eqref{eq_Hellinger_metric} and the fact that $\Norm{ \sqrt{\frac{\rd\mu_2}{\rd\nu}}}_{L^2_\nu}=1$ since $\mu_2(\Theta)=1$.
 If equality holds, i.e. if 
 \begin{equation}
 \label{eq_intermediate0}
  d_{\Hel}((\mu_1)_\Phi,(\mu_2)_\Phi)=\frac{\exp(-\tfrac{1}{2}\esssup_\nu\Phi)}{Z_{\Phi,\mu_1}^{1/2}}\Abs{d_{\Hel}(\mu_1,\mu_2)-\Abs{\frac{Z_{\Phi,\mu_2}^{1/2}-Z_{\Phi,\mu_1}^{1/2}}{Z_{\Phi,\mu_2}^{1/2} }}},
 \end{equation}
 then equality must hold in the second inequality above.
 Since the second inequality follows by the reverse triangle inequality, equality holds if and only if there exists $\xi\geq 0$ such that 
\begin{equation*}
 \sqrt{\frac{\rd\mu_1}{\rd\nu}}-\sqrt{\frac{\rd\mu_2}{\rd\nu}}=\xi \sqrt{\frac{\rd\mu_2}{\rd\nu}},\quad \text{$\nu$-a.s.},
\end{equation*}
 which would in turn imply that $\sqrt{\frac{\rd\mu_1}{\rd\nu}}-\sqrt{\frac{\rd\mu_2}{\rd\nu}}$ is $\nu$-a.s. nonnegative. 
 By \Cref{lemma_L1norm_difference_of_unnormalised_likelihoods_minus_abs_diff_normalisation_constants_prior_perturbation}\ref{lemma_L1norm_difference_of_unnormalised_likelihoods_minus_abs_diff_normalisation_constants_prior_perturbation_item1}, $\sqrt{\tfrac{\rd\mu_1}{\rd\nu}}-\sqrt{\tfrac{\rd\mu_2}{\rd\nu}}$ is $\nu$-a.s. nonnegative if and only if $d_{\TV}(\mu_1,\mu_2)=0$, i.e. $\mu_1=\mu_2$.
 Conversely, if $\mu_1=\mu_2$, then both sides of \eqref{eq_intermediate0} equal zero.
 
 If $\nu=\lambda \mu_1+(1-\lambda)\mu_2$ for $0<\lambda<1$, then by \Cref{lemma_esssup_essinf_different_measures_different_functions}, $\esssup_\nu \Phi=\esssup_{\mu_1}\Phi \vee \esssup_{\mu_2}\Phi$, and thus $\exp(-\tfrac{1}{2}\esssup_\nu\Phi)=\exp(-\tfrac{1}{2}\norm{\Phi}_{L^\infty_{\mu_1}})\wedge \exp(-\tfrac{1}{2}\norm{\Phi}_{L^\infty_{\mu_2}})$.
 Substituting this into the lower bound for $d_{\Hel}((\mu_1)_\Phi,(\mu_2)_\Phi)$ above completes the proof of \cref{proposition_Hellinger_bound_prior_perturbation_lower_bound}.
\end{proof}

\section{Proof of Kullback--Leibler bounds}
\label{section_KLbounds_proofs}

For the proof of \Cref{proposition_KL_bounds_misfit_and_prior_perturbations}, we shall use the following bounds.
\begin{proposition}[{\cite[Proposition 6.1.7]{GineNickl2016}}]
 \label{proposition_bounds_for_KL}
 Let $(\mathcal{X},\mathcal{A})$ be a measurable space, $\nu$ be a measure on $(\mathcal{X},\mathcal{A})$, $p$ and $q$ be probability densities with respect to $\nu$, and $P,Q\in\mathcal{P}(\mathcal{X})$ satisfy $\tfrac{\rd P}{\rd \nu}=p$, $\tfrac{\rd Q}{\rd \nu}=q$.
 \begin{enumerate}
  \item \label{proposition_bounds_for_KL_item1}
  It holds that $d_{\TV}(P,Q)\leq \sqrt{d_{\KL}(P\Vert Q)/2}$. 
  \item \label{proposition_bounds_for_KL_item2}
  If $P\ll Q$, then 
 \begin{equation*}
    \int_{pq>0} \Abs{\log \frac{p}{q}}p\rd \nu\leq d_{\KL}(P\Vert Q)+\sqrt{2 d_{\KL}(P\Vert Q)}.
  \end{equation*}  
 \end{enumerate}
\end{proposition}

\KLboundsJointperturbations*
\begin{proof}[Proof of \Cref{proposition_KL_bounds_misfit_and_prior_perturbations}]
By \eqref{eq_posterior_function} and \eqref{eq_likelihood_function}, $\tfrac{\rd (\mu_i)_{\Phi_i}}{\rd \mu_i}=\tfrac{\exp(-\Phi_i)}{Z_{\Phi_i,\mu_i}}$ for $i=1,2$.
Since $\Phi_2 \in L^0(\Theta,\R)$, it follows that $0<\exp(-\Phi_2)<\infty$ on $\Theta$.
Hence, $\mu_2\ll (\mu_2)_{\Phi_2}$, and by the hypothesis that $\mu_1\ll\mu_2$, we have
\begin{subequations}
 \begin{align}
 \log \frac{\rd (\mu_1)_{\Phi_1}}{\rd (\mu_2)_{\Phi_2}}=&\log \frac{\rd (\mu_1)_{\Phi_1}}{\rd\mu_1}+\log\frac{\rd\mu_1}{\rd\mu_2}+\log \frac{\rd\mu_2}{\rd (\mu_2)_{\Phi_2}}
 \nonumber
 \\
 =&-\Phi_1-\log Z_{\Phi_1,\mu_1}+\log\frac{\rd\mu_1}{\rd\mu_2}+\Phi_2+\log Z_{\Phi_2,\mu_2},
 \label{eq_log_RadonNikodymDerivative}
 \\
 d_{\KL}((\mu_1)_{\Phi_1}\Vert (\mu_2)_{\Phi_2})=&\int \log \frac{\rd (\mu_1)_{\Phi_1}}{\rd (\mu_2)_{\Phi_2}} \rd (\mu_1)_{\Phi_1}
 \nonumber
 \\
 =&\int \Phi_2-\Phi_1 ~\rd (\mu_1)_{\Phi_1}+\log \frac{Z_{\Phi_2,\mu_2}}{Z_{\Phi_1,\mu_1}}+ \int \log\frac{\rd\mu_1}{\rd\mu_2}~\rd (\mu_1)_{\Phi_1}.
 \label{eq_KL_divergence_perturbed_misfit_and_prior}
\end{align}
\end{subequations}
We now prove the statements of \Cref{proposition_KL_bounds_misfit_and_prior_perturbations}.
Recall that for a real function $f$, $f_+\coloneqq \max\{0,f\}$ denotes the positive part of $f$.

Proof of \cref{proposition_KL_bounds_misfit_and_prior_perturbations_upper_bounds}:
We have 
\begin{align*}
 &d_{\KL}((\mu_1)_{\Phi_1}\Vert (\mu_2)_{\Phi_2})
\\
\leq & \int \Abs{\Phi_2-\Phi_1 }~\rd (\mu_1)_{\Phi_1}+\Abs{\log \frac{Z_{\Phi_2,\mu_2}}{Z_{\Phi_1,\mu_1}}}+ \int \Abs{\log\frac{\rd\mu_1}{\rd\mu_2}} ~\rd (\mu_1)_{\Phi_1} & \text{by \eqref{eq_KL_divergence_perturbed_misfit_and_prior}}
\\
\leq & \frac{1}{Z_{\Phi_1,\mu_1}} \Norm{\Phi_2-\Phi_1}_{L^1_{\mu_1}} + \Abs{\log \frac{Z_{\Phi_2,\mu_2}}{Z_{\Phi_1,\mu_1}}}+\frac{1}{Z_{\Phi_1,\mu_1}} \int \Abs{\log\frac{\rd\mu_1}{\rd\mu_2}} ~\rd \mu_1 & \text{by \eqref{eq_posterior_function}}
\\
\leq & \frac{1}{Z_{\Phi_1,\mu_1}} \Norm{\Phi_2-\Phi_1}_{L^1_{\mu_1}} + \Abs{\log \frac{Z_{\Phi_2,\mu_2}}{Z_{\Phi_1,\mu_1}}}+\frac{1}{Z_{\Phi_1,\mu_1}} \left( d_{\KL}(\mu_1\Vert \mu_2)+\sqrt{2d_{\KL}(\mu_1\Vert \mu_2)}\right) .
\end{align*}
In the second inequality, we use the hypothesis that $\essinf_{\mu_1}\Phi_1=0$.
In the third inequality, we use the hypothesis that $\mu_1\ll\mu_2$ and apply \Cref{proposition_bounds_for_KL} with $p\leftarrow \tfrac{\rd\mu_1}{\rd\mu_2}$, $q\leftarrow \tfrac{\rd\mu_2}{\rd\mu_2}=1$, and $\nu\leftarrow \mu_2$.

By the triangle inequality,
\begin{equation*}
 \Abs{\log \frac{Z_{\Phi_2,\mu_2}}{Z_{\Phi_1,\mu_1}}}\leq \Abs{\log \frac{Z_{\Phi_2,\mu_2}}{Z_{\Phi_1,\mu_2}}}+\Abs{\log \frac{Z_{\Phi_1,\mu_2}}{Z_{\Phi_1,\mu_1}}}.
 \end{equation*}
We bound the two terms on the right-hand side separately, using essentially the same argument that was used in the proof of \cite[Theorem 11]{Sprungk2020}: by the hypothesis that $\essinf_{\mu_2}\Phi_i=0$ for $i=1,2$, by local Lipschitz continuity \eqref{eq_Lipschitz_continuity_log_function} of the logarithm, and by \Cref{proposition_TV_bounds_misfit_perturbations_via_triangle_inequality}\ref{item_proposition_TV_bounds_misfit_perturbations_via_triangle_inequality_upper_bound} with $\mu\leftarrow \mu_2$,
 \begin{equation*}
   \Abs{\log \frac{Z_{\Phi_2,\mu_2}}{Z_{\Phi_1,\mu_2}}} \leq \frac{\Abs{Z_{\Phi_2,\mu_2}-Z_{\Phi_1,\mu_2}}}{Z_{\Phi_1,\mu_2}\wedge Z_{\Phi_2,\mu_2}}\leq \frac{\norm{\Phi_1-\Phi_2}_{L^1_{\mu_2}}}{Z_{\Phi_1,\mu_2}\wedge Z_{\Phi_2,\mu_2}}.
 \end{equation*}
 To bound the second term $\Abs{\log \tfrac{Z_{\Phi_1,\mu_2}}{Z_{\Phi_1,\mu_1}}}$ on the right-hand side, we again use \eqref{eq_Lipschitz_continuity_log_function}, and combine it with \Cref{proposition_TV_bounds_prior_perturbation_via_triangle_inequality}\ref{item_proposition_TV_bounds_prior_perturbation_via_triangle_inequality_upper_bound} and \Cref{proposition_bounds_for_KL}\ref{proposition_bounds_for_KL_item1}:
 \begin{align*}
  \Abs{\log \frac{Z_{\Phi_1,\mu_2}}{Z_{\Phi_1,\mu_1}}}\leq & \frac{\Abs{Z_{\Phi_1,\mu_2}-Z_{\Phi_1,\mu_1}}}{Z_{\Phi_1,\mu_1}\wedge Z_{\Phi_1,\mu_2}} \leq \frac{2d_{\TV}(\mu_1,\mu_2)}{Z_{\Phi_1,\mu_1}\wedge Z_{\Phi_1,\mu_2}}
  \leq \frac{\sqrt{2d_{\KL}(\mu_1\Vert \mu_2)}}{Z_{\Phi_1,\mu_1}\wedge Z_{\Phi_1,\mu_2}}.
 \end{align*}
  
Proof of \cref{proposition_KL_bounds_misfit_and_prior_perturbations_lower_bounds}: By the hypothesis that $(\mu_1)_{\Phi_1}\ll(\mu_2)_{\Phi_2}$, we may apply \Cref{proposition_bounds_for_KL}\ref{proposition_bounds_for_KL_item2} with $\nu\leftarrow (\mu_2)_{\Phi_2}$, $p\leftarrow \tfrac{\rd(\mu_1)_{\Phi_1}}{\rd(\mu_2)_{\Phi_2}}$, and $q\leftarrow 1$, to obtain
\begin{equation*}
 \int \Abs{ \log \frac{\rd(\mu_1)_{\Phi_1}}{\rd(\mu_2)_{\Phi_2}}} \rd (\mu_1)_{\Phi_1}\leq d_{\KL}((\mu_1)_{\Phi_1}\Vert (\mu_2)_{\Phi_2})+ \sqrt{2d_{\KL}((\mu_1)_{\Phi_1}\Vert (\mu_2)_{\Phi_2})}.
\end{equation*}
Now
\begin{align*}
 &\int \Abs{ \log \frac{\rd(\mu_1)_{\Phi_1}}{\rd(\mu_2)_{\Phi_2}}} \rd (\mu_1)_{\Phi_1}
 \\
 =&\int \Abs{ -\Phi_1-\log Z_{\Phi_1,\mu_1}+\log\frac{\rd\mu_1}{\rd\mu_2}+\Phi_2+\log Z_{\Phi_2,\mu_2}} \rd (\mu_1)_{\Phi_1} & \text{by \eqref{eq_log_RadonNikodymDerivative}}
 \\
 \geq & \frac{\exp(-\norm{\Phi_1}_{L^\infty_{\mu_1}})}{Z_{\Phi_1,\mu_1}} \int \Abs{ -\Phi_1-\log Z_{\Phi_1,\mu_1}+\log\frac{\rd\mu_1}{\rd\mu_2}+\Phi_2+\log Z_{\Phi_2,\mu_2}} \rd \mu_1 & \text{by \eqref{eq_posterior_function}}
 \\
 \geq & \frac{\exp(-\norm{\Phi_1}_{L^\infty_{\mu_1}})}{Z_{\Phi_1,\mu_1}} \Abs{\int  \left(-\Phi_1-\log Z_{\Phi_1,\mu_1}+\log\frac{\rd\mu_1}{\rd\mu_2}+\Phi_2+\log Z_{\Phi_2,\mu_2}\right) \rd \mu_1} 
 \\
 =&\frac{\exp(-\norm{\Phi_1}_{L^\infty_{\mu_1}})}{Z_{\Phi_1,\mu_1}} \Abs{\int  \Phi_2-\Phi_1~\rd \mu_1+\log Z_{\Phi_2,\mu_2}-\log Z_{\Phi_1,\mu_1} +d_{\KL}(\mu_1\Vert \mu_2)} ,
\end{align*}
where in the first inequality we use the hypothesis that $\Phi_1\in L^\infty_{\mu_1}$.
\end{proof}

\section{Proofs of 1-Wasserstein bounds}
\label{section_W1bounds_proofs}

We first recall some facts about Wasserstein metrics, in the setting where $X$ and $Y$ be Polish spaces, and $c:X\times Y \to \R$ is a lower semicontinuous cost function. 
Given $\phi:X \to [-\infty,\infty)$, the $c$-conjugate function $\phi^c$ is defined by $\phi^c(y)\coloneqq \inf_{x\in X} \{c(x,y)-\phi(x)\}$.
Similarly, given $\psi:Y\to [-\infty,\infty)$, the $c$-conjugate function $\psi^c$ is defined by $\psi^c(x)\coloneqq \inf_{y\in Y}\{c(x,y)-\psi(y)\}$. See e.g. \cite[Definition 3.12]{Ambrosio2021}.
Recall that a function $\phi:X\to [-\infty,\infty)$ (respectively, $\psi:Y\to [-\infty,\infty)$) is $c$-concave if and only if $\phi=\psi^c$ for some $\psi:Y\to [-\infty,\infty)$ (respectively, $\psi=\phi^c$ for some $\phi:X\to [-\infty,\infty)$); see \cite[Definition 3.13]{Ambrosio2021} and the text below it.

\begin{theorem}[{\cite[Theorem 4.2(ii)]{Ambrosio2021}}]
\label{theorem_Ambrosio_4_2}
 Let $\mu\in\mathcal{P}(X)$ and $\nu\in\mathcal{P}(Y)$. If $c:X\times Y\to [0,\infty]$ is lower semicontinuous and if there exists $a\in L^1_\mu(X)$ and $b\in L^1_\nu(Y)$ such that $c(x,y)\leq a(x)+b(y)$, then there exists a $c$-concave function $\phi:X\to [-\infty,\infty)$ such that $\phi\in L^1_\mu(X)$, $\phi^c\in L^1_\nu(Y)$, and 
 \begin{align*}
  &\min\left\{\int_{X\times Y} c(x,y)\rd \pi(x,y)\ :\ \pi\in\Pi(\mu,\nu)\right\}=\int_X \phi\rd \mu+\int_Y \phi^c\rd\nu.
 \end{align*}
\end{theorem}
\begin{lemma}[{\cite[pp.39-40]{Ambrosio2021}}]
 \label{lemma_c_concavity_for_specific_costs}
If $X=Y$, $(X,d)$ is a metric space, and $c(x,y)=d(x,y)$, then $\phi$ is $c$-concave if and only if $\phi$ is 1-Lipschitz. In this case, $\phi^c=-\phi$.
\end{lemma}
\Cref{theorem_Ambrosio_4_2} and \Cref{lemma_c_concavity_for_specific_costs} yield the Kantorovich--Rubinstein duality formula \eqref{eq_Kantorovich_Rubinstein_duality_formula}, as well as the following corollary.
Recall the definition of the $p$-Wasserstein space $\mathcal{P}_p(\Theta)$ in \eqref{eq_Wasserstein_space}.
\begin{corollary}
 \label{corollary_attainment_of_supremum_in_Kantorovich_Duality}
 If $\mu,\nu\in\mathcal{P}_1(X)$, then there exists a 1-Lipschitz function $f_\opt:X\to\R$ such that
 \begin{equation*}
 \Was_1(\mu,\nu)  =\int_X f_\opt \rd \mu-\int_X f_\opt \rd \nu.
 \end{equation*}
\end{corollary}
\begin{proof}[Proof of \Cref{corollary_attainment_of_supremum_in_Kantorovich_Duality}]
 Fix an arbitrary $x_0\in X$.
 For every $x,y\in X$, let $a(x)\coloneqq d(x_0,x)$, $b(y)\coloneqq  d(x_0,y)$, and $c(x,y)\coloneqq d(x,y)$.
For this choice of $c$, \Cref{lemma_c_concavity_for_specific_costs} implies that a $c$-concave function is 1-Lipschitz.
Since $\mu,\nu\in\mathcal{P}_1$, it follows from \eqref{eq_Wasserstein_space} that $a\in L^1_\mu(X)$ and $b\in L^1_\nu(X)$, and $c(x,y)\leq a(x)+b(y)$ follows from the triangle inequality. 
Now we may apply \Cref{theorem_Ambrosio_4_2} with $Y\leftarrow X$, $c\leftarrow d$, and $a$ and $b$ as above, to conclude that there exists a 1-Lipschitz function $f_\opt$ that satisfies the desired equation.
\end{proof}

The following properties of the function class in \eqref{eq_Kantorovich_Rubinstein_duality_formula} were stated in the proofs of \cite[Theorem 14, Theorem 15]{Sprungk2020}. 
\begin{remark}[Duality formula]
\label{remark_duality_formula}
\begin{enumerate}
\item \label{remark_duality_formula_item1} Fix an arbitrary $x_0\in\Theta$. 
 For any 1-Lipschitz function $f:\Theta\to\R$, the function $g:\Theta\to\R$ defined by $g(x)\coloneqq f(x)-f(x_0)$ is 1-Lipschitz and satisfies $g(x_0)=0$. Moreover, 
 \begin{equation*}
  \int f \rd \mu-\int f \rd \nu=\int f-f(x_0)\rd \mu-\int f-f(x_0)\rd \nu
 \end{equation*}
since $f(x_0)$ is constant. Thus, we may restrict the supremum in \eqref{eq_Kantorovich_Rubinstein_duality_formula} to the subset of 1-Lipschitz functions vanishing at $x_0$, which yields \eqref{eq_Kantorovich_Rubinstein_duality_formula_fixed_x0}.
\item \label{remark_duality_formula_item2} If for some $x_0\in\Theta$ it holds that $f(x_0)=0$ and $\norm{f}_{\Lip}\leq 1$, then 
\begin{equation}
 \label{eq_upper_bound_1_Lipschitz_function_by_metric}
 \abs{f(x)}=\abs{f(x)-f(x_0)}\leq d(x,x_0).
\end{equation}
In particular, if $(\Theta,d)$ is a bounded metric space, i.e. if $\sup_{x,y\in \Theta}d(x,y)<\infty$, then any $g\in\{f:\Theta\to\R\ :\ \norm{f}_{\Lip}\leq 1,\ f(x_0)=0\}$ satisfies $\norm{g}_{\infty}\leq \sup_{x,y\in \Theta}d(x,y)$.
\end{enumerate}
\end{remark}

\subsection{Proofs of 1-Wasserstein bounds for misfit perturbations}
\label{section_W1bounds_misfit_perturbations_proofs}

\WassersteinboundsMisfitperturbations*
\begin{proof}[Proof of \Cref{proposition_W1_bounds_misfit_perturbations}]
The proof of the upper bound follows the same strategy as the proof of \cite[Theorem 14]{Sprungk2020}.
First observe that
\begin{align*}
 \int f \rd \mu_{\Phi_1}-\int f \rd \mu_{\Phi_2} =& \int f\left(  \frac{e^{-\Phi_1}}{Z_{\Phi_1,\mu}}-\frac{e^{-\Phi_2}}{Z_{\Phi_2,\mu}} \right)\rd \mu 
 \\
 =&   \int f\left(  \frac{e^{-\Phi_1}-e^{-\Phi_2}}{Z_{\Phi_1,\mu}}+e^{-\Phi_2}\left(\frac{1}{Z_{\Phi_1,\mu}}-\frac{1}{Z_{\Phi_2,\mu}}\right) \right)\rd \mu
 \\
 =&\frac{1}{Z_{\Phi_1,\mu}} \int f\left(  e^{-\Phi_1}-e^{-\Phi_2}+\frac{e^{-\Phi_2}}{Z_{\Phi_2,\mu}}(Z_{\Phi_2,\mu}-Z_{\Phi_1,\mu}) \right)\rd \mu
 \\
 =& \frac{1}{Z_{\Phi_1,\mu}} \int f\left(  e^{-\Phi_1}-e^{-\Phi_2}\right)\rd \mu+\int f \rd \mu_{\Phi_2}(Z_{\Phi_2,\mu}-Z_{\Phi_1,\mu}) .
\end{align*}

Proof of \cref{proposition_W1_bounds_misfit_perturbations_upper_bound}: Let $x_0\in\Theta$ be arbitrary.
By the Kantorovich--Rubinstein duality formula \eqref{eq_Kantorovich_Rubinstein_duality_formula_fixed_x0},
\begin{equation}
\label{eq_W1_equation00}
 \Was_1(\mu_{\Phi_1},\mu_{\Phi_2})=\frac{1}{Z_{\Phi_1,\mu}}\sup_{\norm{f}_{\Lip}\leq 1,f(x_0)=0}\Abs{ \int f\left(  e^{-\Phi_1}-e^{-\Phi_2}\right)\rd \mu+\int f \rd \mu_{\Phi_2}(Z_{\Phi_2,\mu}-Z_{\Phi_1,\mu}) }.
\end{equation}

  Recall from \eqref{eq_radius_parameter} that $R(\mu)=\diam{\supp{\mu}}$ and the hypothesis that $R(\mu)$ is finite.
  By definition \eqref{eq_Lipschitz_constant_and_supremum_norm} of the Lipschitz constant, $\norm{f}_{\Lip(\supp{\mu},d)}\leq \norm{f}_{\Lip(\Theta,d)}$.
 By \Cref{remark_duality_formula}\ref{remark_duality_formula_item2}, if $\norm{f}_{\Lip}\leq 1$ and $f(x_0)=0$, then for every $x\in\Theta$, $\abs{f(x)}\leq d(x,x_0)$. 
 Applying this observation with $f$ replaced by $f\mathbb{I}_{\supp{\mu}}$ and the definition \eqref{eq_radius_parameter} of $R(\mu)$, we obtain for arbitrary $x_0\in\supp{\mu}$ that
 \begin{equation}
 \label{eq_supremum_norm_of_Lipschitz_f_subject_to_f_x0_equals_0}
  \Norm{f\mathbb{I}_{\supp{\mu}}}_\infty\leq \sup_{x\in\supp{\mu}} d(x,x_0)\leq R(\mu).
 \end{equation}
 We now bound the term inside the supremum on the right-hand side of \eqref{eq_W1_equation00}:
 \begin{align*}
 &\Abs{\int f\left(  e^{-\Phi_1}-e^{-\Phi_2}\right)\rd \mu+(Z_{\Phi_2,\mu}-Z_{\Phi_1,\mu})\int f \rd \mu_{\Phi_2}}
 \\
 \leq &\Abs{\int f\left(  e^{-\Phi_1}-e^{-\Phi_2}\right)\rd \mu}+\abs{Z_{\Phi_2,\mu}-Z_{\Phi_1,\mu}}\Abs{\int f \rd \mu_{\Phi_2}}
 \\
 \leq & R(\mu)\Norm{ e^{-\Phi_1}-e^{-\Phi_2}}_{L^1_\mu} +R(\mu)\Abs{Z_{\Phi_2,\mu}-Z_{\Phi_1,\mu}}  & \text{by \eqref{eq_supremum_norm_of_Lipschitz_f_subject_to_f_x0_equals_0}.}
 \end{align*}
  We thus obtain
 \begin{equation*}
  \Was_1(\mu_{\Phi_1},\mu_{\Phi_2})\leq \frac{R(\mu)}{Z_{\Phi_1,\mu}}\left(\Norm{ e^{-\Phi_1}-e^{-\Phi_2}}_{L^1_\mu} +\Abs{Z_{\Phi_2,\mu}-Z_{\Phi_1,\mu}}\right)
 \end{equation*}
 and by switching $\Phi_1$ and $\Phi_2$, we obtain an analogous inequality. 
 Using that $Z_{\Phi,\mu_1}^{-1}\wedge Z_{\Phi,\mu_2}^{-1}=(Z_{\Phi,\mu_1}\vee Z_{\Phi,\mu_2})^{-1}$ implies the conclusion of \cref{proposition_W1_bounds_misfit_perturbations_upper_bound}. 
 If $\Phi_1-\Phi_2$ is $\mu$-a.s. constant, then the left-hand side of the inequality above vanishes, since $\mu_{\Phi_1}=\mu_{\Phi_2}$ by \Cref{lemma_noninjectivity_of_dataMisfit_maps}\ref{item_lemma_noninjectivity_of_dataMisfit_maps_equivalent_condition_for_agreement}, but the right-hand side vanishes  if and only if $\Phi_1=\Phi_2$ $\mu$-a.s. 
 If in addition $\Phi_i\in L^1_\mu(\Theta,\R)$ for $i=1,2$, then by \eqref{eq_Lipschitz_continuity_exp_function} it follows that $  \Norm{ e^{-\Phi_1}-e^{-\Phi_2}}_{L^1_\mu}\leq \Norm{\Phi_1-\Phi_2}_{L^1_\mu}$, and thus
 \begin{equation*}
  \Was_1(\mu_{\Phi_1},\mu_{\Phi_2})\leq \frac{R(\mu)}{Z_{\Phi_1,\mu}}\left(\Norm{\Phi_1-\Phi_2}_{L^1_\mu} +\Abs{Z_{\Phi_2,\mu}-Z_{\Phi_1,\mu}}\right).
 \end{equation*}
 
 Proof of \cref{proposition_W1_bounds_misfit_perturbations_lower_bound}: By the Kantorovich--Rubinstein duality formula \eqref{eq_Kantorovich_Rubinstein_duality_formula}, 
 \begin{equation*}
 \Was_1(\mu_{\Phi_1},\mu_{\Phi_2})=\frac{1}{Z_{\Phi_1,\mu}}\sup_{\norm{f}_{\Lip}\leq 1}\Abs{ \int f\left(  e^{-\Phi_1}-e^{-\Phi_2}\right)\rd \mu+\int f \rd \mu_{\Phi_2}(Z_{\Phi_2,\mu}-Z_{\Phi_1,\mu}) }.
\end{equation*}
 Given the hypothesis that $\mu_{\Phi_i}\in\mathcal{P}_1(\Theta)$, it follows by \Cref{corollary_attainment_of_supremum_in_Kantorovich_Duality} that the supremum above is attained.
 Hence, the supremum in \eqref{eq_W1_equation00} is attained at some 1-Lipschitz function $f_\opt:\Theta\to\R$.
If $Z_{\Phi_1,\mu}=Z_{\Phi_2,\mu}$, then by the hypothesis that $0<\norm{e^{-\Phi_1}-e^{-\Phi_2}}_{\Lip(\supp{\mu},d)}<\infty$, it follows that the function $(e^{-\Phi_1}-e^{-\Phi_2})\mathbb{I}_{\supp{\mu}}/\norm{e^{-\Phi_1}-e^{-\Phi_2}}_{\Lip(\supp{\mu},d)}$ has unit Lipschitz norm.
Thus,
\begin{align*}
 \sup_{\norm{f}_{\Lip}\leq 1} \Abs{ \int f (e^{-\Phi_1}-e^{-\Phi_2})\rd \mu}\geq & \Abs{ \int \frac{e^{-\Phi_1}-e^{-\Phi_2}}{\norm{e^{-\Phi_1}-e^{-\Phi_2}}_{\Lip(\supp{\mu},d)}} (e^{-\Phi_1}-e^{-\Phi_2})\rd \mu}
 \\
 =&\frac{\Norm{e^{-\Phi_1}-e^{-\Phi_2}}_{L^2_\mu}^{2}}{\norm{e^{-\Phi_1}-e^{-\Phi_2}}_{\Lip(\supp{\mu},d)}},
\end{align*}
and thus
\begin{equation*}
 \Was_1(\mu_{\Phi_1},\mu_{\Phi_2})\geq \frac{1}{Z_{\Phi_1,\mu}}\frac{\Norm{e^{-\Phi_1}-e^{-\Phi_2}}_{L^2_\mu}^{2}}{\norm{e^{-\Phi_1}-e^{-\Phi_2}}_{\Lip(\supp{\mu},d)}}.
\end{equation*}
If in addition $\Phi_i\in L^\infty_\mu(\Theta,\R)$ for $i=1,2$, then by \eqref{eq_Lipschitz_continuity_exp_function}, we have 
\begin{equation*}
 \Norm{e^{-\Phi_1}-e^{-\Phi_2}}_{L^2_\mu}^{2}\geq \left(\exp(-2\norm{\Phi_1}_{L^\infty_\mu})\wedge \exp(-2\norm{\Phi_2}_{L^\infty_\mu})\right)\norm{\Phi_1-\Phi_2}_{L^2_\mu}^{2},
\end{equation*}
which completes the proof of \Cref{proposition_W1_bounds_misfit_perturbations}.
\end{proof}

\subsection{Proofs of 1-Wasserstein bounds for prior perturbations}
\label{section_W1bounds_prior_perturbations_proofs}

\WassersteinUpperboundPriorperturbations*
\begin{proof}[Proof of \Cref{proposition_W1_upper_bound_prior_perturbations}]
Proof of \cref{proposition_W1_upper_bound_prior_perturbations_normalisation_constant}:  
  By the definition \eqref{eq_normalisationConstant_function} of $Z_{\Phi,\mu_i}$,
 \begin{align*}
  \abs{Z_{\Phi,\mu_1}-Z_{\Phi,\mu_2}}=\Abs{ \int e^{-\Phi}\rd \mu_1-\int e^{-\Phi}\rd \mu_2}\leq & \norm{e^{-\Phi}}_{\Lip(\supp{\mu_1+\mu_2},d)} \Was_1(\mu_2,\mu_2).
 \end{align*}
To prove the inequality, we distinguish between two cases.
The Lipschitz constant of $e^{-\Phi}$ on $\supp{\mu_1+\mu_2}$ vanishes if and only if $\Phi$ is constant on $\supp{\mu_1+\mu_2}$.
In this case, $Z_{\Phi,\mu_1}=Z_{\Phi,\mu_2}$ by \eqref{eq_normalisationConstant_function}, and the inequality that we seek to prove reduces to the equation $0=0$.
If the Lipschitz constant of $e^{-\Phi}$ on $\supp{\mu_1+\mu_2}$ is strictly positive, then $e^{-\Phi}/\norm{e^{-\Phi}}_{\Lip(\supp{\mu_1+\mu_2},d)}$ is a 1-Lipschitz function on $\supp{\mu_1+\mu_2}$, and the inequality that we seek to prove follows by the duality formula \eqref{eq_Kantorovich_Rubinstein_duality_formula}. 

By the hypothesis that $(\mu_i)_\Phi\in\mathcal{P}_1(\Theta)$ for $i=1,2$, it follows from \Cref{corollary_attainment_of_supremum_in_Kantorovich_Duality} that there exists a 1-Lipschitz function $f_\opt$ on $\Theta$ such that 
\begin{equation*}
 \Was_1((\mu_1)_\Phi,(\mu_2)_\Phi)=\Abs{\int f_\opt \rd (\mu_1)_\Phi-\int f_\opt \rd (\mu_2)_\Phi}.
\end{equation*}
Since $\supp{\mu_1+\mu_2}\subseteq\Theta$, it follows that $f_\opt$ is 1-Lipschitz on $\supp{\mu_1+\mu_2}$, by definition of the Lipschitz constant.
 
 Proof of \cref{proposition_W1_upper_bound_prior_perturbations_zero_Lipschitz_constant_case}: Suppose that $\norm{f_\opt e^{-\Phi}}_{\Lip(\supp{\mu_1+\mu_2},d)}=0$. 
 By the definition of the Lipschitz constant, this is equivalent to the existence of some $\lambda\in\R$ such that $f_\opt e^{-\Phi}=\lambda$ on $\supp{\mu_1+\mu_2}$.
 Since the supremum in the duality formula \eqref{eq_Kantorovich_Rubinstein_duality_formula} is attained at $f_\opt$, we have
 \begin{align*}
  \Was_1((\mu_1)_\Phi,(\mu_2)_\Phi)=&\Abs{\int f_\opt \frac{e^{-\Phi}}{Z_{\Phi,\mu_1}} \rd \mu_1 -\int f_\opt \frac{e^{-\Phi}}{Z_{\Phi,\mu_2}} \rd \mu_2} & \text{by \eqref{eq_posterior_function}}
  \\
  =&\Abs{ \int  \frac{\lambda}{Z_{\Phi,\mu_1}} \rd \mu_1 -\int \frac{\lambda}{Z_{\Phi,\mu_2}} \rd \mu_2} & f_\opt e^{-\Phi}=\lambda
  \\
  =&\frac{\Abs{\lambda(Z_{\Phi,\mu_2}-Z_{\Phi,\mu_1})}}{Z_{\Phi,\mu_1}Z_{\Phi,\mu_2}} & \mu_1,\mu_2\in\mathcal{P}(\Theta)
  \\
  \leq & \frac{\abs{\lambda}}{Z_{\Phi,\mu_1}Z_{\Phi,\mu_2}}\norm{e^{-\Phi}}_{\Lip(\supp{\mu_1+\mu_2},d)}\Was_1(\mu_1,\mu_2) & \text{by \cref{proposition_W1_upper_bound_prior_perturbations_normalisation_constant}.}
\end{align*}

 Proof of \cref{proposition_W1_upper_bound_prior_perturbations_nonzero_Lipschitz_constant_case}: For brevity, let $\mathcal{F}_1\coloneqq \{f:\Theta\to\R\ :\ \norm{f}_{\Lip(\supp{\mu_1+\mu_2},d)}\leq 1,\ f(x_0)=0\}$, where $x_0\in\Theta$ is fixed.
 We specify $x_0$ further below.
 
 For any $f\in\mathcal{F}_1$,
\begin{align}
&\int f \rd (\mu_1)_\Phi-\int f \rd (\mu_2)_\Phi
=\int f  \frac{e^{-\Phi}}{Z_{\Phi,\mu_1}}\rd\mu_1-\int f\frac{e^{-\Phi}}{Z_{\Phi,\mu_2}}\rd\mu_2
& \text{by \eqref{eq_posterior_function}}
\nonumber
\\
=&\int f  \frac{e^{-\Phi}}{Z_{\Phi,\mu_1}}\rd\mu_1-\int f  \frac{e^{-\Phi}}{Z_{\Phi,\mu_1}}\rd\mu_2+\int f  \frac{e^{-\Phi}}{Z_{\Phi,\mu_1}}\rd\mu_2-\int f\frac{e^{-\Phi}}{Z_{\Phi,\mu_2}}\rd\mu_2
\nonumber
\\
=&\frac{1}{Z_{\Phi,\mu_1}}\left(\int f e^{-\Phi}\rd\mu_1-\int f  e^{-\Phi}\rd\mu_2\right)+Z_{\Phi,\mu_2}\left(\frac{1}{Z_{\Phi,\mu_1}}-\frac{1}{Z_{\Phi,\mu_2}}\right)\int f  \frac{e^{-\Phi}}{Z_{\Phi,\mu_2}}\rd\mu_2
\nonumber
\\
 =&\frac{1}{Z_{\Phi,\mu_1}} \left(\int f e^{-\Phi}\rd \mu_1-\int f e^{-\Phi}\rd \mu_2+(Z_{\Phi,\mu_2}-Z_{\Phi,\mu_1})\int f \rd(\mu_2)_\Phi\right).
 \label{eq_W1_equation00_prior_perturbations}
\end{align}
By \Cref{proposition_W1_upper_bound_prior_perturbations}\ref{proposition_W1_upper_bound_prior_perturbations_normalisation_constant}, there exists $f_\opt \in\mathcal{F}_1$ such that
\begin{equation*}
\Was_1((\mu_1)_\Phi,(\mu_2)_\Phi)= \frac{1}{Z_{\Phi,\mu_1}} \left(\int f_\opt e^{-\Phi}\rd \mu_1-\int f_\opt e^{-\Phi}\rd \mu_2+(Z_{\Phi,\mu_2}-Z_{\Phi,\mu_1})\int f_\opt \rd(\mu_2)_\Phi\right).
\end{equation*}

 \textbf{Step 1: upper bound of} $ \Abs{\int f\rd (\mu_2)_\Phi }$.
 Let $x_0\in\supp{(\mu_1+\mu_2)}$ be arbitrary. 
  By \Cref{remark_duality_formula}\ref{remark_duality_formula_item2}, for every 1-Lipschitz function $f$ that satisfies $f(x_0)=0$, it follows that $\abs{f(x)}\leq d(x,x_0)$ for every $x\in\Theta$.
  Combining this fact with the fact $\supp{(\mu_1+\mu_2)}\supset\supp{\mu_2}$ and using the radius parameter from \eqref{eq_radius_parameter}, we have 
 \begin{equation*}
  \Abs{\int f\rd (\mu_2)_\Phi }=\Abs{\int f\mathbb{I}_{\supp{\mu_2}} \frac{e^{-\Phi}}{Z_{\Phi,\mu_2}}\rd \mu_2}\leq \int \Abs{f\mathbb{I}_{\supp{\mu_2}}} \frac{e^{-\Phi}}{Z_{\Phi,\mu_2}}\rd \mu_2\leq R(\mu_1+\mu_2)\cdot 1.
 \end{equation*}
\textbf{Step 2: upper bound of }$\Abs{\int f_\opt e^{-\Phi}\rd \mu_1-\int f_\opt e^{-\Phi}\rd \mu_2}$. 
\begin{align*}
&\Abs{\int f_\opt e^{-\Phi}\rd \mu_1-\int f_\opt e^{-\Phi}\rd \mu_2}
\\
=& \norm{f_\opt e^{-\Phi}}_{\Lip(\supp{\mu_1+\mu_2},d)}\Abs{\frac{\int  f_\opt e^{-\Phi}\rd \mu_1-\int  f_\opt e^{-\Phi}\rd \mu_2}{\norm{f_\opt e^{-\Phi}}_{\Lip(\supp{\mu_1+\mu_2},d)}}} 
  \\
  \leq & \sup_{f\in\mathcal{F}_1,\norm{fe^{-\Phi}}_{\Lip(\supp{\mu_1+\mu_2},d)}>0}\norm{fe^{-\Phi}}_{\Lip(\supp{\mu_1+\mu_2},d)} \Abs{\frac{\int  f e^{-\Phi}\rd \mu_1-\int  f e^{-\Phi}\rd \mu_2}{\norm{f e^{-\Phi}}_{\Lip(\supp{\mu_1+\mu_2},d)}}}
  \\
  \leq &\sup_{f\in\mathcal{F}_1}\norm{fe^{-\Phi}}_{\Lip(\supp{\mu_1+\mu_2},d)} \sup_{g\in\mathcal{F}_1}\Abs{\int g\rd \mu_1-\int g\rd \mu_2}
  \\
  \leq &\sup_{f\in\mathcal{F}_1}\norm{fe^{-\Phi}}_{\Lip(\supp{\mu_1+\mu_2},d)} \Was_1(\mu_1,\mu_2) 
\end{align*}
where the equation follows from the hypothesis that $ 0<\norm{f_\opt e^{-\Phi}}_{\Lip(\supp{\mu_1+\mu_2},d)}<\infty$, the first inequality follows since $f_\opt$ belongs to $\mathcal{F}_1$ and $\norm{f_\opt e^{-\Phi}}_{\Lip(\supp{\mu_1+\mu_2},d)}$ is positive, the second inequality follows from the fact that the supremum of the product is less than or equal to the product of the suprema, and the third inequality follows from the duality formula \eqref{eq_Kantorovich_Rubinstein_duality_formula_fixed_x0}.

\textbf{Step 3: upper bound of} $\sup_{f\in\mathcal{F}_1}\norm{fe^{-\Phi}}_{\Lip(\supp{\mu_1+\mu_2},d)}$.

For suitable functions $g,h:\Theta\to\R$, the definition of $\norm{\cdot}_{\Lip}$ and $\norm{\cdot}_\infty$ and the triangle inequality imply 
  \begin{equation}
  \label{eq_Lipschitz_norm_of_product}
   \norm{gh}_{\Lip}\leq \norm{g}_\infty\norm{h}_{\Lip}+\norm{h}_\infty\norm{g}_{\Lip}.
  \end{equation}
By replacing $\mu$ with $\mu_1+\mu_2$ in \eqref{eq_supremum_norm_of_Lipschitz_f_subject_to_f_x0_equals_0}, we obtain $\norm{f\mathbb{I}_{\supp{\mu_1+\mu_2}}}_\infty\leq R(\mu_1+\mu_2)$.
Since $f\in\mathcal{F}_1$ implies $\norm{f}_{\Lip(\supp{\mu_1+\mu_2,d})}\leq 1$, it follows that
 \begin{align*}
  \norm{fe^{-\Phi}}_{\Lip(\supp{\mu_1+\mu_2},d)} \leq& \norm{e^{-\Phi}\mathbb{I}_{\supp{\mu_1+\mu_2}}}_\infty +\norm{e^{-\Phi}}_{\Lip(\supp{\mu_1+\mu_2},d)} R(\mu_1+\mu_2)
  \\
  \leq & 1+\norm{e^{-\Phi}}_{\Lip(\supp{\mu_1+\mu_2},d)} R(\mu_1+\mu_2) ,
 \end{align*}
where the second inequality follows from the hypothesis that $\essinf_{\mu_i}\Phi=0$ for $i=1,2$. 
By combining Step 2 and 3, it follows that
 \begin{align*}
 \Abs{\int f_\opt e^{-\Phi}\rd \mu_1-\int f_\opt e^{-\Phi}\rd \mu_2}\leq \left(1+\norm{e^{-\Phi}}_{\Lip(\supp{\mu_1+\mu_2},d)} R(\mu_1+\mu_2)\right) \Was_1(\mu_1,\mu_2).
\end{align*}
 
 Applying the bounds from the above steps to the terms in \eqref{eq_W1_equation00_prior_perturbations}, we obtain
 \begin{align*}
   \Was_1((\mu_1)_\Phi,(\mu_2)_\Phi) \leq &
   \frac{1+\norm{e^{-\Phi}}_{\Lip(\supp{\mu_1+\mu_2},d)} R(\mu_1+\mu_2)}{Z_{\Phi,\mu_1}} \Was_1(\mu_1,\mu_2)
   \\
   &+\frac{\abs{Z_{\Phi,\mu_1}-Z_{\Phi,\mu_2}}}{Z_{\Phi,\mu_1}} R(\mu_1+\mu_2).
 \end{align*}
Switching $\mu_1$ and $\mu_2$ yields an analogous inequality with $Z_{\Phi,\mu_1}$ replaced by $Z_{\Phi,\mu_2}$.
Using that $Z_{\Phi,\mu_1}^{-1}\wedge Z_{\Phi,\mu_2}^{-1}=(Z_{\Phi,\mu_1}\vee Z_{\Phi,\mu_2})^{-1}$ proves the first inequality in \Cref{proposition_W1_upper_bound_prior_perturbations}\ref{proposition_W1_upper_bound_prior_perturbations_nonzero_Lipschitz_constant_case}.
The second inequality follows from the bound on $\abs{Z_{\Phi,\mu_1}-Z_{\Phi,\mu_2}}$ in  \cref{proposition_W1_upper_bound_prior_perturbations_normalisation_constant}.
\end{proof}

\WassersteinLowerboundPriorperturbations*
\begin{proof}[Proof of \Cref{proposition_W1_lower_bound_prior_perturbations}]
 By the hypothesis that $\mu_i\in\mathcal{P}_1$ for $i=1,2$, it follows from \Cref{corollary_attainment_of_supremum_in_Kantorovich_Duality} that there exists a 1-Lipschitz function $g_\opt$ such that 
 \begin{equation*}
    \Was_1(\mu_1,\mu_2)=\Abs{\int g_\opt \rd \mu_1-\int g_\opt \rd \mu_2}.
 \end{equation*}
Now
 \begin{align*}
 &\Abs{\int g_\opt \rd \mu_1-\int g_\opt \rd \mu_2}=\frac{Z_{\Phi,\mu_1}}{Z_{\Phi,\mu_1}}\Abs{\int g_\opt \rd \mu_1-\int g_\opt \rd \mu_2}
  \\
  =& Z_{\Phi,\mu_1} \Abs{\int g_\opt ~e^\Phi \frac{e^{-\Phi}}{Z_{\Phi,\mu_1}}\rd \mu_1-\int g_\opt ~e^{\Phi}\frac{e^{-\Phi}}{Z_{\Phi,\mu_1}}\rd \mu_2 }
  \\
  =& Z_{\Phi,\mu_1} \Abs{\int g_\opt ~e^\Phi \frac{e^{-\Phi}}{Z_{\Phi,\mu_1}}\rd \mu_1-\int g_\opt ~e^\Phi \frac{e^{-\Phi}}{Z_{\Phi,\mu_2}}\rd\mu_2+\int g_\opt ~e^\Phi \frac{e^{-\Phi}}{Z_{\Phi,\mu_2}}\rd\mu_2 -\int g_\opt~ e^{\Phi}\frac{e^{-\Phi}}{Z_{\Phi,\mu_1}}\rd \mu_2 }
  \\
  =& Z_{\Phi,\mu_1} \Abs{\int g_\opt ~e^\Phi \rd (\mu_1)_\Phi-\int g_\opt ~ e^\Phi \rd (\mu_2)_\Phi+\int g_\opt \rd \mu_2\left(\frac{1}{Z_{\Phi,\mu_2}}-\frac{1}{Z_{\Phi,\mu_1}}\right)}.
 \end{align*}
 This proves the first statement of the proposition.

For the second statement of the proposition, suppose that $Z_{\Phi,\mu_2}=Z_{\Phi,\mu_1}$ and that $g_\opt e^\Phi$ is constant, so that its Lipschitz constant vanishes. Then
\begin{align*}
 \Was_1(\mu_1,\mu_2)=&Z_{\Phi,\mu_1}\Abs{\int g_\opt ~e^\Phi \rd (\mu_1)_\Phi-\int g_\opt ~e^\Phi \rd (\mu_2)_\Phi}
 \\
 =&0= Z_{\Phi,\mu_1}\norm{g_\opt e^\Phi}_{\Lip(\supp{\mu_1+\mu_2},d)}\Was_1((\mu_1)_\Phi,(\mu_2)_\Phi).
\end{align*}
If $0<\norm{g_\opt e^\Phi}_{\Lip(\supp{\mu_1+\mu_2},d)}<\infty$, then
\begin{align*}
 &\Abs{\int g_\opt ~e^\Phi \rd (\mu_1)_\Phi-\int g_\opt ~e^\Phi \rd (\mu_2)_\Phi}
 \\
 =&\norm{g_\opt ~e^\Phi}_{\Lip(\supp{\mu_1+\mu_2},d)}\frac{\Abs{\int g_\opt ~e^\Phi \rd (\mu_1)_\Phi-\int g_\opt ~e^\Phi \rd (\mu_2)_\Phi}}{\norm{g_\opt ~e^\Phi}_{\Lip(\supp{\mu_1+\mu_2},d)}}
 \\
 \leq & \norm{g_\opt ~e^\Phi}_{\Lip(\supp{\mu_1+\mu_2},d)}\sup_{\norm{h}_{\Lip(\Theta,d)}\leq 1}\Abs{\int h \rd (\mu_1)_\Phi-\int h \rd (\mu_2)_\Phi}
 \\
 =& \norm{g_\opt ~e^\Phi}_{\Lip(\supp{\mu_1+\mu_2},d)}\Was_1((\mu_1)_\Phi,(\mu_2)_\Phi),
\end{align*}
where the last equation follows from the duality formula \eqref{eq_Kantorovich_Rubinstein_duality_formula}.
This yields
\begin{align*}
 \Was_1(\mu_1,\mu_2)=&Z_{\Phi,\mu_1}\Abs{\int g_\opt ~e^\Phi \rd (\mu_1)_\Phi-\int g_\opt ~e^\Phi \rd (\mu_2)_\Phi}
 \\
 \leq& Z_{\Phi,\mu_1}\norm{g_\opt ~e^\Phi}_{\Lip(\supp{\mu_1+\mu_2},d)}\Was_1((\mu_1)_\Phi,(\mu_2)_\Phi)
\end{align*}
and completes the proof.
\end{proof}

\section{Proofs of continuity in Wasserstein metrics}
\label{section_continuity_in_Wp_metrics_proofs}

We recall the following definition.
\begin{definition}[{\cite[Definition 6.8]{Villani2009}}]
 \label{def_weak_convergence_in_Wasserstein_space}
 Let $(X,\mathbf{d})$ be a Polish space and $p\in [1,\infty)$. Let $(\nu_k)_{k\in\N}$ be a sequence of probability measures in $\mathcal{P}_p(X)$ and let $\nu$ be another element of $\mathcal{P}_p(X)$. Then $(\nu_k)_k$ converges weakly in $\mathcal{P}_p(X)$ to $\nu$ if any one of the following equivalent properties is satisfied for some, and thus any, $x_0\in X$:
 \begin{enumerate}
  \item \label{item_weak_convergence_in_Wasserstein_space_1}
  $\nu_k\rightharpoonup\nu$ and $\int d(x_0,\cdot)^p~\rd\nu_k\to \int d(x_0,\cdot)^p~\rd\nu$.
  \item \label{item_weak_convergence_in_Wasserstein_space_2}
  For all continuous functions $\varphi$ with $\abs{\varphi}\leq C(1+d(x_0,\cdot)^p)$, $C\in\R$, one has $\int \varphi~\rd \nu_k\to\int\varphi~\rd\nu$.
 \end{enumerate}
\end{definition}
By \cite[Theorem 6.9]{Villani2009}, convergence in the $p$-Wasserstein distance is equivalent to weak convergence of measures in $\mathcal{P}_p(\Theta)$.

\LemmaWassersteinContinuityMisfitToPosteriorMap*
\begin{proof}[Proof of \Cref{lemma_Wasserstein_continuity_misfit_to_posterior_map}]
The hypothesis that $\Phi_n\to\Phi$ $\mu$-a.s. is equivalent to $\exp(-\Phi_n)\to \exp(-\Phi)$ $\mu$-a.s. 
The hypothesis that $\essinf_\mu\Phi_n=0$ implies that $\vert \exp(-\Phi_n)\vert \leq 1$ $\mu$-a.s.
By the dominated convergence theorem,
\begin{equation}
\label{eq_convergence_of_normalization_const}
    \lim_{n\to \infty}Z_{\Phi_n,\mu}=\lim_{n\to\infty} \int \exp(-\Phi_n)~\rd\mu = \int \exp(-\Phi)~\rd\mu=Z_{\Phi,\mu}.
\end{equation}
Let $f\in C_b(\Theta)$ be arbitrary. Since $\essinf_\mu\Phi_n=0$ for every $n\in\N$, $\abs{f\exp(-\Phi_n)}\leq \abs{f}$ $\mu$-a.s. Since $\Phi_n\to\Phi$ $\mu$-a.s., $f\exp(-\Phi_n)\to f\exp(-\Phi)$ $\mu$-a.s.
By the dominated convergence theorem, $\int f \exp(-\Phi_n)~\rd\mu \to\int f\exp(-\Phi)~\rd\mu$.
Thus, by \eqref{eq_convergence_of_normalization_const}, 
\begin{align*}
\lim_{n\to \infty}\int f~\rd\mu_{\Phi_n} &= \lim_{n\to \infty}\frac{1}{Z_{\Phi_n,\mu}} \int f\exp(-\Phi_n)~\rd\mu = \frac{1}{Z_{\Phi,\mu}} \int f\exp(-\Phi)~\rd\mu =  \int f(x) \mu_{\Phi}(\rd x). 
\end{align*}
Since $f$ was arbitrary, we have proven that $\mu_{\Phi_n}\rightharpoonup\mu_\Phi$.

Fix an arbitrary $x_0\in\Theta$. Since $\essinf_\mu\Phi_n=0$ for every $n\in\N$, $\abs{d(\cdot,x_0)^p\exp(-\Phi_n)}\leq d(\cdot,x_0)^p$ $\mu$-a.s. 
Now $\int d(\cdot,x_0)^p\exp(-\Phi_n)~\rd \mu\leq \int  d(\cdot,x_0)^p ~\rd\mu=\abs{\mu}_{\mathcal{P}_p}$.
Thus, since $\Phi_n\to\Phi$ $\mu$-a.s., we may apply the dominated convergence theorem to conclude that
\begin{align*}
 \lim_{n\to\infty} \int d(\cdot,x_0)^p \exp(-\Phi_n)~\rd\mu =\int d(\cdot,x_0)^p \exp(-\Phi)~\rd\mu.
\end{align*}
By \eqref{eq_convergence_of_normalization_const},
\begin{align*}
\lim_{n\to\infty}\int d(\cdot,x_0)^p~\rd\mu_{\Phi_n} &= \lim_{n\to\infty} \frac{1}{Z_{\Phi_n,\mu}}\int d(\cdot,x_0)^p\exp(-\Phi_n)~\rd\mu \\
&= \frac{1}{Z_{\Phi,\mu}}\int d(\cdot,x_0)^p \exp(-\Phi)~\rd\mu = \int d(\cdot,x_0)^p~\rd\mu_{\Phi}.
\end{align*}
Thus, by \Cref{def_weak_convergence_in_Wasserstein_space}\ref{item_weak_convergence_in_Wasserstein_space_1}, $(\mu_{\Phi_n})_n$ converges weakly in $\mathcal{P}_p(\Theta)$ to $\mu_\Phi$.
By \cite[Theorem 6.9]{Villani2009}, this weak convergence is equivalent to $\lim_{n\to \infty} \Was_p(\mu_{\Phi_n},\mu_{\Phi})=0$.
\end{proof}

\LemmaWassersteinContinuityPosteriorToMisfitMap*
\begin{proof}[Proof of \Cref{lemma_Wasserstein_continuity_posterior_to_misfit_map}]
Under the hypotheses on $\Phi$ and $(\Phi_n)_n$, it follows from \eqref{eq_likelihood_function} that $\ell_{\Phi,\mu}$ and $(\ell_{\Phi_n,\mu})_n$ belong to $C_b(\Theta,\R_{\geq 0})$. 
In particular, $\ell_{\Phi_n,\mu}-\ell_{\Phi,\mu}\in C_b(\Theta,\R)$ for every $n\in\N$.

 If $\Was_p(\mu_{\Phi_n},\mu_\Phi)\to 0$, then by \Cref{def_weak_convergence_in_Wasserstein_space}\ref{item_weak_convergence_in_Wasserstein_space_1}, $\mu_{\Phi_n}\rightharpoonup\mu_\Phi$.
 By definition of weak convergence in $\mathcal{P}(\Theta)$, it follows that for every $f\in C_b(\Theta,\R)$, $\int f~\rd \mu_{\Phi_n}\to\int f~\rd\mu_\Phi$.
Let $(a_k)_{k\in\N}\subset\R$ be an arbitrary sequence converging to zero. 
 Since $\ell_{\Phi_m,\mu}-\ell_{\Phi,\mu}\in C_b(\Theta,\R)$ for every $m\in\N$, it follows that 
 \begin{equation*}
  \int \left(\ell_{\Phi_m,\mu}-\ell_{\Phi,\mu}\right)\left(\ell_{\Phi_n,\mu}-\ell_{\Phi,\mu}\right)\rd \mu\xrightarrow[n\to\infty]{}0,\quad \forall m\in\N.
 \end{equation*}
Thus, for every $k\in\N$, there exists some $n_k\in\N$ such that
\begin{equation*}
\Norm{\ell_{\Phi_{n_k},\mu}-\ell_{\Phi,\mu}}_{L^2_\mu}^{2} = \int \Abs{\ell_{\Phi_{n_k},\mu}-\ell_{\Phi,\mu}}^2\rd \mu\leq a_k^2,
\end{equation*}
i.e. there exists a subsequence $(\ell_{\Phi_{n_k},\mu})_{k\in\N}$ converging in $L^2_\mu$ to $\ell_{\Phi,\mu}$ in $L^2_\mu$ as $k\to\infty$.

If in addition $(\Phi_n)_n$ and $\Phi$ are such that $Z_{\Phi_n,\mu}=Z_{\Phi,\mu}$ for every $n\in\N$, then by the definition \eqref{eq_likelihood_function} and the local Lipschitz continuity \eqref{eq_Lipschitz_continuity_exp_function} of the exponential function,
\begin{equation*}
\frac{\exp(-\norm{\Phi}_\infty)\wedge \min_k\exp(-\norm{\Phi_{n_k}}_\infty)}{Z_{\Phi,\mu}^2}\Norm{\Phi_{n_k}-\Phi}_{L^2_\mu}^2\leq \Norm{\ell_{\Phi_{n_k},\mu}-\ell_{\Phi,\mu}}_{L^2_\mu}^{2} 
\end{equation*}
\end{proof}

\LemmaWassersteinContinuityPosteriorToPriorMap*
\begin{proof}[Proof of \Cref{lemma_Wasserstein_continuity_posterior_to_prior_map}]
By the hypothesis that $\Was_p((\mu_n)_\Phi,\mu_{\Phi})\to 0$, it follows from \cite[Theorem 6.9]{Villani2009} that $((\mu_n)_\Phi)_n$ and $\mu_\Phi$ satisfy the conditions in \Cref{def_weak_convergence_in_Wasserstein_space}\ref{item_weak_convergence_in_Wasserstein_space_1}.
Since $\Phi\in C_b(\Theta,\R_{\geq 0})$, it follows that $\exp(\Phi)\in C_b(\Theta,\R)$. By the definition of weak convergence of probability measures and \eqref{eq_posterior_function},
\begin{equation}\label{eq_posterior_convergence_implies_convergence_of_normalisation_constants}
\lim_{n\to\infty}\frac{1}{Z_{\Phi,\mu_n}}=\lim_{n\to \infty} \int \exp(\Phi)~\rd(\mu_n)_\Phi = \int \exp(\Phi)~\rd\mu_{\Phi}=\frac{1}{Z_{\Phi,\mu}},
\end{equation}  
which is equivalent to $Z_{\Phi,\mu_n} \to Z_{\Phi,\mu}$.

Let $f\in C_b(\Theta,\R)$ be arbitrary. Then, using \eqref{eq_posterior_convergence_implies_convergence_of_normalisation_constants} and \eqref{eq_posterior_function},
\begin{equation*}
\lim_{n\to \infty} \int f(x)\mu_n(\rd x) = \lim_{n\to \infty} Z_{\Phi,\mu_n} \int  f \exp(\Phi)~\rd(\mu_n)_\Phi = Z_{\Phi,\mu} \int  f\exp(\Phi)~\rd\mu_{\Phi}= \int  f(x)\mu(\rd x).
\end{equation*}
Thus, $\mu_n\rightharpoonup\mu$.

Fix an arbitrary $x_0\in\Theta$. 
Since $\Phi\in C_b(\Theta,\R_{\geq 0})$, it follows that $x\mapsto\varphi(x)\coloneqq d(x_0,x)^p\exp(\Phi(x))$ satisfies $\abs{\varphi}\leq C(1+d(x_0,\cdot)^p)$ for $C\coloneqq \norm{\exp(\Phi)}_\infty$.
By the hypothesis that $\Was_p((\mu_n)_\Phi,\mu_{\Phi})\to 0$ and \cite[Theorem 6.9]{Villani2009}, $((\mu_n)_\Phi)_n$ converges weakly in $\mathcal{P}_p(\Theta)$ to $\mu_\Phi$.
Thus, by \Cref{def_weak_convergence_in_Wasserstein_space}\ref{item_weak_convergence_in_Wasserstein_space_2} and \eqref{eq_posterior_function},
\begin{equation*}
 \lim_{n\to\infty}\int d(x_0,\cdot)^p\exp(\Phi)~\rd(\mu_n)_\Phi = \int d(x_0,\cdot)^p\exp(\Phi)~\rd\mu_{\Phi}=\frac{1}{Z_{\Phi,\mu}}\int d(x_0,\cdot)^p~\rd\mu.
\end{equation*}
For every $n\in\N$, $\int d(x_0,\cdot)^p\exp(\Phi)~\rd(\mu_n)_\Phi=\frac{1}{Z_{\Phi,\mu_n}}\int d(x_0,\cdot)^p~\rd\mu_n$. 
These observations yield
\begin{equation*}
\lim_{n\to \infty}\frac{1}{Z_{\Phi,\mu_n}}\int d(x_0,\cdot)^p ~\rd\mu_n = \frac{1}{Z_{\Phi,\mu}}\int d(x_0,\cdot)^p~\rd\mu.
\end{equation*}
Given \eqref{eq_posterior_convergence_implies_convergence_of_normalisation_constants}, we conclude
\begin{equation*}
\lim_{n\to\infty}\int d(x_0,\cdot)^p~\rd\mu_n = \int d(x_0,\cdot)^p~\rd\mu.
\end{equation*}
Since $\mu_n\rightharpoonup\mu$, it follows by \Cref{def_weak_convergence_in_Wasserstein_space}\ref{item_weak_convergence_in_Wasserstein_space_1} that $(\mu_n)_n$ converges weakly in $\mathcal{P}_p(\Theta)$ to $\mu$.
By \cite[Theorem 6.9]{Villani2009}, $\Was_p(\mu_n,\mu)\to 0$.
\end{proof}

\end{document}